\documentclass[11pt, reqno]{amsart}
\usepackage{amsmath}
\usepackage{amsfonts,bbm}
\usepackage{amssymb}
\usepackage{amsthm}
\usepackage{graphicx}
\usepackage{bm}
\usepackage{subfigure}
\usepackage[usenames]{color}
\usepackage{pstricks,pst-node}

\usepackage{hyperref}
\newtheorem{theorem}{Theorem}
\newtheorem{proposition}[theorem]{Proposition}
\newtheorem{definition}[theorem]{Definition}

\newtheorem{condition}[theorem]{Condition}

\newtheorem{lemma}[theorem]{Lemma}

\newtheorem{corollary}[theorem]{Corollary}
\theoremstyle{remark}
\newtheorem*{remark}{Remark}

\newcommand{\spam}{\mathop{\mathrm{span}}}

\newcommand{\supp}[1]{{\operatorname{supp}}({#1})}
\newcommand{\ints}{\mathbb{Z}}
\newcommand{\nats}{\mathbb{N}}
\newcommand{\reals}{\mathbb{R}}
\newcommand{\comps}{\mathbb{C}}

\newcommand{\dif}{\mathrm{d}}

\newcommand{\diam}{\operatorname{diam}}
\newcommand{\dist}{\operatorname{dist}}

\renewcommand{\L}{{L}}

\newcommand{\nchoosek}[2]{\displaystyle \binom{#1}{#2}}

\newcommand{\bfg}{\mathbf{g}}
\newcommand{\bfh}{\mathbf{h}}

\newcommand{\C}{\mathsf{C}}
\newcommand{\B}{\mathsf{B}}
\newcommand{\sC}{\mathsf{C'}}
\newcommand{\sB}{\mathsf{B'}}
\newcommand{\CL}{\mathsf{S}}
\newcommand{\BL}{\mathsf{R}}
\newcommand{\sCL}{\mathsf{S'}}
\newcommand{\sBL}{\mathsf{R'}}
\newcommand{\M}{\mathsf{M}}

\newcommand{\Deltad}{\Delta_{{d-1}}}
\renewcommand{\b}{\mathsf{b}}
\renewcommand{\c}{\mathsf{c}}

\newcommand{\V}{\mathbf{\mathcal{V}}}

\newcommand{\T}{\mathcal{T}}

\newcommand{\Riesz}{\mathcal{R}}
\newcommand{\riesz}{\mathrm{r}}
\newcommand{\Space}{\mathcal{L}}
\newcommand{\Err}{\mathcal{E}}
\def\dotp{\operatorname{\bm{\cdot}}}

\newcommand{\oddq}[1]{\operatorname{odd}(#1)}
\newcommand{\Oddq}[1]{\lceil #1/2 - \lceil (#1-1)/2\rceil \rceil}
\newcommand{\diag}[1]{\operatorname{diag}\left(#1\right)}
\newcommand{\id}{\mathbbm{1}}
\newcommand{\Sym}{\mathfrak{S}}
\newcommand{\demph}[1]{\emph{#1}}
\newcommand{\ds}{\displaystyle}

\def\qand{\quad\hbox{and}\quad}

\newcommand{\sfp}{\mathsf{p}}
\newcommand{\sfq}{\mathsf{q}}

\newpsobject{showgrid}{psgrid}{%
  subgriddiv=1,griddots=10,gridlabels=6pt}
\newif\ifhrule\hrulefalse

\psset{griddots=6,gridlabels=0,gridcolor=darkgray,subgriddiv=1}

\newcommand{\ChrPath}[1][myred]{\psline[linewidth=0.05,linecolor=#1]{->}}

	\definecolor{myred}{rgb}{1,0,0}
	\definecolor{myblue}{rgb}{0,0,1}
	\definecolor{mygreen}{rgb}{0,1,0}
	\definecolor{myorange}{rgb}{.9,.4,0}
	\definecolor{mypink}{rgb}{1,.5,.5}

\numberwithin{equation}{section}
\numberwithin{theorem}{section}

\title{The Polyharmonic Dirichlet Problem and Path Counting}
\author{Thomas Hangelbroek}
\address{Department of Mathematics, University of Hawaii, Honolulu, HI 96822, USA. } 
\email{hangelbr@math.hawaii.edu}
\thanks{Thomas Hangelbroek was supported by grant DMS-1232409 from the National
    Science Foundation.}
\author{Aaron Lauve}
 \address{Department of Mathematics, Loyola University Chicago, Chicago, IL 60660, USA.}
\email{lauve@math.luc.edu}
\thanks{Aaron Lauve was supported by grant 98230-12-1-0286 from the National Security Agency.}
\begin{document}

\begin{abstract}
The purpose of this article is to provide a solution to the
$m$-fold Laplace equation in the half space $\reals^d_+$ under certain Dirichlet
 conditions. The solutions we present are a series of $m$ boundary layer potentials.
We give explicit formulas for these layer potentials as  linear combinations of 
powers of the Laplacian 
applied to the Dirichlet data, with coefficients determined by certain path counting problems.
\end{abstract}

\maketitle

%
%
\section{Introduction}\label{Introduction}
In this article we consider a class of polyharmonic Dirichlet problems in the half-space $\reals^d_+:=\reals^{d-1}\times [0,\infty)$. Specifically, for $m>d/2$, we consider
%
\begin{equation}\label{Dirichlet}
\begin{cases} 
  \Delta^m u(x) = 0,& \text{for $x\in \reals^{d}_+$};\\
  \lambda_k u = h_k,& \text{for }k=0,\dots, m-1.\\
\end{cases}
\end{equation}
%
where the $m$ boundary conditions are given by 
%
\begin{equation}\label{DirichletOperators}
\lambda_k u(x):= 
\begin{cases} 
  \Delta^{\frac{k}{2}}u (x_1,\dots,x_{d-1},0),& \text{for $k$ even},\\
  - \frac{\partial}{\partial x_d} \Delta^{\frac{k-1}{2}}u (x_1,\dots,x_{d-1},0),& \text{for $k$ odd}.
\end{cases}
\end{equation}
%
Our interest is to describe 
solutions to \eqref{Dirichlet} with $m$ simple boundary layer potentials. Specifically, we seek solutions of 
the form 
%
%
\begin{equation}\label{BLP_Solution}
u(x) = \sum_{j=0}^{m-1} \int_{\partial \reals_+^{d}}g_j(y) 
k_j(x,y)
 \dif \sigma(y)
\end{equation}
where the kernel
$k_j(x,y) :\reals^d\times \partial \reals_+^{d}\to \reals $ is 
defined as $k_j(x,y) :=\lambda_{j,y} \phi(x-y)$
and
$\phi$ is a fundamental solution to $\Delta^m$ in $\reals^d$.
In this case, we wish to find the \emph{auxiliary} functions 
$\bfg = (g_j)_{j=0}^{m-1}$, defined on $\reals^{d-1} \sim \partial \reals_+^{d}$. 
The motivation for this comes from approximation theory, and is discussed below in Section \ref{Motivation}.

The problem \eqref{Dirichlet} is a longstanding one in the theory of elliptic PDE.
In 1905, a closed formula for the Green's function for the complementary problem 
(with $\Delta^mu=f$ and homogeneous boundary data $h_k=0$) 
was presented by Boggio  in \cite[Eqn. (48)]{Boggio}.
This Green's function can be extended to the half space without much difficulty (see \cite[Eqn. (3.1)]{ReiWeth} for instance), and from there a Poisson kernel can be obtained by taking suitable high order derivatives of the Green's function 
(of course, general constant coefficient elliptic boundary value problems of all kinds have Poisson kernels via 
\cite[Eqn. (2.1)]{ADN}).
This problem and its variants  (in  bounded regions, with alternative boundary values, etc.) continue to receive much attention 
(we mention \cite{AmRaud, DallSweers,
MM,ReiWeth} 
as a very abbreviated list of recent works).
For an introduction to this vast topic and many other related ones we direct the reader to
the recent monograph of Gazzola, Grunau and Sweers \cite{GGS}.

Requiring the solution of \eqref{Dirichlet} to take the form \eqref{BLP_Solution} converts this problem into an integral equation.
We seek to match the Dirichlet data $\mathbf{h} = (h_k)_{k=0}^{m-1}$ 
with boundary values of  the \emph{boundary layer potentials}
$\sum_{j=0}^{m-1} \int_{\reals^{d-1}}g_j(y) \lambda_{j,y} \phi(x-y) \dif \sigma(y)$.
In short, we wish to solve $\V \bfg=\mathbf{h} $ where  $\V$ is 
an $m\times m$ matrix  of integral operators on $\reals^{d-1}$.
We note that this set up has been considered in 
planar domains for $m=2$ by Chen and Zhou \cite[Chapter 8]{ChZh} and Costabel and Dauge \cite{CD}.

Our goal is to solve this system explicitly for modest assumptions on the Dirichlet data $\mathbf{h}$. 
The solution involves using Fourier analysis to recast this as a multiplier problem, and to identify the matrix \emph{symbol} of the multiplier equation.
Namely, we consider the Fourier transform $\widehat{\mathbf{h}}(\xi) = \mu(\xi) \widehat{\bfg}(\xi)$ where
$\mu = \sigma(\V):\reals^{d-1}\to \comps^{m\times m}$ is a Fourier multiplier.
This multiplier can be factored $\mu(\xi) = D_1(\xi) \M D_2(\xi)$ as  a product of simple diagonal multipliers $D_1$ and $D_2$
(having entries that are, essentially, integral powers of $|\xi|$), and a matrix
of constant entries $\M$. 

The invertibility of the system hinges on the matrix $\M$, which is shown to have a checkerboard pattern: zeros in the
odd entries (i.e., $\M_{ij} = 0$ when  $i+j$ is odd) and a two-block structure in the even entries. These blocks are Hankel matrices $\B$ and $\C$ of binomial coefficients and Catalan numbers, respectively. 
At this point our second objective emerges: the entries in either block 
count minimal lattice paths in certain planar graphs. 
In the Catalan case, these are Dyck paths;
for the binomial block we have different but related paths. 

Developing a connection to these lattice paths allows interesting observations
about the fundamental matrix $\M$. 
In particular, its minors count certain non-intersecting paths in slightly modified graphs---these are easily
shown to be positive, and thus the underlying matrix is totally non-negative ($\B$ and $\C$ are totally positive).
Moreover, the path counting perspective easily reveals that the determinant of $\M$ is a power of $2$ (hence $\M$ is invertible, hence $\widehat{\mathbf g}$ is uniquely determined from $\widehat{\mathbf h}$). 
The problem of determining the inverse of $\M$  (for all $m$) is a challenging combinatorial problem, but essential for obtaining the closed form solution $\bfg$.
A solution to this problem can be phrased in terms of path counting, but finding it by this method proves quite difficult. We determine explicit formulas for the entries of $\M^{-1}$ by other means. 

%
%
\subsection{Motivation}\label{Motivation}
%
The motivation for this problem comes from surface spline approximation and interpolation,
introduced in the 1970's by Duchon \cite{D1} and Meinguet \cite{Meinguet} and further developed  
by many others. 
This is an approach to multivariate approximation where elementary functions are constructed by taking linear 
combinations of scattered translates of the fundamental solution $\phi$ for $\Delta^m$ in $\reals^d$, 
perhaps adding a polynomial term of low degree,
$$ s_{\Xi} (x)= \sum_{\xi\in\Xi} A_{\xi} \phi(x - \xi)+p(x).$$
This is an appealing methodology because such elementary functions can be constructed and evaluated directly 
(without imposing additional structure like triangulations or grids).

One reason for the popularity of this method is  its effectiveness in treating scattered data: 
as an example, consider $(z_{\xi})_{\xi\in\Xi}$ sampled from a continuous function $f$ on the set 
$\Xi\subset \reals^d$ (i.e., $z_{\xi} = f(\xi)$ for some unknown target function $f$). One can attempt
to match $f$ by a function of the form $s_{\Xi}$ \footnote{Note that in this example the set of centers $\Xi$ and the set of sampling points coincide---this is not strictly necessary.}
for instance, by interpolation: 
solving the linear system $s_{\Xi} (\xi) = z_{\xi}$ for $\xi\in \Xi$. 
(By now the reader may have remarked that we require $m>d/2$---one reason for this is to ensure that such an interpolation problem is well-posed, since  $m>d/2$ implies that $\phi$ is continuous.)

When $d=2$ and $m=2$, this is known as {\em thin plate spline} interpolation, because of the connection
of the bi-Laplacian to the elasticity theory for idealized thin plates.
In this case, the interpolation problem can be viewed as fixing the position of a thin plate at some points 
$(\xi,z_{\xi})\in \Xi\times \reals$ in space---the value of the interpolant $s_{\Xi}(x)$ gives the height of this plate at the point $x\in\reals^2$.
Likewise, the Dirichlet problem \eqref{Dirichlet}  in this case can be viewed as a {\em clamped} plate problem: position and outer normal derivative of an idealized plate are fixed by the boundary values $h_k$ along $\reals\times\{0\}$; $\Delta^mu=0$ encodes the elasticity constraints; and the value of the solution $u(x)$ gives the displacement of this plate at $x\in\reals^2$ (see \cite[Chapter 8]{ChZh} or \cite[Chapter 1]{GGS} for a discussion of the thin plate model).

A major challenge for this approximation method deals with negative {\em boundary effects}, occurring when the points $\Xi$  lie within a region $\Omega$ --in other words,  on {\em one side} of  the boundary $\partial \Omega$.
In this case, the error between the approximant and the target function in a neighborhood of the boundary 
is larger than elsewhere. As the spacing of $\Xi$ becomes finer  in $\Omega$, the rate of 
decay of this error  can be quantified;
it is substantially slower than in the interior of $\Omega$.
A satisfactory method for understanding and overcoming the boundary effects for thin plate spline approximation ($m=2$, $d=2$) was treated  in \cite{Hdisk}, where a new approximation scheme was developed using
an integral representation for smooth functions that uses a minimal number of boundary layer potentials. 
This representation takes the form
%
 \begin{equation*}
 f(x) 
 = 
 \int_{\Omega} 
     \Delta^2 f(y) \phi(x-y) 
 \dif y 
 +  
 \int_{\partial \Omega} 
     \left( 
         N_3 f(y) \lambda_{0,y} \phi(x-y) 
         +
         N_2 f(y) \lambda_{1,y} \phi(x-y)  
     \right)
 \dif \sigma(y)
 +
 p(x),
 \end{equation*}
%
 with $\phi$ the fundamental solution for $\Delta^2$ in $\reals^2$, $p$ a polynomial of degree at most $1$, and
 $N_2$ and $N_3$ are pseudodifferential operators on the boundary.\footnote{This representation should be compared to ``Green's representation'' from potential theory. 
The crucial difference is the presence of extra boundary integrals involving higher order derivatives of the kernel $\phi$ (and therefore use kernels that are more singular).}
 
Developing higher order and higher dimensional versions of such representations are necessary to create more general (higher order, higher dimensional) approximation schemes. These representations, in turn, follow from considering the structure of solutions to the Dirichlet problem for arbitrary spatial dimensions $d\ge 2$ and orders $m>d/2$.
 
Finally, the study of the polyharmonic Dirichlet problem is closely related to a number of other applications. 
The problem of  {\em transfinite interpolation} (see, e.g., \cite{Bej}), modifies the basic interpolation setup to treat (non-discrete) data on curves and surfaces. 
Dirichlet problems for polyharmonic and other equations often play a key role in image {\em inpainting} in 
digital image analysis  \cite{BBBW}. 
The solution of polyharmonic Dirichlet problems on rectangular regions underly the signal and image processing approach considered  Saito and his colleagues, \cite{SaitoZhao,SaitoRemy}. 
Kounchev has developed {\em polyspline approximation}
(using piecewise polyharmonic functions with prescribed smoothness conditions), which has been used to treat a variety of problems in numerical analysis, approximation theory and wavelet theory. For an introduction, we cite
the manuscript \cite{Koun}, which has a substantial bibliography (although by now this subject has grown considerably).
 
%
%
\subsection{Solution}\label{Solution}
 Under suitable conditions (i.e., smoothness, decay and vanishing moment conditions to be elaborated in the following sections) 
 on Dirichlet data $\bfh$, the Dirichlet problem has solution $\bfg$ with
%
\begin{equation*}
 g_j 
 = 
 \sqrt{-\Delta_{d-1}}
 \sum_{\substack{k=0\\k+j\in 2\ints}}^{m-1} 
     (\M^{-1})_{j,k}(\Delta_{d-1})^{m - 1- (j+k)/2} h_k \,,
\end{equation*}
%
 where $\Delta_{d-1}= \sum_{\ell=1}^{d-1}\frac{\partial^2}{\partial x_\ell^2}$ is the Laplacian on $\reals^{d-1}$. We note
 that the auxiliary boundary functions with odd indices depend only on the odd boundary values $h_{k}$ and likewise, the even $g_j$ 
 depend only on boundary data $h_k$ with even indices $k$.
Moreover, the coefficients $(\M^{-1})_{j,k}$ have simple, explicit formulations which are given in Section 
\ref{inverting_M}.

%
%
\subsection{Outline} \label{Outline}
We have organized this article in the following way.
In Section \ref{The Dirichlet Problem}, we set some basic notation and discuss some background from analysis.
%
In Section \ref{Boundary Layer Potential Operators} we analyze the boundary layer potentials,
 their regularity and their frequency representation. In this section we make conditions on the auxiliary functions $\mathbf{g}$ that permit the boundary layer potential solution \eqref{BLP_Solution} to have a manageable Fourier transform.
 %
 Section \ref{Boundary values of the boundary layer potentials} is concerned with understanding the boundary values of the boundary layer potential solution---the main result in this section is the identification of the matrix symbol of the operator $\mathcal{V}$ mapping auxiliary boundary functions $\bfg$ to boundary data $\mathbf{h}$.
 %
 
 Section \ref{sec: path counting and invertibility} 
relates coefficients from this symbol to the problem of counting paths in certain graphs, and shows
that various properties of the symbol can be related to counting non-intersecting paths. 
%
Section \ref{Invertibility of the boundary value operator} presents our main results, which includes the closed form representation of the map from Dirichlet data $\mathbf{h}$ to
boundary functions $\bfg$.
%
Section \ref{inverting_M} is concerned with inverting the binomial and Catalan matrices that form the blocks $\B$ and $\C$ of $\M$.  This involves computing their $LDL^T$ (Cholesky) decompositions and providing a closed form formula for the triangular matrix $L^{-1}$.
%
Section \ref{Decay of Riesz transforms} treats the Riesz transform. We demonstrate that the Riesz transform of a  function with many vanishing moments decays rapidly. Although
this is well known for compactly supported functions, we must extend this result to treat functions with algebraic decay.
%
%
The final section treats two extensions of the method we've presented, treating problems with different boundary conditions, and different elliptic, constant coefficient differential operators.

%
%
\section{Background and basic assumptions}\label{The Dirichlet Problem}
\subsubsection*{Points in $\reals^d$}We denote the inner product of two points $x,y\in \reals^d$ by $\langle x,y\rangle$ and the Euclidean norm of $x\in \reals^d$ by $|x|$. The ball of radius $\rho$ centered at $c\in\reals^d$ is denoted $B(c,\rho)$. Often we identify  the hyperplane $\partial \reals_+^{d}$ with $\reals^{d-1}$, and for 
an integral  of $f:\reals^d\to \comps$ taken with respect to surface measure on $\partial \reals_+^{d}$ we 
simply write $\int_{\reals^{d-1}} f(y) \dif \sigma(y)$.
We frequently make use of the decomposition 
$x =(x',x_d)\in \reals^d$ with $x_d\in\reals$ and $x'$ the projection onto $\reals^{d-1}$.

\subsubsection*{Multi-integers} For a multi-integer $\alpha = (\alpha_1,\dots,\alpha_d) \in \nats^d$,  
we denote the size $|\alpha| = \sum_{j=1}^d \alpha_j$ (this should not be confused with the Euclidean norm 
$|x|$ of $x\in \reals^d$---the context will make clear what is the proper meaning of $|\cdot|$). These are used to denote monomials, $x^{\alpha} = \prod_{j=1}^d x_j^{\alpha_j}$ and partial derivatives $D^{\alpha} = \prod_{j=1}^d \frac{\partial^{\alpha_j}}{\partial x_j^{\alpha_j}}$.

\subsubsection*{Spaces of functions}
The space $\spam_{|\alpha|\le L} x^{\alpha}$ consisting of polynomials of degree $L$ or less  is denoted 
by $\Pi_L = \Pi_L(\reals^d)$.
For a measurable subset $\Omega \subset \reals^d$ and $1\le p\le \infty$,  $L_p(\Omega)$ indicates the Banach space of (equivalence classes of) functions for which $\int_{\Omega} |f(x)|^p \dif x<\infty$ (or $\mathrm{ess} \sup |f(x)|<\infty$ in case $p=\infty$.) For an open set $\Omega$, $1\le p\le \infty$ and $k\in \nats$, the Sobolev space 
$W_p^{k}(\Omega) = \{f\in L_p(\Omega)\mid \text{for }|\alpha|\le k, \ D^{\alpha} f \in L_p(\Omega)\}$ 
is the subspace of $L_p(\Omega)$ for which all distributional derivatives of order 
$k$ or less are also in $L_p(\Omega)$. When $p=\infty$, we prefer the space $ C^{k}(\reals^d)$ of functions having continuous partial derivatives of order $k$.

Given $\epsilon>0$, the function $f$ is H{\"o}lder continuous  with exponent $\epsilon$ at a point $x$ if 
the pointwise H{\"o}lder coefficient 
$[f]_{\epsilon,x} := \sup_{|y-x|\le 1} \frac{|f(y) -f(x)|}{ |y-x|^{\epsilon}}$
is finite. 
For non-integer $s>0$,  the H{\"o}lder space $C^s(\reals^d)$ consists of the functions 
$f\in C^k(\reals^d)$  with $k= \lfloor s\rfloor$ 
so that for each multi-index $|\alpha|=k$, 
$D^{\alpha}f$ is H{\"o}lder continuous  with exponent $\epsilon = s-k$. 
The H{\"o}lder seminorm of $f$ is denoted by
$$
|f|_{C^s(\Omega)} : = \begin{cases} 
\max_{|\alpha|=k} \sup_{x\in \Omega}\left|D^{\alpha} f(x)\right|_{\infty}, &\text{for } s=k\in \ints\\
\max_{|\alpha|=k} \sup_{x\in\Omega} \left[D^{\alpha} f\right]_{\epsilon,x}, &\text{for } s\notin \ints.
 \end{cases}
 $$

\subsubsection*{Spaces simultaneously capturing smoothness and decay} On top of these classical functions spaces which capture smoothness, we need also to encode decay properties of the function.
We can 
describe basic assumptions on the Dirichlet data 
with the help of a  family of Banach spaces:
for  $J>0$ and $s,\delta \ge 0$, define $\Space_{J,\delta}^s = \Space_{J,\delta}^s(\reals^d)$ 
to be the class of functions $f\in C^s(\reals^d)$  for which
\begin{equation}\label{appendix_decay}
\|f\|_{\Space_{J,\delta}^s}:=
\sup_{0\le \sigma\le s}
\sup_{\rho\ge 0}
\left( 
  (1+\rho)^{J+\sigma}  
  \bigl( \log \bigl(e +\rho \bigr)\bigr)^{\delta}
  \left|f \right|_{C^{\sigma}\bigl(\reals^{d}\setminus B(0,\rho)\bigr)}
\right)
\end{equation}
is finite.

It is not hard to see that  if $s_1<s_2$ then 
$\Space_{J,\delta}^{s_2}\subset \Space_{J,\delta}^{s_1}$.
Likewise if $J_1<J_2$ then  
$\Space_{J_2,\delta_a}^{s}\subset \Space_{J_1,\delta_b}^{s}$ 
(this holds regardless of $\delta_a$ and $\delta_b$).
When $s=0$, we simply write $\Space_{J,\delta}=\Space_{J,\delta}^s$.
For $f\in\Space_{J,\delta}$, 
the estimate
\begin{equation*}
|f(x)|\le \|f\|_{\Space_{J,p}}(1+|x|)^{-J} \bigl( \log \bigl(e + |x|\bigr)\bigr)^{-\delta}
\end{equation*} 
holds. 
Moreover, when  $J\ge d+ L$ (assuming $\delta>1$ in the case of equality $J=d+L$), 
the integrals 
$\langle f,p\rangle = \int_{\reals^d} f(x) p(x) \dif x $ for $p\in\Pi_L$
are well-defined. 

%
%
\subsubsection*{Fourier transform and multipliers}\label{Fourier transform}
We  define the Fourier transform for $L_1(\reals^d)$ functions as
 \begin{equation*}
 \widehat{f} (\xi) 
 := 
 \int_{\reals^d} f(x) e^{-i\langle x,\xi\rangle} \dif x,
 \end{equation*} 
 and extend to tempered distributions in the usual way. In a similar way, we denote the inverse 
 Fourier transform of an $L_1$ function $g$  by 
 $g^{\vee}(x) = (2\pi)^{-d}  \int_{\reals^d} g(\xi) e^{i\langle x,\xi\rangle}\, \dif \xi$.
 For functions of sufficient smoothness and rapid decay, the inversion theorem gives that
 $f(x) = (2\pi)^{-d} \int_{\reals^d}  \widehat{f} (\xi) e^{i\langle x,\xi\rangle} \, \dif \xi.$
 
 We make use of a class of linear operators called multiplier operators. The study of such operators, which include convolution and constant coefficient differential operators, is an important and enduring aspect of harmonic analysis. For our purposes, these are linear operators defined  as
 $$ f \mapsto (2\pi)^{-d} \int_{\reals^d}  \mu(\xi) \widehat{f} (\xi) e^{i\langle x,\xi\rangle} \, \dif \xi.$$
 The function $\mu:\reals^d\to \comps$ is referred to as the \demph{symbol}. 
For an operator $v$, the  symbol is denoted by $\sigma(v) = \mu$.

Of particular interest are operators with a matrix symbol. This is  the obvious generalization to vector valued functions $\bfg=(g_1,\dots,g_{M})\in L_1(\reals^d,\comps^M)$ of the previous concept. In this case, we consider an operator $\mathcal{V}$ with matrix symbol 
$\sigma(\mathcal{V}) = (\mu_{k,j})
: \reals^d\to \comps^{N\times M}$.
In other words, for $j=1\dots M$ and $k=1\dots N$, each entry of the symbol 
is $\mu_{k,j}:\reals^d \to  \comps$, and the operator is
$\bfg \mapsto \mathcal{V} \bfg = 
\left((2\pi)^{-d} \sum_{j=1}^{M}
  \int_{\reals^d}\mu_{k,j}(\xi)\widehat{g_j}(\xi)e^{i\langle x,\xi\rangle} 
\dif \xi
\right)_{k=1,\dots,N}$.

\section{Boundary Layer Potential Operators}\label{Boundary Layer Potential Operators}
We seek solutions to \eqref{Dirichlet} comprising $m$ boundary layers using the potential function
$\phi$, where
%
\begin{equation*}
\phi(x)
=\
\phi_{m,d}(x)
:=
C_{m,d}
\begin{cases}
  |x|^{2m-d}\log|x|, &\text{for $d$ even and }2m\ge d,\\
  |x|^{2m-d}, &\text{otherwise}
\end{cases}
\end{equation*}
%
is a fundamental solution of $\Delta^m$ in $\reals^d$.
This result, and the precise value of the constant $C_{m,d}$ for which $\Delta^m \phi_{m,d} = \delta$ 
can be found in \cite[(2.11)]{Poly}. 
We note that $\Delta^j \phi_m = \phi_{m-j}$. 

Our goal is to write the solution of \eqref{Dirichlet} in the form
%
%
\begin{equation}\label{3:BhRep}
u(x) 
= 
T{\bfg}(x) 
:=
\sum_{j=0}^{m-1} 
    \int_{\partial \Omega} 
           g_j(y) \lambda_{j,y}\phi(x-y)\, 
    \dif\sigma(y).
\end{equation}
We begin our analysis of the solution $T{\bfg}$ by considering its
constituent functions, the boundary layer potentials 
$\int_{\reals^{d-1}} g_j(y) \lambda_{j,y} \phi(\cdot - y) \dif \sigma(y)$.
We express such functions in two key ways.
First, we consider them as integral operators applied to functions supported on $\reals^{d-1}$:
$$V_j g_j (x):=  \int_{\reals^{d-1}} g_j(y)k_j(x,y)\, \dif\sigma(y)
\quad \text{with}\quad k_j(x,y) :=\lambda_{j,y} \phi(x-y).$$
Second, we consider them as the convolution of $\phi$ with distributions supported
on $\partial \reals_+^{d}$, which leads to their characterization via the Fourier transform.

%
%
\subsection{Boundary layer potential operators}\label{Boundary layer potential operators}
One observes, via a direct computation, that
as $|x-y|\to 0$, 
 $k_j(x,y) = \mathcal{O}(|x-y|^{2m-d-j}\log{|x-y|})$.
Indeed, we can compute $D^{\beta} \phi(x)$ explicitly, as indicated by the following
 formula 
 (this has already been observed in
 \cite[Eqn. (51)]{Boggio} and perhaps earlier).
For any multi-index $\beta$, a direct computation shows that
%
%
\begin{equation}\label{fs_derivs}
D^{\beta} \phi(x)  
= 
\sfp_{2m-d-|\beta|}(x)
\bigl(\log|x|^2\bigr)
+
\sfq_{2m-d-|\beta|}(x) ,
\end{equation}
%
where $\sfp_{\gamma}$ and $\sfq_{\gamma}$ are  homogeneous functions of degree $\gamma$. 
From this, it follows that 
$
|D^{\beta} \phi(x)| 
\lesssim  
|x|^{2m-d -|\beta|}
(1+ \log|x|).
$ 
In fact, if $d$ is odd or if $|\beta|>2m-d$, the homogeneous
function $\sfp_{2m-d-|\beta|}$ vanishes.
This leads to the following key observation.
%
%
\begin{lemma}\label{Hom_decay} 
For a multi-index $\alpha$, with $|\alpha|>2m-d$, the derivative
$D^{\alpha} \phi (x)$ is homogeneous of degree $2m-d-|\alpha|$. In particular, there is
a constant $C$ (depending on $m,d$ and $\alpha$) so that for $|x|>0$,
\begin{equation*}
|D^{\alpha} \phi (x) | 
\le 
C |x|^{2m-d-|\alpha|} .
\end{equation*}
\end{lemma}
For a compactly supported distribution $T$, the convolution $\phi*T(x)$ is well-defined for $x\notin \mathrm{supp}(T)$. In this case, we have the following corollary.
\begin{corollary}\label{Convolution_Decay} If $f$ is a compactly supported function and satisfies moment conditions $f\perp \Pi_L(\reals^{d})$ with $L>2m-d$, then as $|x| \to \infty$,
$$|\phi*f(x)| \le \mathcal{O}(|x|^{2m-d-L}).$$
\end{corollary}

At this point we state the first of two restrictions on the domain of $V_j$ which will be used later. 
%
%
\begin{condition}[Decay]\label{firstDecay}
For $j=0,\dots, m-1$ we assume $g_j$ is 
continuous and  there are constants $C>0$ and $\delta>2$ so that   
\begin{equation*}
|g_j (y)| \le C (1+|y|)^{1+j-2m} \bigl(\log(e+|y|)\bigr)^{-\delta}.
\end{equation*}
\end{condition}
%
In other words, we assume  $g_j\in \Space_{2m-j-1,\delta}$ for $j=0,\dots,m-1$.
This condition, in conjunction with formula \eqref{fs_derivs}, ensures that 
$
\int_{\reals^{d-1}} 
    g_j(y) \lambda_{j,y}\phi(x-y) 
\dif\sigma(y)
$ 
is well-defined. 

%
%
\begin{lemma}\label{boundary_values}
For a family of functions $\mathbf{g}=(g_j)_{j=0}^{m-1}$ 
satisfying Condition \ref{firstDecay},
$T\mathbf{g}\in C^{m-1}(\reals^{d})$ and satisfies 
$\Delta^m T\mathbf{g}=0$ 
in 
$\reals^{d}_+$ as well as $\reals^d_-$.
\end{lemma}
%
\begin{proof}
The smoothness on $\phi$ away from the origin
and the decay condition on $g_j\times k_j(x,\cdot) = g_j\times \lambda_j \phi(x-\cdot)$
ensures that $V_j g_j$ is $C^{\infty}$  in each open half space.  
Likewise, the fact that 
$\Delta^m \phi(x) = 0$ for $x\ne 0$
and that 
$\Delta^m \lambda_{j,y} \phi(x-y) = \lambda_{j,y} \Delta^m \phi(x-y)$
ensures that each $V_j g$ is $m$-fold polyharmonic in both half spaces.

The bounds obtained from \eqref{fs_derivs} show that $V_j g_j$ extends to 
$C^{2m-j-2}\bigl(\reals^{d})$,
 since 
%
\begin{equation*}
D^{\beta} V_j g_j (x)
=
\int_{\reals^{d-1}} 
    g_j(y) 
    D^{\beta}_x \lambda_{j,y} \phi (x-y)
\dif \sigma(y)
= 
\int_{\reals^{d-1}} 
    g_j(y) 
    \lambda_{j,y} D^{\beta}_x \phi (x-y) 
\dif \sigma(y).
\end{equation*}
%
The first step involves dominated convergence, justified by the fact that
 $|g_j (y) \lambda_{j,y} D^{\beta}_x \phi (x-y)|$ 
 is integrable:
 the product decays rapidly, thanks to Condition \ref{firstDecay} while 
 $\lambda_{j,y} D^{\beta}_x \phi (x-y)$ 
 is locally integrable, since 
 $j+|\beta|\le 2m-2$. 
 \end{proof}
 In summary, it follows that $T {\bfg}$ satisfies $\Delta^mT {\bfg}=0$, and is  sufficiently smooth near the boundary
 to have well-defined boundary values $\lambda_j T{\bfg}$, $j=0,\dots,m-1$.
%
%
%
%
%
\subsection{Fourier characterization of the boundary layer potentials}
\label{Fourier characterization of the boundary layer potentials}
The alternative representation of $T{\bfg}$ is to think of it
as a convolution of $\phi$ with a distribution supported on $\reals^{d-1}$, i.e., $T\bfg = \phi * \mu_{\bfg}$.
We adopt this view in order to give a Fourier description of the boundary layer potential operators. 

This idea is most easily illustrated for the initial ($j=0$) term $V_0$. In this case, for a given $g$ we have
$
V_0g(x) 
=  
\int_{\reals^{d-1}} g(y)\lambda_{0,y} \phi(x-y)\, \dif\sigma (y)
 =
 \left\langle 
  \phi(x- \cdot), g\otimes \delta
\right\rangle$,
 where $g\otimes \delta$ is the distribution defined by
 $\gamma \mapsto 
  \langle \gamma, g\otimes \delta \rangle = \int_{\reals^{d-1}}g(x') \gamma(x',0) \dif x'$.
For the higher order boundary layer potentials, we 
 apply $ \Lambda_j$, the natural extension of the boundary operator $\lambda_j$. 
Correspondingly, $ \Lambda_j^* =  \Delta^{\frac{j}{2}}$ when $j$ is even and  
$ \Lambda_j^*=\frac{\partial}{\partial x_d} \Delta^{\frac{j-1}{2}}$ when $j$ is odd.
In this case, the individual boundary layer potential operators are 
\begin{equation*}
V_j g (x)
= 
\int_{\reals^{d-1}} g(y)\lambda_{j,y} \phi(x-y)\, \dif\sigma (y)
=
\left\langle 
\Lambda_j  \bigl(  \phi(x- \cdot)\bigr), g\otimes \delta
\right\rangle
= 
\left\langle 
  \phi(x- \cdot), \Lambda_j^* \bigl( g\otimes \delta \bigr)
\right\rangle.
\end{equation*}
In other words, $V_j g = \phi*\bigl( \Lambda_j^* \bigl( g\otimes \delta \bigr) \bigr)$. 
\subsubsection*{Formal calculation} 
This leads to a standard calculation of the symbol of the operator $V_j$, which we now present formally. 
(In the interest of presenting a self-contained exposition, we present a slightly more direct approach below.)
 This can be simplified with  the usual formula expressing a convolution as the inverse Fourier transform of a product:
 $$
 V_j g 
 =
 \phi*\bigl( \Lambda_j^* \bigl( g\otimes \delta \bigr) \bigr)
 = 
 \left( 
   \left[\bigl( \Lambda_j^* \bigl( g\otimes \delta \bigr)\bigr)
 \right]^{\wedge}
 \widehat{\phi}
 \right)^{\vee}
 .$$
 The first factor can be expressed as 
 $ 
\sigma(\Lambda_j^*) (\xi) (g\otimes\delta)^{\wedge}(\xi)$.
The symbol of the differential operator is simply $\sigma(\Lambda_j^* )(\xi) = i^j |\xi|^j$ when $j$ is even and
$\sigma(\Lambda_j^*)(\xi) = i^j  |\xi|^{j-1}\xi_d$ when $j$ is  odd,
while  the Fourier transform of the distribution $g\otimes \delta $
is  $(g\otimes\delta)^{\wedge}(\xi) =\widehat{g}(\xi_1,\dots,\xi_{d-1})$. The Fourier transform of $\phi$ is a 
tempered distribution, leading to a distributional description of $V_j$. 
However, when considering test functions having a  Fourier transform supported away from the origin, 
$\widehat{\phi}(\xi) = (-1)^m|\xi|^{-2m}$.
This line of reasoning results in a distributional version of the formula \eqref{BLPO_Fourier} shown below. 
In order to obtain the pointwise formula, and to demonstrate that it is valid for a larger class of functions, we perform the following direct calculation. 

\subsubsection*{Direct calculation} 
 The Fourier transform of the tempered distribution $\phi$, restricted to test functions $\gamma$ with 
 $\supp{\widehat{\gamma}}\subset \reals^d\setminus\{0\}$, 
 can be represented by $(-1)^m|\xi|^{-2m}$. 
This is nearly sufficient for our purposes---we need to give a slightly more careful description of the behavior of the \emph{pseudo-function} 
 $(-1)^m|\xi|^{-2m}$ 
near the origin.
 The singular behavior at the origin can be treated by modifying the kernel  $e^{i\langle x,\xi\rangle}$ to force a high order zero (see \cite[Chapter 4.4]{GV2} or \cite{MaNe}). 
 By subtracting the Taylor polynomial of $e^{i\langle x,\xi\rangle}$ of degree $2m-d$ multiplied by a smooth, 
 compactly supported cut-off function $\psi$ equalling $1$ in a neighborhood of the origin,
  we obtain the kernel  
 $E(x,\xi) = e^{i\langle x,\xi\rangle} - \psi(\xi)\sum_{\ell=0}^{2m-d} \frac{(i\langle x,\xi\rangle)^{\ell}}{\ell!}$.
 \begin{lemma} 
 Given $\psi\in C_c^{\infty}(\reals^d)$ where $\psi(\xi) =1$ for $|\xi|\le 1$, 
 there is a polynomial $p\in\Pi_{2m-d}(\reals^d)$  (depending on $\phi$ and $\psi$) so that
 \begin{equation}\label{FourierTransformIdentity}
 D^{\alpha} \phi(x) = 
  (-1)^m\frac{i^{|\alpha|}}{(2\pi)^{d}}
  \int_{\reals^d}
    \frac{ \xi^{\alpha}}{|\xi|^{2m}}
    	\left(
	  e^{i\langle x,\xi\rangle} 
	  - 
	  \psi(\xi)\sum_{\ell=0}^{2m-d-|\alpha|} \frac{(i\langle x,\xi\rangle)^{\ell}}{\ell!}
	\right) 
\dif \xi 
  +D^{\alpha}p(x).
 \end{equation}
 \end{lemma}
 \begin{proof}
 It is not difficult to see that $\phi_2 := (-1)^m (2\pi)^{-d}\int_{\reals^d}|\xi|^{-2m} E(\cdot,\xi)\dif \xi$ is a
 continuous function, which exhibits polynomial growth. 
 It serves as a fundamental solution for test functions in
 $\mathcal{D}_{2m-d+1}$, the space of compactly supported $C^{\infty}$ functions $\gamma$ for which 
 $\widehat{\gamma}(\xi) = \mathcal{O}(|\xi|^{2m-d+1})$ near the origin. 
 Indeed, for such  $\gamma$, a direct calculation reveals
 $
 \phi_2*\gamma (x) 
 =  \frac{(-1)^m}{(2\pi)^{d}}
 \int_{\reals^d} \widehat{\gamma}(\xi) |\xi|^{-2m}e^{i\langle x,\xi\rangle} \dif \xi
$. 
By this observation, the Riemann-Lebesgue lemma applies and $\phi_2*\gamma$ vanishes at $\infty$. 
By Corollary \ref{Convolution_Decay}, the solution $\phi*\gamma$ also decays. 
Because $w :=(\phi_2 - \phi)*\gamma$ is a tempered distribution solving $\Delta^m w = 0$, 
its Fourier transform is supported at the origin, and $u$ must therefore be a polynomial.
Because $u$ vanishes at $\infty$, it must be trivial, and 
$\langle \phi_2,\gamma\rangle = \langle\phi,\gamma\rangle$  
for all $\gamma\in \mathcal{D}_{2m-d+1}$. 
By \cite[Proposition 2.6]{MaNe}, $\phi = \phi_2 + p$ for  some $p\in \Pi_{2m-d}$.
The lemma follows by dominated convergence, differentiating under the integral.
\end{proof}

To simplify the convolution of $\Lambda_j^*(g\otimes \delta)$  with $\phi$, we force $\widehat{g}$ 
(and therefore $(g\otimes \delta)^{\wedge}$) 
to have a high order zero at the origin.
%
%
\begin{condition}[Vanishing moments]
\label{fourierMoments}
For $j=0,\dots,\min(m-1,2m-d)$, 
we assume  there is $C>0$ and $\tau>1$ so that  
$
|\widehat{g}_j(\xi')|
\le 
C |\xi'|^{2m-j -d} 
\bigl| \log|\xi'|\bigr|^{-\tau} 
 $ 
 holds in a neighborhood of the origin.
\end{condition}
%
It is clear that Condition \ref{fourierMoments} and Condition \ref{firstDecay} together ensure that 
$g_j\perp \Pi_{2m-d-j}(\reals^{d-1})$.
From this, it follows that 
for  $p\in\Pi_{2m-d}$
the integral
$\int_{\reals^{d-1}} g_j(y) \Lambda_jp(x-y)\dif \sigma(y)$ 
vanishes, 
as does
$
\int_{\reals^{d-1}} g_j(y)
\sum_{\ell=0}^{2m-d-j} 
\frac{(i\langle x-y,\xi\rangle)^{\ell}}{\ell!} \dif \sigma(y)
$.

Note that $V_j g_j(x) =(-1)^j \int_{\reals^{d-1}} g_j(y') (\Lambda_j \phi) \bigl(x- (y',0)\bigr) \dif y'$, where 
$\Lambda_j  = \sum_{|\beta| = j} b_{\beta} D^{\beta}$ is a homogeneous operator of order $j$. 
A direct calculation with the identity \eqref{FourierTransformIdentity}  and Fubini's theorem 
gives the Fourier description of the boundary layer potential alluded to earlier (we write $\xi = (\xi',\xi_d)$ at this point):
%
\begin{eqnarray}
V_j g_j (x',x_d)  
 &=&
 (-1)^{m+j}
\sum_{|\beta|=j} 
\frac{i^{j}}{(2\pi)^{d}}
\int_{\reals^d}
 \frac{b_{\beta} \xi^{\beta}}{|\xi|^{2m}}
 e^{i\langle x,\xi\rangle}
  \left(
  \int_{\reals^{d-1}}
    e^{-i\langle y',\xi'\rangle}
    g_j(y') 
\,  \dif y'
 \right)   
\dif \xi\nonumber\\
&=&
(-1)^{m+j}
\sum_{|\beta|=j} 
\frac{i^{j}}{(2\pi)^{d}}
\int_{\reals^d}
 \frac{b_{\beta} \xi^{\beta} }{|\xi|^{2m}}\widehat{g_j}(\xi')
 e^{i\langle x,\xi\rangle}
 \,
 \dif \xi.\label{initial_symbol}
\end{eqnarray}
%

Observe that
$|\xi|^{2m} =(\xi_d^2 +|\xi'|^2)^m = |\xi'|^{2m}(1 + (\xi_d/|\xi'|)^2)^m$. 
Integrating first with respect to $\zeta:=\xi_d/|\xi'|$, and then with respect to  
$\xi'$, we can simplify \eqref{initial_symbol}.
For even values of $j$, this becomes
%
\begin{equation*}
V_j g_j (x) 
=\frac{(-1)^{j/2+m}  }{(2\pi)^{d}}
  \int_{\reals^{d-1}} 
  \left(
  \int_{\reals}
   \frac{ e^{i |\xi'| x_d \zeta}}{|\zeta^2 +1|^{\frac{2m-j}{2}} }
   \dif \zeta
   \right)
\frac{\widehat{g_j}(\xi') e^{i\langle x',\xi'\rangle}}{|\xi'|^{2m-j-1}}   \dif \xi'.
  \end{equation*}
For odd $j$, we have
\begin{equation*}
V_j g_j (x) = \frac{(-1)^{(j-1)/2 +m}}{ (2\pi)^{d}}
 \int_{\reals^{d-1}} 
  \left(
  \int_{\reals}
   \frac{ i \zeta e^{i  |\xi'| x_d \zeta }}{|\zeta^2 +1|^{\frac{2m-(j-1)}{2}} }
   \dif \zeta
   \right)
\frac{\widehat{g_j}(\xi')e^{i\langle x',\xi' \rangle} }{|\xi'|^{2m-j-1}}   \dif \xi'.
\end{equation*}
%
For general $j$, we write the boundary layer potential (regardless of parity) as 
%
%
\begin{gather}
\label{BLPO_Fourier}
\qquad V_j g_j (x)
=
  (2\pi)^{-d}
\int_{\reals^{d-1}} 
	\left(   
		\int_{\reals}
     			\tau_j(\zeta) e^{ i |\xi'| x_d \zeta  }
   		\dif \zeta  
	\right)
   	\frac{\widehat{g_j}(\xi') e^{i\langle x', \xi' \rangle}}{|\xi'|^{2m-j-1}}
\dif \xi',
\\
\intertext{with}
\tau_j(\zeta)  
:= 
\begin{cases}
         {(-1)^{\frac{j}{2}+m} }{\left|1+\zeta^2\right| ^{\frac{j}{2}-m}}, &\text{for $j$ even,}\\
	{(-1)^{\frac{j-1}{2}+m} i\zeta}{ \left|1+\zeta^2\right|^{\frac{j-1}2-m}}, &\text{for $j$ odd.}
\end{cases}
\nonumber
\end{gather}
%

Our goal is to represent the boundary values of the boundary layer potentials. 
To this end, we assume that $\widehat{g_j}(\xi')$ has the following behavior as $|\xi'|\to \infty$.
%
%
\begin{condition}[Smoothness]\label{fourierSmoothness}
For $j=0,\dots,m-1$, 
$\int_{\reals^{d-1}}|\widehat{g_j}(\xi')| (1+|\xi'|)^{j-m}\dif \xi' <\infty$ .
\end{condition}
%
Condition \ref{fourierSmoothness} allows partial derivatives to pass inside the integral in 
\eqref{BLPO_Fourier} of $V_j$. In this case we use the fact that $\sigma(\Lambda_k)(\xi) = (-1)^{k/2} |\xi|^k$
when $k$ is even and $\sigma(\Lambda_k)(\xi) = -i(-1)^{\frac{k-1}{2}} |\xi|^{\frac{k-1}{2}}\xi_d$ when $k$ is odd.
We  arrive at the following lemma.
%
%
\begin{lemma}\label{Poisson_symbol}
Let $\bfg$  satisfy Conditions \ref{firstDecay}, \ref{fourierMoments} and \ref{fourierSmoothness}.
For $k,j = 0\dots,m-1$, 
the $k$th order differential operator ${\Lambda}_k$ applied to the $j$th boundary layer potential satisfies
$$
{\Lambda}_k V_j g_j (x) = 
\int_{\reals^{d-1}} 
     \left(   \int_{\reals}
      \omega_{k,j}(\zeta) e^{ i |\xi'| x_d \zeta  }
   \dif \zeta  \right)
   \frac{\widehat{g_j}(\xi') e^{i\langle x', \xi' \rangle}}{|\xi'|^{2m-j-k-1}}
\dif \xi',
$$
where
$$
\omega_{k,j}(\zeta)  := \frac{(-1)^m}{(2\pi)^{d}}
\begin{cases}
(-1)^{\frac{j+k}{2}} \left|1+\zeta^2\right|^{\frac{j+k}{2}-m}, &\text{for $j,k$ both even,}\\
-(-1)^{\frac{j+k}{2}}\zeta^2 \left|1+\zeta^2\right|^{\frac{j+k-2}{2}-m}, &\text{for $j,k$ both odd,}\\
(-1)^{\frac{j+k-1}{2}}(i\zeta) \left|1+\zeta^2\right|^{\frac{j+k-1}{2}-m}, &\text{for $j+k$ odd.}
\end{cases}
$$ 
\end{lemma}
We note that, since $k+j\le 2m-2$, each function $\omega_{j,k}$ is in $L_1(\reals)$.

%
%
\section{Boundary values of the boundary layer potentials}
\label{Boundary values of the boundary layer potentials}
In this section, we begin our analysis of the solution $T{\bfg}$ by considering its boundary values.

%
%
\subsection{The matrix-valued kernel}\label{The matrix-valued kernel}
Since the boundary layer potentials $V_jg_j$ admit smooth extensions to functions in 
$C^{2m-j-2}(\reals^{d-1}\times[0,\infty))$ 
we may simply look for solutions of the system of integral equations
\begin{equation}\label{System}
\V 
\begin{pmatrix}
  g_0\\ 
  g_1\\
  \vdots\\
  g_{m-1}
\end{pmatrix}
:=
\begin{pmatrix}
  \lambda_0[ V_0 g_0 + V_1 g_1 + \dots + V_{m-1} g_{m-1}]\\
   \lambda_1 [V_0 g_0 +  V_1 g_1 + \dots +V_{m-1} g_{m-1}]\\
   \vdots\\
    \lambda_{m-1} [V_0 g_0 + V_1 g_1 + \dots + V_{m-1} g_{m-1}]\\
\end{pmatrix}
=
\begin{pmatrix}
h_0\\
h_1\\
\vdots\\
h_{m-1}
\end{pmatrix}.
\end{equation}
We remark that 
the 
operator $\V$ can be viewed as an integral operator with
matrix valued kernel, 
\begin{equation}\label{System_kernel}
\V 
\mathbf{g} (x)= \int_{\reals^{d-1}} \mathbf{K}(x,y) \mathbf{g} (y) \dif \sigma(y).
\end{equation}
The entries of the kernel are 
$\bigl(\mathbf{K}(x,y)\bigr)_{k,j} = (\lambda_k)_x  (\lambda_{j})_{y} \phi(x-y),$
and each entry gives rise (as a scalar kernel) to an operator $v_{k,j} = \lambda_k V_{j}$.
Furthermore, each such entry has behavior around the diagonal of order 
$\mathcal{O}\left( |x-y|^{(2m-d) - (j+k)} \bigl|\log|x-y|\bigr| \right)$, and since  $k+j\le 2m-2$, 
each is locally integrable on $\reals^{d-1}.$ 
In other words, 
$$v_{k,j} g(x) = 
\lambda_{k} V_j g(x) 
= 
 \lambda_{k}\int_{\reals^{d-1}}g(y)  \lambda_{j,y} \phi(x-y) \dif \sigma(y)
=  
\int_{\reals^{d-1}} g(y) \lambda_{k,x}\lambda_{j,y} \phi(x-y) \dif \sigma(y).$$

%
%
\subsection{Fourier characterization of the matrix-valued kernel} 
\label{Fourier characterization of the matrix-valued kernel} 
Another way of viewing the operator 
$\V$ is through the Fourier transform. 
In this case, we use Lemma \ref{Poisson_symbol} and dominated convergence to obtain
$v_{k,j} g_j= 
 \lim_{x_d \to 0} 
   \Lambda_k V_j g_j (x)
$. This is
\begin{equation}\label{scalar_operators}
v_{k,j} g_{j}(x)=
\int_{\reals^{d-1}} 
   \left(   
     \int_{\reals}
       \omega_{k,j}(\zeta) 
     \dif \zeta  
   \right)
   \frac{\widehat{g_j}(\xi') e^{i\langle y, \xi' \rangle}}{|\xi'|^{2m-j-k-1}}
\dif \xi'.
\end{equation}
%
Letting 
$     
\int_{\reals}
       \omega_{k,j}(\zeta) 
\dif \zeta 
= (2\pi)^{-(d-1)}(-1)^{\frac{j+k-2m}{2}}\M_{k,j}
$
we observe the symbol of the multiplier $v_{k,j} $ is 
\begin{equation}\label{scalar_symbol_first}
\sigma(v_{k,j})
(\xi') 
= 
(-1)^{\frac{j+k-2m}{2}} \M_{k,j} |\xi'|^{(j+k+1)-2m} .
\end{equation}
%
%
%
It follows that
\begin{equation}
\M_{k,j}
=
(2\pi)^{-1}\begin{cases} 
 \int_{-\infty}^{\infty} (1+\zeta^2)^{ (j+k)/2-m} \dif \zeta, & \text{for $j,k$ both even,}\\
- \int_{-\infty}^{\infty} \zeta^2\,(1+\zeta^2)^{ (j+k)/2-(m+1)} \dif \zeta, & \text{for $j,k$ both odd,}\\
0, & \text{for $j+k$ odd.}
\end{cases}
\end{equation}
We can compute the coefficients of these explicitly as products of dyadic integers with either (middle) binomial coefficients 
$
\b_j
:= \binom{2j}{j}$ or numbers 
$4\b_{j-1} - \b_j$, which are multiples of the ubiquitous\footnote{Richard Stanley's webpage {\tt http://www-math.mit.edu/~rstan/ec/catadd.pdf} now lists more than 180 different combinatorial structures counted by the Catalan numbers.} 
Catalan numbers 
$\c_{j}:=\binom{2j}{j} - \binom{2j}{j\,{+}\,1}$
(see Section \ref{sec: path counting and invertibility}). Indeed, using the binomial identities 
$\binom{2p}{p-1} =  \binom{2p}{p+1}$ and 
$\binom{p}{q} = \binom{p-1}{q-1} + \binom{p-1}{q}$,
one readily shows that $4\b_{j-1}-\b_j = 2\c_{j-1}$. 

%
%
\begin{proposition}\label{scalar_symbol}
The operators $v_{k,j}: g\mapsto \int_{\reals^{d-1}} g(y)\lambda_{k,x}\lambda_{j,y} \phi(x-y)\dif \sigma(y)$ are 
multiplier operators, having symbol 
$
\sigma(v_{k,j})
(\xi') = (-1)^{(j+k -2m)/2} \M_{k,j}\,|\xi'|^{j+k+1 -2m}$, where
\begin{gather}\label{matrix_symbol}
\M_{k,j} 
= 
\begin{cases}
2^{1+j+k-2m} \, \b_{m-(j+k)/2-1}, 
	& \text{for $j,k$ both even,} \\
-2^{j+k-2m} \,\c_{m-(j+k)/2-1},
	& \text{for $j,k$ both odd,}\\
0, 
	& \text{for $j+k$ odd.}
\end{cases}
\end{gather}
\end{proposition}
\begin{proof}
This follows from \eqref{scalar_symbol_first} by a direct calculation. 
In the first case, the integrand is odd and integrable, hence its integral is zero.

The second and third cases exploit the fact that $\zeta^2+1 = (\zeta - i)(\zeta+i)$, or more precisely, that
%
\begin{eqnarray*}
 \int_{-\infty}^{\infty} \frac{1}{(1+\zeta^2)^{ M} }\dif \zeta 
 &=&
 \frac{2\pi i}{(M-1)!} \left[\left(\frac{d}{d\zeta}\right)^{M-1} \frac{1}{(\zeta+i)^M}\right]_{\zeta = i} \\
 &=&
 \pi 2^{2-2M} \nchoosek{2M-2}{M-1}.
\end{eqnarray*}
%
The second case ($j$ and $k$ both even) follows by taking $M = m - n/2$.
To handle the third case,  observe that 
$\ds 
\frac{\zeta^2}{(1+\zeta^2)^{ M} }
=   
\frac{1}{(1+\zeta^2)^{ M-1} }
  - 
\frac{1}{(1+\zeta^2)^{ M} }.
$  
Thus
\begin{align*} 
\int_{-\infty}^{\infty} \frac{\zeta^2}{(1+\zeta^2)^{ M} }\dif \zeta 
&=
\pi 2^{2-2(M-1)} \b_{M-2} - \pi 2^{2-2M} \b_{M-1} \\
&= 
\pi 2^{2-2M} (4\b_{M-2} -\b_{M-1})
  =
\pi 2^{2-2M} \times2\c_{M-2}.
\end{align*}
The third case now follows by taking
$M = m+1 - n/2$.
\end{proof}

\section{Path Counting and Invertibility of the Matrix Symbol}
\label{sec: path counting and invertibility}
In this section we show that the matrix symbol $\M$ considered in Proposition \ref{scalar_symbol} is invertible by giving an explicit formula for its determinant (Proposition \ref{det_matrix_symbol}). We arrive at the formula by counting non-intersecting paths in certain graphs associated to $\M$. 

We perform a sequence of elementary transformations on $\M$ to uncover the combinatorics at play. First, we remove the powers of $2$ and $-1$ appearing in \eqref{matrix_symbol} by pre- and post-multiplying by diagonal 
matrices, $\M \mapsto D_1 \M D_2$, where
\[
	D_{1} = \diag{ 2^{m-j\,}2^{\oddq j-1}}_{j} 
\qand
	D_2 = \diag{ 2^{m-j\,}(-1)^{\oddq j} }_{j} \,.
\]
Here $j$ runs from 0 to $m{-}1$ and the function $\oddq j$ returns $1$ if $j$ is odd and $0$ otherwise. (E.g., one could put $\oddq{j} := \Oddq{j}$.) The result is a matrix populated with zeros, middle binomial numbers and Catalan numbers in a 
checkerboard pattern. Next, we permute its rows and columns to arrive at the convenient block-diagonal form
$$
\tilde\M = \begin{pmatrix} B & 0 \\ 0 & C \end{pmatrix}\!,
$$
where $B$ and $C$ are Hankel matrices (constant on anti-diagonals), populated with binomial numbers and 
Catalan numbers, respectively. 

\begin{definition} For $n\geq1$, define binomial and shifted binomial matrices by
\[
	\B(n) := \left(\b_{j+k}\right)_{j=0\dots n-1, k=0\dots n-1} 
\quad\hbox{and}\quad
	\sB(n) := \left(\b_{j+k+1}\right)_{j=0\dots n-1, k=0\dots n-1}\,.
\]
For $n\geq1$, define Catalan and shifted Catalan matrices by 
\[
	\C(n) := \left(\c_{j+k}\right)_{j=0\dots n-1, k=0\dots n-1}
\quad\hbox{and}\quad
	\sC(n):= \left(\c_{j+k+1}\right)_{j=0\dots n-1, k=0\dots n-1}\,.
\]
\end{definition}
For example, 
\[
\B(3) = \begin{pmatrix} 1 & 2 & 6 \\ 2 & 6 & 20 \\ 6 & 20 & 70 \end{pmatrix},
\quad 
\sB(3) = \begin{pmatrix} 2 & 6 & 20 \\ 6 & 20 & 70 \\ 20 & 70 & 252 \end{pmatrix}
\quad\hbox{and}\quad
\C(3) = \begin{pmatrix} 1 & 1 & 2 \\ 1 & 2 & 5 \\ 2 & 5 & 14 \end{pmatrix}.
\] 
We have
\begin{gather}\label{eq: tilde M}
\tilde\M = \begin{cases}
	\begin{pmatrix}\B(n+1) & 0 \\ 0 & \sC(n) \end{pmatrix}, & \text{for }m=2n+1\text{ odd,} \\[1ex]
	\begin{pmatrix}\sB(n) & 0 \\ 0 & \C(n) \end{pmatrix}, & \text{for }m=2n\text{ even.} \end{cases}
\end{gather}
Clearly, $\M$ is invertible for all $m\geq1$ if and only if the $\B,\sB,\C,\sC$ are invertible for all $n\geq1$.
Using a computer algebra system to compute determinants, we find some surprising patterns:
\begin{align}
\label{det B} \det\B\, &:\  1,2,4,8,16,32,\ldots \\
\label{det sB} \det\sB &:\  2,4,8,16,32,64\ldots \\
\label{det C} \det\C\, &:\  1,1,1,1,1,1,\ldots \\
\label{det sC} \det\sC &:\  1,1,1,1,1,1,\ldots .
\end{align}
Many techniques have been developed in the combinatorics community to deal with Hankel matrices; see \cite{Krat1,Krat2}. We prove the identities \eqref{det B}--\eqref{det sC} using one of these techniques (path counting) in Sections \ref{Middle binomial path counting} and \ref{Catalan path counting}. They give the following result.
%
%
\begin{proposition}\label{det_matrix_symbol} 
Let $\M =(\M_{j,k})$ be the $m\times m$ matrix in Proposition \ref{scalar_symbol}. Then 
$\displaystyle	
	\det(\M) = (-1)^{\lfloor m/2\rfloor} 2^{m^2}.	
$
\end{proposition}

\begin{proof} 
To produce $\tilde\M$ from $D_1\M D_2$, the rows and columns of $\M$ are permuted by ``deshuffling'' the even and odd indices, then reversing their relative order, e.g., $(0123456\dotsc) \mapsto (0246\dotsc , 135\dotsc) \mapsto (\dotsc6420,\dotsc531)$. This amounts to an orthogonal transformation $(\bm\cdot) \mapsto P (\bm\cdot) P^T$ and thus does not change the determinant. 

Turning to the diagonal matrices $D_1$ and $D_2$, we have 
\begin{gather}\label{eq: diagonal dets}
	\det D_1  = 2^{\lfloor m^2/2 \rfloor}
\quad\hbox{and}\quad
	\det D_2  = 2^{\binom{m+1}{2}}(-1)^{\lfloor m/2 \rfloor}.
\end{gather}
Indeed, the exponent of $2$ in $\det D_1$ is
\[	
	0+2+2+4+4+\dotsb+(m-1)+(m-1) \ = \ 4\left(1+\dotsb+\frac{m-1}2\right) \ = \ \frac{m^2-1}2
\] 
when $m$ is odd, and 
\[	
	1+1+3+3+\dotsb+(m-1)+(m-1) \ = \ 4\left(1+\dotsb+\frac{m}2\right) - m \ = \ \frac{m^2}2
\] 
when $m$ is even. This is the same as $\lfloor m^2/2 \rfloor$. The exponent of $2$ in $\det D_2$ is simpler: 
$1+2+\dotsb+m$. Here, $-1$s occur in the odd positions along the diagonal (indexing by $0\leq j< m$).
 
We next consider the block matrix $\tilde \M$. The identities \eqref{det B}--\eqref{det sC} give $\det (\tilde \M) = 2^{\lfloor m/2\rfloor}$ for all $m\geq2$. 

Finally, we analyze the total contributions of $\tilde \M$, $D_1$ and $D_2$ to the exponent of $2$ in $\det \M$. Since ${m(m-1)}/{2}$ is an integer, we may write
\begin{align*}
	\lfloor m^2/2 \rfloor + \binom{m+1}{2} - \lfloor m/2 \rfloor 
	&= \lfloor m(m-1)/2 + m/2 \rfloor + m(m+1)/2 - \lfloor m/2 \rfloor \\
	&= m(m-1)/2 + \lfloor m/2 \rfloor + m(m+1)/2 - \lfloor m/2 \rfloor = m^2,
\end{align*} 
from which the result follows.
\end{proof}

%
%
\subsection{Path counting with determinants}\label{Path counting with determinants}
Consider the following situation. A taxi cab picks up a passenger in town at the street corner labeled $o$ in Figure \ref{fig: one taxi} and is directed to proceed to the street corner labeled $d$. How many shortest routes are there from $o$ to $d$?
\begin{figure}[!hbt]
\psset{unit=0.8
, linewidth=.06}
\begin{pspicture}(4,2.6)(-1,-1)
\psgrid(-1,-1)(4,2)
\pscircle*[linewidth=0,linecolor=darkgray](0,0){.07}
\pscircle*[linewidth=0,linecolor=darkgray](3,2){.07}
\uput{0.20}[50](3,2){$d$}
\uput{0.20}[220](0,0){$o$}
\ChrPath(0,0.055)(0,1)(1,1)(2,1)(2,2)(2.95,2)

\end{pspicture}
\caption{Count the shortest routes.}
\label{fig: one taxi}
\end{figure}
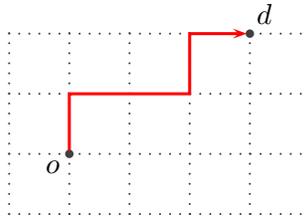
The answer is 10, as one can readily count. Note that each route has five steps and is uniquely determined by when the eastward steps are taken. We have the following generalization.

\begin{lemma} If the taxi must travel $i$ blocks east and $j$ blocks north, then there are $\binom{i+j}{i}$  shortest routes.
\end{lemma}

Now suppose there are $n$ taxi cabs originating from distinct corners $o_0,\ldots, o_{n-1}$ and destined for distinct corners $d_0,\ldots, d_{n-1}$. Given a permutation $\sigma \in \Sym_{n}$, an \demph{$n$-path} of type $\sigma$ is the pairing of taxis and destinations $o_i \leftrightarrow d_{\sigma(i)}$ together with a choice of shortest route for each cab. 
We might ask: How many different $n$-paths are there such that no two taxis' paths share an intersection? (See Figure \ref{fig: two taxis}.) 

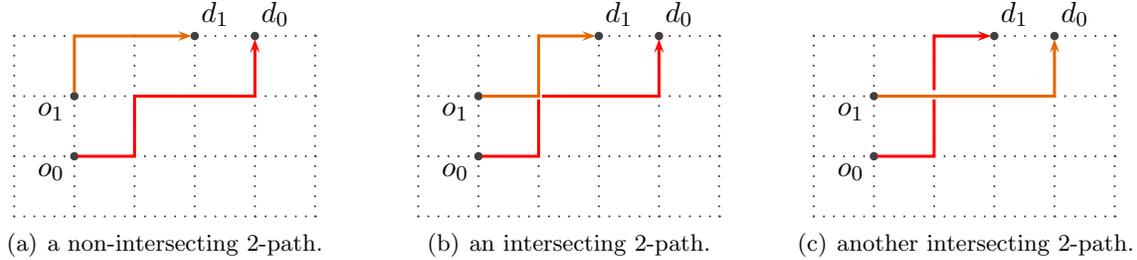
\begin{figure}[!hbt]
\centering
\subfigure[a non-intersecting $2$-path.]{%
\psset{unit=0.8, linewidth=.06}
\begin{pspicture}(5.5,3.6)(-.5,0)
\psgrid(0,0)(5,3)
\pscircle*[linewidth=0,linecolor=darkgray](1,1){.07}
\pscircle*[linewidth=0,linecolor=darkgray](4,3){.07}
\uput{0.20}[220](1,1){$o_0$}
\uput{0.20}[50](4,3){$d_0$}
\pscircle*[linewidth=0,linecolor=darkgray](1,2){.07}
\pscircle*[linewidth=0,linecolor=darkgray](3,3){.07}
\uput{0.20}[220](1,2){$o_1$}
\uput{0.20}[50](3,3){$d_1$}
\ChrPath[myred](1.055,1)(2,1)(2,2)(3,2)(4,2)(4,2.95)
\ChrPath[myorange](1,2.055)(1,3)(2,3)(2.95,3)
\end{pspicture}}
\hskip2.2em
\subfigure[an intersecting $2$-path.]{%
\psset{unit=0.8
, linewidth=.06}
\begin{pspicture}(5,3.6)(0,0)
\psgrid(0,0)(5,3)
\pscircle*[linewidth=0,linecolor=darkgray](1,1){.07}
\pscircle*[linewidth=0,linecolor=darkgray](4,3){.07}
\uput{0.20}[220](1,1){$o_0$}
\uput{0.20}[50](4,3){$d_0$}
\pscircle*[linewidth=0,linecolor=darkgray](1,2){.07}
\pscircle*[linewidth=0,linecolor=darkgray](3,3){.07}
\uput{0.20}[220](1,2){$o_1$}
\uput{0.20}[50](3,3){$d_1$}
\ChrPath[myred](1.055,1)(2,1)(2,2)(3,2)(4,2)(4,2.95)
\psline[linecolor=white,linewidth=.11](1.9,2)(2,2)(2,2.1)
\ChrPath[myorange](1.055,2)(2,2)(2,3)(2.95,3)
\end{pspicture}}
\hskip2em
\subfigure[another intersecting $2$-path.]{%
\psset{unit=0.8
, linewidth=.06}
\begin{pspicture}(5.5,3.6)(-0.45,0)
\psgrid(0,0)(5,3)
\pscircle*[linewidth=0,linecolor=darkgray](1,1){.07}
\pscircle*[linewidth=0,linecolor=darkgray](4,3){.07}
\uput{0.20}[220](1,1){$o_0$}
\uput{0.20}[50](4,3){$d_0$}
\pscircle*[linewidth=0,linecolor=darkgray](1,2){.07}
\pscircle*[linewidth=0,linecolor=darkgray](3,3){.07}
\uput{0.20}[220](1,2){$o_1$}
\uput{0.20}[50](3,3){$d_1$}
\ChrPath[myred](1.055,1)(2,1)(2,2)(2,3)(2.95,3)
\psline[linecolor=white,linewidth=.11](1.9,2)(2,2)(2,2.1)
\ChrPath[myorange](1.055,2)(2,2)(3,2)(4,2)(4,2.95)
\end{pspicture}}
\caption{How many non-intersecting $2$-paths?}
\label{fig: two taxis}
\end{figure}

The answer to the question in Figure \ref{fig: two taxis} is $15$, which also happens to be the determinant of the matrix
\[
\begin{pmatrix}\binom{3+2}{3} & \binom{2+2}{2} \\[1ex]
		\binom{2+1}{2} & \binom{3+1}{3}\end{pmatrix}\!,
\]
whose $(i,j)$-th entry is the number of shortest routes from origin $o_i$ to destination $d_j$. This  coincidence is a special case of a theorem first discovered by Karlin and McGregor \cite{KMc}, rediscovered by Lindstr\"om \cite{Lind} and popularized by Gessel and Viennot \cite{GV}. To properly state it, we first need a bit more notation.

A directed graph $G$ is a pair of finite sets $(\mathcal V,\mathcal E)$, where $\mathcal E \subseteq \mathcal V\times \mathcal V$. The elements $v\in \mathcal V$ are called \demph{vertices} and the elements $(v,w)\in\mathcal E$ are called \demph{(directed) edges.} A \demph{path} of length $k$ from one vertex $o$ of $G$ to another $d$ is a sequence of edges $e_1e_2\dotsb e_k$ satisfying: (a) the origin of $e_1$ is $o$; (b) the destination of $e_i$ is the origin of $e_{i+1}$ (for all $1\leq i < k$); and (c) the destination of $e_k$ is $d$. (One such path is shown in Figure \ref{fig: one taxi}; edges point north or east.) A graph is \demph{acyclic} if there is no (positive-length) path from any $v\in \mathcal V$ to itself. Given $n$ origin and destination vertices, $\{o_i\},\{d_i\} \subseteq \mathcal V$,  an \demph{$n$-path} in $G$ is defined as in the preceding lattice path discussion; it is non-intersecting when no vertex of $G$ is used twice. A choice of origin and destination vertices is called \demph{non-permutable} if all non-intersecting $n$-paths are forced to pair $o_i$ with $d_i$ (for all $0\leq i <n$).

\begin{theorem}\label{thm: path counting} 
Let $G$ be a directed acyclic graph with designated origin and destination nodes $\{o_i\}$ and $\{d_i\}$ $(0\leq i <n)$, and let $A$ be the $n\times n$ matrix whose $(i,j)$-th entry is the number of paths in $G$ from $o_i$ to $d_j$. If $G$ is non-permutable, then the number of nonintersecting $n$-paths is equal to the determinant of $A$.
\end{theorem}

\begin{proof}[Sketch of Proof] We illustrate the key idea of the proof before turning to its application (cf. Figure \ref{fig: labeled edges}). First, label each edge of $G$. Now, instead of counting the paths from $o_i$ to $d_j$, sum the corresponding path monomials to build a matrix $\tilde A$. (Setting all variables equal to $1$ gives the matrix $A$ in the theorem.) 
\begin{figure}[!hbt]
\centering
\subfigure[an intersecting $2$-path.]{%
\psset{unit=1.0, linewidth=.06}
\begin{pspicture}(5.5,3.6)(-.7,0.8)
\psgrid(1,1)(4,3)
\pscircle*[linewidth=0,linecolor=darkgray](1,1){.05}
\pscircle*[linewidth=0,linecolor=darkgray](3,2){.05}
\pscircle*[linewidth=0,linecolor=darkgray](1,2){.05}
\pscircle*[linewidth=0,linecolor=darkgray](2,3){.05}
\ChrPath[myred](1.055,1)(2,1)(2,2)(2.95,2)
\psline[linecolor=white,linewidth=.11](1.9,2)(2,2)(2,2.1)
\ChrPath[myorange](1.055,2)(2,2)(2,2.95)
\uput{0.10}[90](1.4,1){\color{myred}{\small$a$}}
\uput{0.10}[90](2.4,1){\color{gray}{\small$b$}}
\uput{0.10}[90](1.4,2){\color{myorange}{\small$c$}}
\uput{0.10}[90](2.4,2){\color{myred}{\small$d$}}
\uput{0.10}[90](1.4,3){\color{gray}{\small$e$}}
\uput{0.10}[90](2.4,3){\color{gray}{\small$f$}}
\uput{0.10}[180](1,1.6){\color{gray}{\small$u$}}
\uput{0.10}[180](1,2.6){\color{gray}{\small$v$}}
\uput{0.10}[180](2,1.6){\color{myred}{\small$w$}}
\uput{0.10}[180](2,2.6){\color{myorange}{\small$x$}}
\uput{0.10}[180](3,1.6){\color{gray}{\small$y$}}
\uput{0.10}[180](3,2.6){\color{gray}{\small$z$}}
\end{pspicture}\label{fig: labeled-a}}
\hskip3em
\subfigure[another intersecting $2$-path.]{%
\psset{unit=1.0
, linewidth=.06}
\begin{pspicture}(5.5,3.6)(-.6,0.8)
\psgrid(1,1)(4,3)
\pscircle*[linewidth=0,linecolor=darkgray](1,1){.05}
\pscircle*[linewidth=0,linecolor=darkgray](3,2){.05}
\pscircle*[linewidth=0,linecolor=darkgray](1,2){.05}
\pscircle*[linewidth=0,linecolor=darkgray](2,3){.05}
\ChrPath[myred](1.055,1)(2,1)(2,2)(2,2.95)
\psline[linecolor=white,linewidth=.11](1.9,2)(2.1,2)
\ChrPath[myorange](1.055,2)(2,2)(2.95,2)
\uput{0.10}[90](1.4,1){\color{myred}{\small$a$}}
\uput{0.10}[90](2.4,1){\color{gray}{\small$b$}}
\uput{0.10}[90](1.4,2){\color{myorange}{\small$c$}}
\uput{0.10}[90](2.4,2){\color{myorange}{\small$d$}}
\uput{0.10}[90](1.4,3){\color{gray}{\small$e$}}
\uput{0.10}[90](2.4,3){\color{gray}{\small$f$}}
\uput{0.10}[180](1,1.6){\color{gray}{\small$u$}}
\uput{0.10}[180](1,2.6){\color{gray}{\small$v$}}
\uput{0.10}[180](2,1.6){\color{myred}{\small$w$}}
\uput{0.10}[180](2,2.6){\color{myred}{\small$x$}}
\uput{0.10}[180](3,1.6){\color{gray}{\small$y$}}
\uput{0.10}[180](3,2.6){\color{gray}{\small$z$}}
\end{pspicture}\label{fig: labeled-b}}
\caption[Sample matrix $\tilde A(G)$]{$\tilde A(G) = 
\protect{\begin{pmatrix}  
	ucd + awd + aby  &  uve + ucx + awx \\ 
	cd  &  ve + cx
\end{pmatrix}}$.}
\label{fig: labeled edges}
\end{figure}
The idea is to use the permutation definition of determinant, $\det\tilde A = \sum{\sigma\in\Sym_n} \mathrm{sgn}(\sigma) a_{0\sigma(0)}a_{1\sigma(1)}\dotsb a_{{n-1}\sigma(n-1)}$. Notice that the determinant of $\tilde A$ sees any intersection within an $n$-path twice: once with $\mathrm{sgn}(\sigma)=+1$ and once with $\mathrm{sgn}(\sigma)=-1$. 
Hence the only terms surviving in $\det\tilde A$ come from non-intersecting $n$-paths.
\end{proof}

If the number of non-intersecting $n$-paths in a graph $G$ can be computed by inspection, then the theorem provides a simple way to compute $\det A$. We now apply this technique to compute the determinants of $\B(n)$, $\sB(n)$, $\C(n)$ and $\sC(n)$ for all $n\geq1$.

%
%
\subsection{Middle binomial path counting}
\label{Middle binomial path counting}
Let $G_n$ be the triangular lattice with $2n{-}1$ vertices on the diagonal. Arrange the origin and destination vertices in opposite directions along the diagonal, starting from the center. Orient lattice edges northward and eastward, as in the taxi discussion, and let there be two edges northward from the main diagonal; see Figure \ref{fig: folded-a}. The $n$-paths in $G_n$ are in $1${-}$1$ correspondence with $n$-paths in a square lattice; see Figure \ref{fig: folded-c}. 
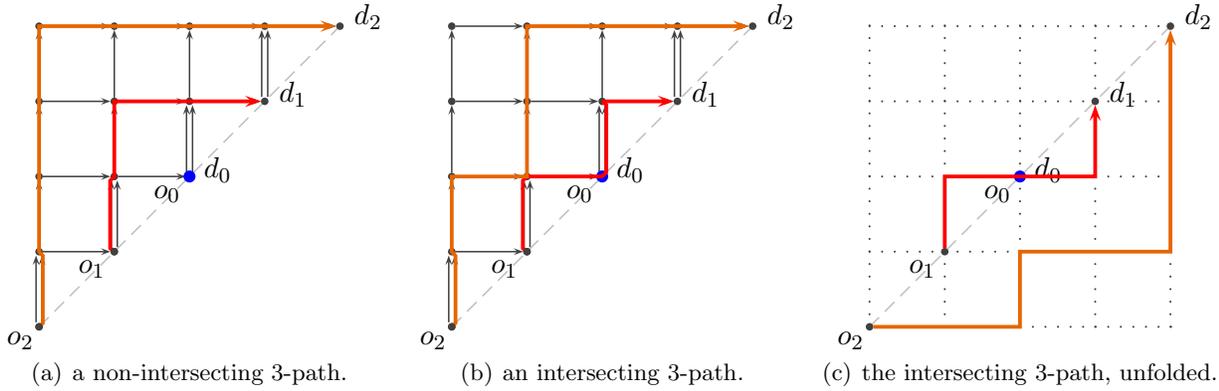
\begin{figure}[!hbt]
\subfigure[a non-intersecting $3$-path.]{%
\psset{unit=1.0, linewidth=.02,linecolor=darkgray,nodesep=0}
\begin{pspicture}(2.15,2.25)(-2.15,-2.25)
%
\psline[linewidth=0.02,linecolor=lightgray,linestyle=dashed]{-}(-2,-2)(2,2)
\cnode*[linewidth=0](-2,-2){.05}{P00}
\cnode*[linewidth=0](-2,-1){.05}{P01}
\cnode*[linewidth=0](-2,0){.05}{P02}
\cnode*[linewidth=0](-2,1){.05}{P03}
\cnode*[linewidth=0](-2,2){.05}{P04}
\cnode*[linewidth=0](-1,-1){.05}{P11}
\cnode*[linewidth=0](-1,0){.05}{P12}
\cnode*[linewidth=0](-1,1){.05}{P13}
\cnode*[linewidth=0](-1,2){.05}{P14}
\cnode*[linewidth=0](0,0){.05}{P22}
\cnode*[linewidth=0](0,1){.05}{P23}
\cnode*[linewidth=0](0,2){.05}{P24}
\cnode*[linewidth=0](1,1){.05}{P33}
\cnode*[linewidth=0](1,2){.05}{P34}
\cnode*[linewidth=0](2,2){.05}{P44}
%
\ncline[offset=-.04,linestyle=solid,dotsep=.08,nodesepA=0.01]{->}{P00}{P01}
\ncline[offset=.04,linestyle=solid,dotsep=.08,nodesepA=0.01]{->}{P00}{P01}
\ncline[offset=-.04,linestyle=solid,dotsep=.08,nodesepA=0.01]{->}{P11}{P12}
\ncline[offset=.04,linestyle=solid,dotsep=.08,nodesepA=0.01]{->}{P11}{P12}
\ncline[offset=-.04,linestyle=solid,dotsep=.08,nodesepA=0.01]{->}{P22}{P23}
\ncline[offset=.04,linestyle=solid,dotsep=.08,nodesepA=0.01]{->}{P22}{P23}
\ncline[offset=-.04,linestyle=solid,dotsep=.08,nodesepA=0.01]{->}{P33}{P34}
\ncline[offset=.04,linestyle=solid,dotsep=.08,nodesepA=0.01]{->}{P33}{P34}
%
%
\ncline[linestyle=solid,dotsep=.08]{->}{P01}{P02}
\ncline[linestyle=solid,dotsep=.08]{->}{P02}{P03}
\ncline[linestyle=solid,dotsep=.08]{->}{P03}{P04}
\ncline[linestyle=solid,dotsep=.08]{->}{P12}{P13}
\ncline[linestyle=solid,dotsep=.08]{->}{P13}{P14}
\ncline[linestyle=solid,dotsep=.08]{->}{P23}{P24}
\ncline[linestyle=solid,dotsep=.08]{->}{P01}{P11}
\ncline[linestyle=solid,dotsep=.08]{->}{P02}{P12}
\ncline[linestyle=solid,dotsep=.08]{->}{P03}{P13}
\ncline[linestyle=solid,dotsep=.08]{->}{P04}{P14}
\ncline[linestyle=solid,dotsep=.08]{->}{P12}{P22}
\ncline[linestyle=solid,dotsep=.08]{->}{P13}{P23}
\ncline[linestyle=solid,dotsep=.08]{->}{P14}{P24}
\ncline[linestyle=solid,dotsep=.08]{->}{P23}{P33}
\ncline[linestyle=solid,dotsep=.08]{->}{P24}{P34}
\ncline[linestyle=solid,dotsep=.08]{->}{P34}{P44}
\uput{0.17}[225](0,0){$o_0$}
\uput{0.17}[225](-1,-1){$o_1$}
\uput{0.10}[225](-2,-2){$o_2$}
\uput{0.19}[15](0,0){$d_0$}
\uput{0.19}[15](1,1){$d_1$}
\uput{0.19}[20](2,2){$d_2$}
\pscircle*[linewidth=0.05,linecolor=myblue](0,0){.08}
\ChrPath[myred](-1,-1)(-1.055,-0.955)(-1.055,-0.055)(-1,0)(-1,1)(0.95,1)
\ChrPath[myorange](-2,-2)(-1.95,-1.95)(-1.95,-1.055)(-2,-1)(-2,2)(0,2)(1,2)(1.95,2)
\pscircle*[linewidth=0.05,linecolor=darkgray](-2,-2){.05}
\pscircle*[linewidth=0.05,linecolor=darkgray](-1,-1){.05}
\end{pspicture}\label{fig: folded-a}}
\hskip2.75em
\subfigure[an intersecting $3$-path.]{%
\psset{unit=1.0, linewidth=.02,linecolor=darkgray,nodesep=0}
\begin{pspicture}(2.15,2.25)(-2.15,-2.25)
%
\psline[linewidth=0.02,linecolor=lightgray,linestyle=dashed]{-}(-2,-2)(2,2)
\cnode*[linewidth=0](-2,-2){.05}{P00}
\cnode*[linewidth=0](-2,-1){.05}{P01}
\cnode*[linewidth=0](-2,0){.05}{P02}
\cnode*[linewidth=0](-2,1){.05}{P03}
\cnode*[linewidth=0](-2,2){.05}{P04}
\cnode*[linewidth=0](-1,-1){.05}{P11}
\cnode*[linewidth=0](-1,0){.05}{P12}
\cnode*[linewidth=0](-1,1){.05}{P13}
\cnode*[linewidth=0](-1,2){.05}{P14}
\cnode*[linewidth=0](0,0){.05}{P22}
\cnode*[linewidth=0](0,1){.05}{P23}
\cnode*[linewidth=0](0,2){.05}{P24}
\cnode*[linewidth=0](1,1){.05}{P33}
\cnode*[linewidth=0](1,2){.05}{P34}
\cnode*[linewidth=0](2,2){.05}{P44}
%
\ncline[offset=-.04,linestyle=solid,dotsep=.08,nodesepA=0.01]{->}{P00}{P01}
\ncline[offset=.04,linestyle=solid,dotsep=.08,nodesepA=0.01]{->}{P00}{P01}
\ncline[offset=-.04,linestyle=solid,dotsep=.08,nodesepA=0.01]{->}{P11}{P12}
\ncline[offset=.04,linestyle=solid,dotsep=.08,nodesepA=0.01]{->}{P11}{P12}
\ncline[offset=-.04,linestyle=solid,dotsep=.08,nodesepA=0.01]{->}{P22}{P23}
\ncline[offset=.04,linestyle=solid,dotsep=.08,nodesepA=0.01]{->}{P22}{P23}
\ncline[offset=-.04,linestyle=solid,dotsep=.08,nodesepA=0.01]{->}{P33}{P34}
\ncline[offset=.04,linestyle=solid,dotsep=.08,nodesepA=0.01]{->}{P33}{P34}
%
%
\ncline[linestyle=solid,dotsep=.08]{->}{P01}{P02}
\ncline[linestyle=solid,dotsep=.08]{->}{P02}{P03}
\ncline[linestyle=solid,dotsep=.08]{->}{P03}{P04}
\ncline[linestyle=solid,dotsep=.08]{->}{P12}{P13}
\ncline[linestyle=solid,dotsep=.08]{->}{P13}{P14}
\ncline[linestyle=solid,dotsep=.08]{->}{P23}{P24}
\ncline[linestyle=solid,dotsep=.08]{->}{P01}{P11}
\ncline[linestyle=solid,dotsep=.08]{->}{P02}{P12}
\ncline[linestyle=solid,dotsep=.08]{->}{P03}{P13}
\ncline[linestyle=solid,dotsep=.08]{->}{P04}{P14}
\ncline[linestyle=solid,dotsep=.08]{->}{P12}{P22}
\ncline[linestyle=solid,dotsep=.08]{->}{P13}{P23}
\ncline[linestyle=solid,dotsep=.08]{->}{P14}{P24}
\ncline[linestyle=solid,dotsep=.08]{->}{P23}{P33}
\ncline[linestyle=solid,dotsep=.08]{->}{P24}{P34}
\ncline[linestyle=solid,dotsep=.08]{->}{P34}{P44}
\uput{0.17}[225](0,0){$o_0$}
\uput{0.17}[225](-1,-1){$o_1$}
\uput{0.10}[225](-2,-2){$o_2$}
\uput{0.19}[15](0,0){$d_0$}
\uput{0.19}[15](1,1){$d_1$}
\uput{0.19}[20](2,2){$d_2$}
\pscircle*[linewidth=0.05,linecolor=myblue](0,0){.08}
\psline[linewidth=0.05,linecolor=myred]{-}(-1,-1)(-1.055,-0.955)(-1.055,-0.055)(-1,0)(0,0)(0.055,0.055)(0.055,0.955)(0,1)
\ChrPath[myred](0,1)(0.95,1)
\ChrPath[myorange](-2,-2)(-1.95,-1.95)(-1.95,-1.055)(-2,-1)(-2,0)(-1,0)(-1,2)(0,2)(1,2)(1.95,2)
\pscircle*[linewidth=0.05,linecolor=darkgray](-2,-2){.05}
\pscircle*[linewidth=0.05,linecolor=darkgray](-1,-1){.05}
\end{pspicture}\label{fig: folded-b}}
\hskip1.5em
\subfigure[the intersecting $3$-path, unfolded.]{%
\psset{unit=1.0, linewidth=.02,linecolor=darkgray,nodesep=0}
\begin{pspicture}(2.7,2.25)(-2.7,-2.25)
\psgrid(-2,-2)(2,2)
\psline[linewidth=0.02,linecolor=lightgray,linestyle=dashed]{-}(-2,-2)(2,2)
\cnode*[linewidth=0](-2,-2){.05}{P00}
\cnode*[linewidth=0](-1,-1){.05}{P11}
\cnode*[linewidth=0](0,0){.05}{P22}
\cnode*[linewidth=0](1,1){.05}{P33}
\cnode*[linewidth=0](2,2){.05}{P44}
\uput{0.17}[225](0,0){$o_0$}
\uput{0.17}[225](-1,-1){$o_1$}
\uput{0.10}[225](-2,-2){$o_2$}
\uput{0.19}[15](0,0){$d_0$}
\uput{0.19}[15](1,1){$d_1$}
\uput{0.19}[20](2,2){$d_2$}
\pscircle*[linewidth=0.05,linecolor=myblue](0,0){.08}
\ChrPath[myred](-1,-0.95)(-1,0)(0,0)(1,0)(1,0.95)
\ChrPath[myorange](-1.95,-2)(-1,-2)(0,-2)(0,-1)(1,-1)(2,-1)(2,0)(2,1)(2,1.95)
\end{pspicture}\label{fig: folded-c}}
\caption{How many non-intersecting $3$-paths?}
\label{fig: folded}
\end{figure}

Figure \ref{fig: folded-c} makes it clear that $A(G_n)$ is populated with middle binomial numbers. Specifically, $A_{ij} = \b_{i+j}.$ 
That is, $A(G_n) = \B(n)$. 
Figure \ref{fig: folded-a} makes it clear that there are precisely $2^{n-1}$ distinct non-intersecting $n$-paths. Conclude that $\det\B(n) = 2^{n-1}$.

The fact that $\det\sB(n) = 2^{n}$ is now easy to see. Let $G'_n$ be as above, except having $2n$ vertices on the diagonal. Again arrange the $n$ origin and destination vertices along the diagonal. Since $o_0$ and $d_0$ no longer overlap, we get an extra factor of $2$ in the determinant.

%
%
\subsection{Catalan path counting}\label{Catalan path counting} 
Consider the triangular lattice $T_n$ with $n{+}1$ vertices on the diagonal; see Figure \ref{fig: catalan}. Among the many things the Catalan numbers are known to count are the distinct lattice paths between the southwest and northeast corners of $T_n$ \cite[Exercise 6.19(h)]{Sta}---commonly known as \demph{Dyck paths}. These are counted by  $\c_n = \binom{2n}{n} - \binom{2n}{n+1}$ using a reflection trick of Aebly \cite{Aeb,Gar1,Gar2,Kos}. Briefly, the binomial coefficient $\binom{2n}{n}$ counts paths in the square grid from $(0,0)$ to $(n,n)$. The paths that cross the diagonal are removed from the count by identifying them with paths from $(0,0)$ to $(n{+}1,n)$ that end in a vertical step $\binom{2n}{n+1}$.\footnote{Recall that $\binom{2n+1}{n+1} = \binom{2n}{n+1} + \binom{2n}{n}$. The first summand accounts for paths ending in a vertical step (there are still $n{+}1$ horizontal steps to take); the second for paths ending in a horizontal step.}
\begin{figure}[!hbt]
\psset{unit=0.8, linewidth=.02,linecolor=darkgray,nodesep=0}
\begin{pspicture}(1.15,1.15)(-2.15,-1.85)
%
\psline[linewidth=0.02,linecolor=lightgray,linestyle=dashed]{-}(-2,-2)(1,1)
\cnode*[linewidth=0](-2,-2){.07}{P00}
\cnode*[linewidth=0](-2,-1){.07}{P01}
\cnode*[linewidth=0](-2,0){.07}{P02}
\cnode*[linewidth=0](-2,1){.07}{P03}
\cnode*[linewidth=0](-1,-1){.07}{P11}
\cnode*[linewidth=0](-1,0){.07}{P12}
\cnode*[linewidth=0](-1,1){.07}{P13}
\cnode*[linewidth=0](0,0){.07}{P22}
\cnode*[linewidth=0](0,1){.07}{P23}
\cnode*[linewidth=0](1,1){.07}{P33}
%
\ncline[linestyle=solid,dotsep=.08]{->}{P00}{P01}
\ncline[linestyle=solid,dotsep=.08]{->}{P11}{P12}
\ncline[linestyle=solid,dotsep=.08]{->}{P22}{P23}
\ncline[linestyle=solid,dotsep=.08]{->}{P01}{P02}
\ncline[linestyle=solid,dotsep=.08]{->}{P02}{P03}
\ncline[linestyle=solid,dotsep=.08]{->}{P12}{P13}
\ncline[linestyle=solid,dotsep=.08]{->}{P01}{P11}
\ncline[linestyle=solid,dotsep=.08]{->}{P02}{P12}
\ncline[linestyle=solid,dotsep=.08]{->}{P03}{P13}
\ncline[linestyle=solid,dotsep=.08]{->}{P12}{P22}
\ncline[linestyle=solid,dotsep=.08]{->}{P13}{P23}
\ncline[linestyle=solid,dotsep=.08]{->}{P14}{P24}
\ncline[linestyle=solid,dotsep=.08]{->}{P23}{P33}
\ncline[linestyle=solid,dotsep=.08]{->}{P24}{P34}
\ChrPath[myred](-2,-1.95)(-2,-1)(-2,1)(0.95,1)
\pcline[offset=-10pt]{|-|}(-2,-2)(1,1)
\lput*{:U}{3 steps}
\end{pspicture}
\hskip1.0em
\begin{pspicture}(1.15,1.15)(-2.15,-1.85)
%
\psline[linewidth=0.02,linecolor=lightgray,linestyle=dashed]{-}(-2,-2)(1,1)
\cnode*[linewidth=0](-2,-2){.07}{P00}
\cnode*[linewidth=0](-2,-1){.07}{P01}
\cnode*[linewidth=0](-2,0){.07}{P02}
\cnode*[linewidth=0](-2,1){.07}{P03}
\cnode*[linewidth=0](-1,-1){.07}{P11}
\cnode*[linewidth=0](-1,0){.07}{P12}
\cnode*[linewidth=0](-1,1){.07}{P13}
\cnode*[linewidth=0](0,0){.07}{P22}
\cnode*[linewidth=0](0,1){.07}{P23}
\cnode*[linewidth=0](1,1){.07}{P33}
%
\ncline[linestyle=solid,dotsep=.08]{->}{P00}{P01}
\ncline[linestyle=solid,dotsep=.08]{->}{P11}{P12}
\ncline[linestyle=solid,dotsep=.08]{->}{P22}{P23}
\ncline[linestyle=solid,dotsep=.08]{->}{P01}{P02}
\ncline[linestyle=solid,dotsep=.08]{->}{P02}{P03}
\ncline[linestyle=solid,dotsep=.08]{->}{P12}{P13}
\ncline[linestyle=solid,dotsep=.08]{->}{P01}{P11}
\ncline[linestyle=solid,dotsep=.08]{->}{P02}{P12}
\ncline[linestyle=solid,dotsep=.08]{->}{P03}{P13}
\ncline[linestyle=solid,dotsep=.08]{->}{P12}{P22}
\ncline[linestyle=solid,dotsep=.08]{->}{P13}{P23}
\ncline[linestyle=solid,dotsep=.08]{->}{P14}{P24}
\ncline[linestyle=solid,dotsep=.08]{->}{P23}{P33}
\ncline[linestyle=solid,dotsep=.08]{->}{P24}{P34}
\ChrPath[myred](-2,-1.95)(-2,0)(-1,0)(-1,1)(0.95,1)
\end{pspicture}
\hskip1.0em
\begin{pspicture}(1.15,1.15)(-2.15,-1.85)
%
\psline[linewidth=0.02,linecolor=lightgray,linestyle=dashed]{-}(-2,-2)(1,1)
\cnode*[linewidth=0](-2,-2){.07}{P00}
\cnode*[linewidth=0](-2,-1){.07}{P01}
\cnode*[linewidth=0](-2,0){.07}{P02}
\cnode*[linewidth=0](-2,1){.07}{P03}
\cnode*[linewidth=0](-1,-1){.07}{P11}
\cnode*[linewidth=0](-1,0){.07}{P12}
\cnode*[linewidth=0](-1,1){.07}{P13}
\cnode*[linewidth=0](0,0){.07}{P22}
\cnode*[linewidth=0](0,1){.07}{P23}
\cnode*[linewidth=0](1,1){.07}{P33}
%
\ncline[linestyle=solid,dotsep=.08]{->}{P00}{P01}
\ncline[linestyle=solid,dotsep=.08]{->}{P11}{P12}
\ncline[linestyle=solid,dotsep=.08]{->}{P22}{P23}
\ncline[linestyle=solid,dotsep=.08]{->}{P01}{P02}
\ncline[linestyle=solid,dotsep=.08]{->}{P02}{P03}
\ncline[linestyle=solid,dotsep=.08]{->}{P12}{P13}
\ncline[linestyle=solid,dotsep=.08]{->}{P01}{P11}
\ncline[linestyle=solid,dotsep=.08]{->}{P02}{P12}
\ncline[linestyle=solid,dotsep=.08]{->}{P03}{P13}
\ncline[linestyle=solid,dotsep=.08]{->}{P12}{P22}
\ncline[linestyle=solid,dotsep=.08]{->}{P13}{P23}
\ncline[linestyle=solid,dotsep=.08]{->}{P14}{P24}
\ncline[linestyle=solid,dotsep=.08]{->}{P23}{P33}
\ncline[linestyle=solid,dotsep=.08]{->}{P24}{P34}
\ChrPath[myred](-2,-1.95)(-2,0)(0,0)(0,1)(0.95,1)
\end{pspicture}
\hskip1.0em
\begin{pspicture}(1.15,1.15)(-2.15,-1.85)
%
\psline[linewidth=0.02,linecolor=lightgray,linestyle=dashed]{-}(-2,-2)(1,1)
\cnode*[linewidth=0](-2,-2){.07}{P00}
\cnode*[linewidth=0](-2,-1){.07}{P01}
\cnode*[linewidth=0](-2,0){.07}{P02}
\cnode*[linewidth=0](-2,1){.07}{P03}
\cnode*[linewidth=0](-1,-1){.07}{P11}
\cnode*[linewidth=0](-1,0){.07}{P12}
\cnode*[linewidth=0](-1,1){.07}{P13}
\cnode*[linewidth=0](0,0){.07}{P22}
\cnode*[linewidth=0](0,1){.07}{P23}
\cnode*[linewidth=0](1,1){.07}{P33}
%
\ncline[linestyle=solid,dotsep=.08]{->}{P00}{P01}
\ncline[linestyle=solid,dotsep=.08]{->}{P11}{P12}
\ncline[linestyle=solid,dotsep=.08]{->}{P22}{P23}
\ncline[linestyle=solid,dotsep=.08]{->}{P01}{P02}
\ncline[linestyle=solid,dotsep=.08]{->}{P02}{P03}
\ncline[linestyle=solid,dotsep=.08]{->}{P12}{P13}
\ncline[linestyle=solid,dotsep=.08]{->}{P01}{P11}
\ncline[linestyle=solid,dotsep=.08]{->}{P02}{P12}
\ncline[linestyle=solid,dotsep=.08]{->}{P03}{P13}
\ncline[linestyle=solid,dotsep=.08]{->}{P12}{P22}
\ncline[linestyle=solid,dotsep=.08]{->}{P13}{P23}
\ncline[linestyle=solid,dotsep=.08]{->}{P14}{P24}
\ncline[linestyle=solid,dotsep=.08]{->}{P23}{P33}
\ncline[linestyle=solid,dotsep=.08]{->}{P24}{P34}
\ChrPath[myred](-2,-1.95)(-2,-1)(-1,-1)(-1,1)(0.95,1)
\end{pspicture}
\hskip1.0em
\begin{pspicture}(1.15,1.15)(-2.15,-1.85)
%
\psline[linewidth=0.02,linecolor=lightgray,linestyle=dashed]{-}(-2,-2)(1,1)
\cnode*[linewidth=0](-2,-2){.07}{P00}
\cnode*[linewidth=0](-2,-1){.07}{P01}
\cnode*[linewidth=0](-2,0){.07}{P02}
\cnode*[linewidth=0](-2,1){.07}{P03}
\cnode*[linewidth=0](-1,-1){.07}{P11}
\cnode*[linewidth=0](-1,0){.07}{P12}
\cnode*[linewidth=0](-1,1){.07}{P13}
\cnode*[linewidth=0](0,0){.07}{P22}
\cnode*[linewidth=0](0,1){.07}{P23}
\cnode*[linewidth=0](1,1){.07}{P33}
%
\ncline[linestyle=solid,dotsep=.08]{->}{P00}{P01}
\ncline[linestyle=solid,dotsep=.08]{->}{P11}{P12}
\ncline[linestyle=solid,dotsep=.08]{->}{P22}{P23}
\ncline[linestyle=solid,dotsep=.08]{->}{P01}{P02}
\ncline[linestyle=solid,dotsep=.08]{->}{P02}{P03}
\ncline[linestyle=solid,dotsep=.08]{->}{P12}{P13}
\ncline[linestyle=solid,dotsep=.08]{->}{P01}{P11}
\ncline[linestyle=solid,dotsep=.08]{->}{P02}{P12}
\ncline[linestyle=solid,dotsep=.08]{->}{P03}{P13}
\ncline[linestyle=solid,dotsep=.08]{->}{P12}{P22}
\ncline[linestyle=solid,dotsep=.08]{->}{P13}{P23}
\ncline[linestyle=solid,dotsep=.08]{->}{P14}{P24}
\ncline[linestyle=solid,dotsep=.08]{->}{P23}{P33}
\ncline[linestyle=solid,dotsep=.08]{->}{P24}{P34}
\ChrPath[myred](-2,-1.95)(-2,-1)(-1,-1)(-1,0)(0,0)(0,1)(0.95,1)
\end{pspicture}
\caption{There are five Dyck paths in $T_3$, so $\c_3=5$.}
\label{fig: catalan}
\end{figure}
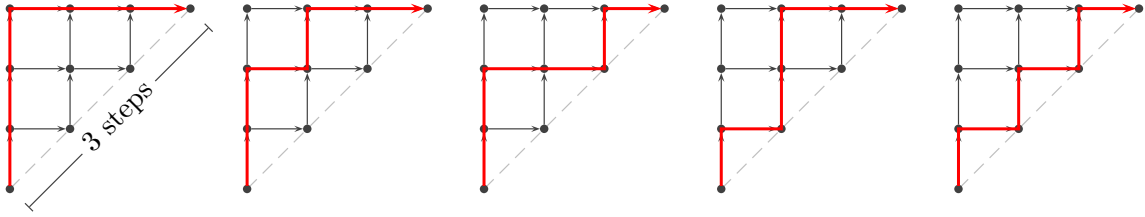

Let $G_n = T_{2n-1}$, where we arrange $n$ origin and $n$ destination vertices along the diagonal, as done in Section \ref{Middle binomial path counting}, it is plain to see that $A(G_n)$ is populated with Catalan numbers. Specifically, $A_{ij} = \c_{i+j}$. That is, $A(G_n) = \C(n)$. There is precisely one non-intersecting $n$-path with this configuration. See Figure \ref{fig: non-intersecting catalan-a}. Appealing to Theorem \ref{thm: path counting}, we conclude that $\det\C(n) = 1$.

\begin{remark}This result has been proven many times over. The path counting proof was given at least as early as \cite{Vien}, though it also appears in \cite{BCQY,MW}. A proof resting on the Cholesky decomposition of $\C$ has also been discovered many times. See \cite{Aig,Rad,Shap}. The latter approach is useful for determining an explicit description of $\C^{-1}$, which we do in Section \ref{inverting_M}. 
\end{remark} 

The fact that $\det\sC(n) = 1$ is now easy to see. Let $G'_n=T_{2n}$ where we again arrange the $n$ origin and destination vertices along the diagonal as in Section \ref{Middle binomial path counting}. Even though $o_0$ and $d_0$ no longer overlap, there is still no room for exotic non-intersecting $n$-paths. See Figure \ref{fig: non-intersecting catalan-b}. 
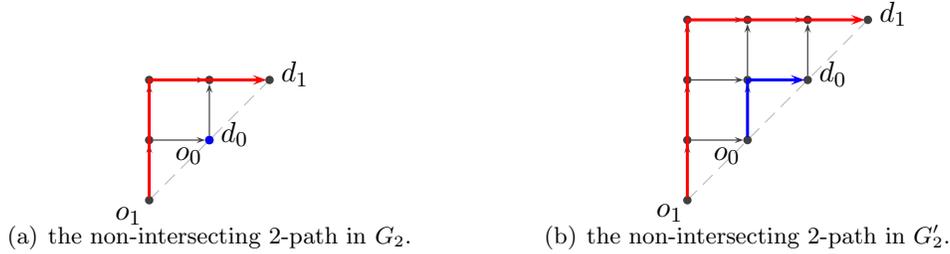
\begin{figure}[!hbt]
\subfigure[the non-intersecting $2$-path in $G_2$.]{%
\psset{unit=0.8, linewidth=.02,linecolor=darkgray,nodesep=0}
\begin{pspicture}(4.15,1.15)(-4.15,-1.15)
%
\psline[linewidth=0.02,linecolor=lightgray,linestyle=dashed]{-}(-1,-1)(1,1)
\cnode*[linewidth=0](-1,-1){.07}{P11}
\cnode*[linewidth=0](-1,0){.07}{P12}
\cnode*[linewidth=0](-1,1){.07}{P13}
\cnode*[linewidth=0](0,0){.07}{P22}
\cnode*[linewidth=0](0,1){.07}{P23}
\cnode*[linewidth=0](1,1){.07}{P33}
%
\ncline[linestyle=solid,dotsep=.08]{->}{P11}{P12}
\ncline[linestyle=solid,dotsep=.08]{->}{P22}{P23}
\ncline[linestyle=solid,dotsep=.08]{->}{P12}{P13}
\ncline[linestyle=solid,dotsep=.08]{->}{P12}{P22}
\ncline[linestyle=solid,dotsep=.08]{->}{P13}{P23}
\ncline[linestyle=solid,dotsep=.08]{->}{P14}{P24}
\ncline[linestyle=solid,dotsep=.08]{->}{P23}{P33}
\ncline[linestyle=solid,dotsep=.08]{->}{P24}{P34}
\uput{0.17}[225](0,0){$o_0$}
\uput{0.17}[225](-1,-1){$o_1$}
\uput{0.19}[15](0,0){$d_0$}
\uput{0.19}[15](1,1){$d_1$}
\pscircle*[linewidth=0.06,linecolor=myblue](0,0){.07}
\ChrPath[myred](-1,-0.95)(-1,1)(0.95,1)
\end{pspicture}\label{fig: non-intersecting catalan-a}}
\hskip1.0em
\subfigure[the non-intersecting $2$-path in $G'_2$.]{%
\psset{unit=0.8, linewidth=.02,linecolor=darkgray,nodesep=0}
\begin{pspicture}(3.15,1.15)(-5.15,-2.15)
%
\psline[linewidth=0.02,linecolor=lightgray,linestyle=dashed]{-}(-2,-2)(1,1)
\cnode*[linewidth=0](-2,-2){.07}{P00}
\cnode*[linewidth=0](-2,-1){.07}{P01}
\cnode*[linewidth=0](-2,0){.07}{P02}
\cnode*[linewidth=0](-2,1){.07}{P03}
\cnode*[linewidth=0](-1,-1){.07}{P11}
\cnode*[linewidth=0](-1,0){.07}{P12}
\cnode*[linewidth=0](-1,1){.07}{P13}
\cnode*[linewidth=0](0,0){.07}{P22}
\cnode*[linewidth=0](0,1){.07}{P23}
\cnode*[linewidth=0](1,1){.07}{P33}
%
\ncline[linestyle=solid,dotsep=.08]{->}{P00}{P01}
\ncline[linestyle=solid,dotsep=.08]{->}{P11}{P12}
\ncline[linestyle=solid,dotsep=.08]{->}{P22}{P23}
\ncline[linestyle=solid,dotsep=.08]{->}{P01}{P02}
\ncline[linestyle=solid,dotsep=.08]{->}{P02}{P03}
\ncline[linestyle=solid,dotsep=.08]{->}{P12}{P13}
\ncline[linestyle=solid,dotsep=.08]{->}{P01}{P11}
\ncline[linestyle=solid,dotsep=.08]{->}{P02}{P12}
\ncline[linestyle=solid,dotsep=.08]{->}{P03}{P13}
\ncline[linestyle=solid,dotsep=.08]{->}{P12}{P22}
\ncline[linestyle=solid,dotsep=.08]{->}{P13}{P23}
\ncline[linestyle=solid,dotsep=.08]{->}{P14}{P24}
\ncline[linestyle=solid,dotsep=.08]{->}{P23}{P33}
\ncline[linestyle=solid,dotsep=.08]{->}{P24}{P34}
\uput{0.17}[225](-1,-1){$o_0$}
\uput{0.10}[225](-2,-2){$o_1$}
\uput{0.19}[15](0,0){$d_0$}
\uput{0.19}[15](1,1){$d_1$}
\ChrPath[myblue](-1,-0.95)(-1,0)(-0.05,0)
\ChrPath[myred](-2,-1.95)(-2,1)(0.95,1)
\end{pspicture}\label{fig: non-intersecting catalan-b}}
\caption{Catalan $2$-paths.}
\label{fig: catalan 2-paths}
\end{figure}

%
%
\subsection{Path counting corollaries}\label{Path counting corollaries}
The path counting approach can be used for more than computing determinants. (Though the complexity of the applications below is certainly sub-optimal.)

\subsubsection{Inverse of $\M$}\label{sec: cramer}
In principle, the path counting approach can also be used to give an explicit description of $\M^{-1}$. 
(As remarked earlier, we follow an alternate approach in Section \ref{inverting_M}, debarking from the Cholesky decompositions of $\B,\sB,\C,\sC$.)

From \eqref{eq: tilde M} and the factorization $\tilde \M = P\, D_1\, \M\, D_2\, P^T$, the only difficulty in computing the inverse of $\M$ arises when computing the inverses of $\B,\sB,\C,\sC$. These may be computed using the path counting method by appealing to Cramer's Rule. Consider $\C=\C(n)$. 
\[
	\left(\C^{-1}\right)_{j,k} = (-1)^{j+k}\frac{\det \C^{k,j}}{\det \C},
\]
where $\C^{k,j}$ is the $\C$ with row $k$ and column $j$ deleted. The entries of $\C^{k,j}$ count paths in the same graph that governs $\C$. The new wrinkle is that one of the origins and destinations have been deleted, opening up the possibility for more non-intersecting paths. In each case however ($\B, \sB, \C, \sC$), the modified graph with $n-1$ origins and destinations is still non-permutable, so Theorem \ref{thm: path counting} applies. See Figure \ref{fig: inverting C} for two examples when $n=3$.
%
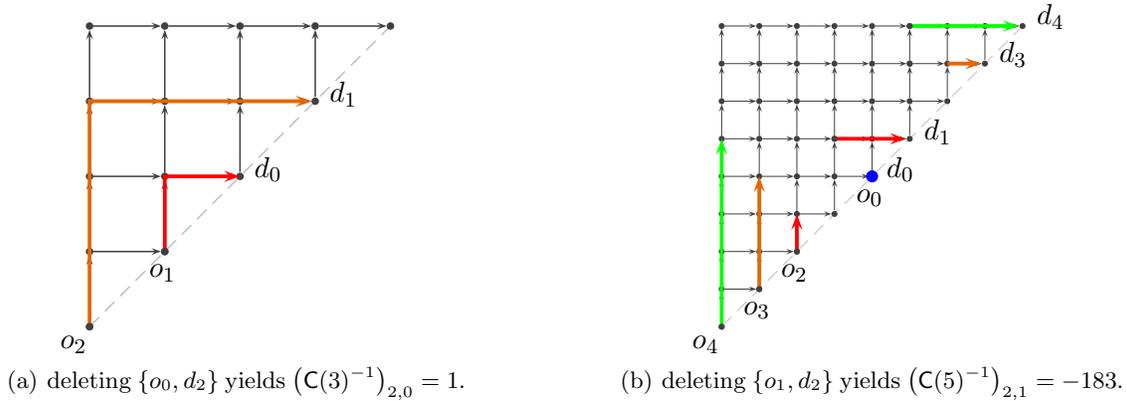
\begin{figure}[!hbt]
\subfigure[deleting $\{o_0,d_2\}$ yields $\left(\C(3)^{-1}\right)_{2,0}=1$.]{%
\psset{unit=1.0, linewidth=.02,linecolor=darkgray,nodesep=0}
\begin{pspicture}(3.0,2.15)(-3.1,-2.35)
%
\psline[linewidth=0.02,linecolor=lightgray,linestyle=dashed]{-}(-2,-2)(2,2)
\cnode*[linewidth=0](-2,-2){.05}{P00}
\cnode*[linewidth=0](-2,-1){.05}{P01}
\cnode*[linewidth=0](-2,0){.05}{P02}
\cnode*[linewidth=0](-2,1){.05}{P03}
\cnode*[linewidth=0](-2,2){.05}{P04}
\cnode*[linewidth=0](-1,-1){.05}{P11}
\cnode*[linewidth=0](-1,0){.05}{P12}
\cnode*[linewidth=0](-1,1){.05}{P13}
\cnode*[linewidth=0](-1,2){.05}{P14}
\cnode*[linewidth=0](0,0){.05}{P22}
\cnode*[linewidth=0](0,1){.05}{P23}
\cnode*[linewidth=0](0,2){.05}{P24}
\cnode*[linewidth=0](1,1){.05}{P33}
\cnode*[linewidth=0](1,2){.05}{P34}
\cnode*[linewidth=0](2,2){.05}{P44}
%
\ncline[linestyle=solid,dotsep=.08,nodesepA=0.01]{->}{P00}{P01}
\ncline[linestyle=solid,dotsep=.08,nodesepA=0.01]{->}{P11}{P12}
\ncline[linestyle=solid,dotsep=.08,nodesepA=0.01]{->}{P22}{P23}
\ncline[linestyle=solid,dotsep=.08,nodesepA=0.01]{->}{P33}{P34}
%
%
\ncline[linestyle=solid,dotsep=.08]{->}{P01}{P02}
\ncline[linestyle=solid,dotsep=.08]{->}{P02}{P03}
\ncline[linestyle=solid,dotsep=.08]{->}{P03}{P04}
\ncline[linestyle=solid,dotsep=.08]{->}{P12}{P13}
\ncline[linestyle=solid,dotsep=.08]{->}{P13}{P14}
\ncline[linestyle=solid,dotsep=.08]{->}{P23}{P24}
\ncline[linestyle=solid,dotsep=.08]{->}{P01}{P11}
\ncline[linestyle=solid,dotsep=.08]{->}{P02}{P12}
\ncline[linestyle=solid,dotsep=.08]{->}{P03}{P13}
\ncline[linestyle=solid,dotsep=.08]{->}{P04}{P14}
\ncline[linestyle=solid,dotsep=.08]{->}{P12}{P22}
\ncline[linestyle=solid,dotsep=.08]{->}{P13}{P23}
\ncline[linestyle=solid,dotsep=.08]{->}{P14}{P24}
\ncline[linestyle=solid,dotsep=.08]{->}{P23}{P33}
\ncline[linestyle=solid,dotsep=.08]{->}{P24}{P34}
\ncline[linestyle=solid,dotsep=.08]{->}{P34}{P44}
\uput{0.17}[265](-1,-1){$o_1$}
\uput{0.13}[235](-2,-2){$o_2$}
\uput{0.19}[15](0,0){$d_0$}
\uput{0.19}[15](1,1){$d_1$}
\ChrPath[myred](-1,-1)(-1,0)(0,0)
\ChrPath[myorange](-2,-2)(-2,-1)(-2,1)(0,1)(0.95,1)
\pscircle*[linewidth=0.05,linecolor=darkgray](-2,-2){.05}
\pscircle*[linewidth=0.05,linecolor=darkgray](-1,-1){.05}
\end{pspicture}\label{fig: inverting C-a}}
\hskip5.0em
\subfigure[deleting $\{o_1,d_2\}$ yields $\left(\C(5)^{-1}\right)_{2,1}=-183$.]{%
\psset{unit=1.0, linewidth=.01,linecolor=darkgray,nodesep=0}
\begin{pspicture}(3.35,2.15)(-3.35,-2.35)
%
\psline[linewidth=0.01,linecolor=lightgray,linestyle=dashed]{-}(-2,-2)(2,2)
\cnode*[linewidth=0](-2,-2){.04}{P00}
\cnode*[linewidth=0](-2,-1.5){.04}{P01}
\cnode*[linewidth=0](-2,-1){.04}{P02}
\cnode*[linewidth=0](-2,-.5){.04}{P03}
\cnode*[linewidth=0](-2,0){.04}{P04}
\cnode*[linewidth=0](-2,.5){.04}{P05}
\cnode*[linewidth=0](-2,1){.04}{P06}
\cnode*[linewidth=0](-2,1.5){.04}{P07}
\cnode*[linewidth=0](-2,2){.04}{P08}
\cnode*[linewidth=0](-1.5,-1.5){.04}{P11}
\cnode*[linewidth=0](-1.5,-1){.04}{P12}
\cnode*[linewidth=0](-1.5,-.5){.04}{P13}
\cnode*[linewidth=0](-1.5,0){.04}{P14}
\cnode*[linewidth=0](-1.5,.5){.04}{P15}
\cnode*[linewidth=0](-1.5,1){.04}{P16}
\cnode*[linewidth=0](-1.5,1.5){.04}{P17}
\cnode*[linewidth=0](-1.5,2){.04}{P18}
\cnode*[linewidth=0](-1,-1){.04}{P22}
\cnode*[linewidth=0](-1,-.5){.04}{P23}
\cnode*[linewidth=0](-1,0){.04}{P24}
\cnode*[linewidth=0](-1,.5){.04}{P25}
\cnode*[linewidth=0](-1,1){.04}{P26}
\cnode*[linewidth=0](-1,1.5){.04}{P27}
\cnode*[linewidth=0](-1,2){.04}{P28}
\cnode*[linewidth=0](-.5,-.5){.04}{P33}
\cnode*[linewidth=0](-.5,0){.04}{P34}
\cnode*[linewidth=0](-.5,.5){.04}{P35}
\cnode*[linewidth=0](-.5,1){.04}{P36}
\cnode*[linewidth=0](-.5,1.5){.04}{P37}
\cnode*[linewidth=0](-.5,2){.04}{P38}
\cnode*[linewidth=0](0,0){.04}{P44}
\cnode*[linewidth=0](0,.5){.04}{P45}
\cnode*[linewidth=0](0,1){.04}{P46}
\cnode*[linewidth=0](0,1.5){.04}{P47}
\cnode*[linewidth=0](0,2){.04}{P48}
\cnode*[linewidth=0](.5,.5){.04}{P55}
\cnode*[linewidth=0](.5,1){.04}{P56}
\cnode*[linewidth=0](.5,1.5){.04}{P57}
\cnode*[linewidth=0](.5,2){.04}{P58}
\cnode*[linewidth=0](1,1){.04}{P66}
\cnode*[linewidth=0](1,1.5){.04}{P67}
\cnode*[linewidth=0](1,2){.04}{P68}
\cnode*[linewidth=0](1.5,1.5){.04}{P77}
\cnode*[linewidth=0](1.5,2){.04}{P78}
\cnode*[linewidth=0](2,2){.04}{P88}
\ncline[linestyle=solid,dotsep=.08,nodesepA=0.01]{->}{P00}{P01}
\ncline[linestyle=solid,dotsep=.08]{->}{P01}{P02}
\ncline[linestyle=solid,dotsep=.08]{->}{P02}{P03}
\ncline[linestyle=solid,dotsep=.08]{->}{P03}{P04}
\ncline[linestyle=solid,dotsep=.08]{->}{P04}{P05}
\ncline[linestyle=solid,dotsep=.08]{->}{P05}{P06}
\ncline[linestyle=solid,dotsep=.08]{->}{P06}{P07}
\ncline[linestyle=solid,dotsep=.08]{->}{P07}{P08}
\ncline[linestyle=solid,dotsep=.08,nodesepA=0.01]{->}{P11}{P12}
\ncline[linestyle=solid,dotsep=.08]{->}{P12}{P13}
\ncline[linestyle=solid,dotsep=.08]{->}{P13}{P14}
\ncline[linestyle=solid,dotsep=.08]{->}{P14}{P15}
\ncline[linestyle=solid,dotsep=.08]{->}{P15}{P16}
\ncline[linestyle=solid,dotsep=.08]{->}{P16}{P17}
\ncline[linestyle=solid,dotsep=.08]{->}{P17}{P18}
\ncline[linestyle=solid,dotsep=.08,nodesepA=0.01]{->}{P22}{P23}
\ncline[linestyle=solid,dotsep=.08]{->}{P23}{P24}
\ncline[linestyle=solid,dotsep=.08]{->}{P24}{P25}
\ncline[linestyle=solid,dotsep=.08]{->}{P25}{P26}
\ncline[linestyle=solid,dotsep=.08]{->}{P26}{P27}
\ncline[linestyle=solid,dotsep=.08]{->}{P27}{P28}
\ncline[linestyle=solid,dotsep=.08,nodesepA=0.01]{->}{P33}{P34}
\ncline[linestyle=solid,dotsep=.08]{->}{P34}{P35}
\ncline[linestyle=solid,dotsep=.08]{->}{P35}{P36}
\ncline[linestyle=solid,dotsep=.08]{->}{P36}{P37}
\ncline[linestyle=solid,dotsep=.08]{->}{P37}{P38}
\ncline[linestyle=solid,dotsep=.08,nodesepA=0.01]{->}{P44}{P45}
\ncline[linestyle=solid,dotsep=.08]{->}{P45}{P46}
\ncline[linestyle=solid,dotsep=.08]{->}{P46}{P47}
\ncline[linestyle=solid,dotsep=.08]{->}{P47}{P48}
\ncline[linestyle=solid,dotsep=.08,nodesepA=0.01]{->}{P55}{P56}
\ncline[linestyle=solid,dotsep=.08]{->}{P56}{P57}
\ncline[linestyle=solid,dotsep=.08]{->}{P57}{P58}
\ncline[linestyle=solid,dotsep=.08,nodesepA=0.01]{->}{P66}{P67}
\ncline[linestyle=solid,dotsep=.08]{->}{P67}{P68}
\ncline[linestyle=solid,dotsep=.08,nodesepA=0.01]{->}{P77}{P78}
\ncline[linestyle=solid,dotsep=.08]{->}{P01}{P11}
\ncline[linestyle=solid,dotsep=.08]{->}{P02}{P12}
\ncline[linestyle=solid,dotsep=.08]{->}{P03}{P13}
\ncline[linestyle=solid,dotsep=.08]{->}{P04}{P14}
\ncline[linestyle=solid,dotsep=.08]{->}{P05}{P15}
\ncline[linestyle=solid,dotsep=.08]{->}{P06}{P16}
\ncline[linestyle=solid,dotsep=.08]{->}{P07}{P17}
\ncline[linestyle=solid,dotsep=.08]{->}{P08}{P18}
\ncline[linestyle=solid,dotsep=.08]{->}{P12}{P22}
\ncline[linestyle=solid,dotsep=.08]{->}{P13}{P23}
\ncline[linestyle=solid,dotsep=.08]{->}{P14}{P24}
\ncline[linestyle=solid,dotsep=.08]{->}{P15}{P25}
\ncline[linestyle=solid,dotsep=.08]{->}{P16}{P26}
\ncline[linestyle=solid,dotsep=.08]{->}{P17}{P27}
\ncline[linestyle=solid,dotsep=.08]{->}{P18}{P28}
\ncline[linestyle=solid,dotsep=.08]{->}{P23}{P33}
\ncline[linestyle=solid,dotsep=.08]{->}{P24}{P34}
\ncline[linestyle=solid,dotsep=.08]{->}{P25}{P35}
\ncline[linestyle=solid,dotsep=.08]{->}{P26}{P36}
\ncline[linestyle=solid,dotsep=.08]{->}{P27}{P37}
\ncline[linestyle=solid,dotsep=.08]{->}{P28}{P38}
\ncline[linestyle=solid,dotsep=.08]{->}{P34}{P44}
\ncline[linestyle=solid,dotsep=.08]{->}{P35}{P45}
\ncline[linestyle=solid,dotsep=.08]{->}{P36}{P46}
\ncline[linestyle=solid,dotsep=.08]{->}{P37}{P47}
\ncline[linestyle=solid,dotsep=.08]{->}{P38}{P48}
\ncline[linestyle=solid,dotsep=.08]{->}{P45}{P55}
\ncline[linestyle=solid,dotsep=.08]{->}{P46}{P56}
\ncline[linestyle=solid,dotsep=.08]{->}{P47}{P57}
\ncline[linestyle=solid,dotsep=.08]{->}{P48}{P58}
\ncline[linestyle=solid,dotsep=.08]{->}{P56}{P66}
\ncline[linestyle=solid,dotsep=.08]{->}{P57}{P67}
\ncline[linestyle=solid,dotsep=.08]{->}{P58}{P68}
\ncline[linestyle=solid,dotsep=.08]{->}{P67}{P77}
\ncline[linestyle=solid,dotsep=.08]{->}{P68}{P78}
\ncline[linestyle=solid,dotsep=.08]{->}{P78}{P88}
\uput{0.17}[265](0,0){$o_0$}
\uput{0.17}[265](-1,-1){$o_2$}
\uput{0.17}[265](-1.5,-1.5){$o_3$}
\uput{0.13}[235](-2,-2){$o_4$}
\uput{0.19}[15](0,0){$d_0$}
\uput{0.19}[15](.5,.5){$d_1$}
\uput{0.19}[20](1.5,1.5){$d_3$}
\uput{0.19}[20](2,2){$d_4$}
\pscircle*[linewidth=0.05,linecolor=myblue](0,0){.08}
\ChrPath[myred](-1,-1)(-1,-.5)  \ChrPath[myred](-.5,.5)(.45,.5)
\ChrPath[myorange](-1.5,-1.5)(-1.5,0)  \ChrPath[myorange](1,1.5)(1.45,1.5)
\ChrPath[mygreen](-2,-2)(-2,.5)  \ChrPath[mygreen](.5,2)(1.95,2)
\pscircle*[linewidth=0.05,linecolor=darkgray](-2,-2){.04}
\pscircle*[linewidth=0.05,linecolor=darkgray](-1.5,-1.5){.04}
\pscircle*[linewidth=0.05,linecolor=darkgray](-1,-1){.04}
\pscircle*[linewidth=0.05,linecolor=darkgray](-.5,.5){.04}
\pscircle*[linewidth=0.05,linecolor=darkgray](.5,1.5){.04}
\pscircle*[linewidth=0.05,linecolor=darkgray](.5,2){.04}
\end{pspicture}\label{fig: inverting C-b}}
\caption{Inverting $\C$.}\label{fig: inverting C}
\end{figure}

Unfortunately, the combinatorics involved in enumerating the non-intersecting paths in this new situation can be substantially more difficult. (In Figure \ref{fig: inverting C-b}, we draw the beginning and ending portions of the paths that are forced. The reader may verify that there are 183 ways to connect the paths.) A direct and explicit computation of the non-intersecting paths in such graphs---even in the well-studied case of Dyck paths---remains elusive. (The authors in \cite{BCQY,MW} essentially give up after considering the cases $j,k\in \{0,1\}$.) See \cite[Chap\^itre 4]{Vien} and \cite{SCV,GV} for further steps in this direction. 

\subsubsection{$LDU$ factorization of $\M$}
We recount a seldom used fact from matrix theory. Suppose an invertible $n\times n$ matrix $A=(a_{ij})$ has a factorization $A=LDU$, with $L$ and $U$ lower- and upper-triangular, respectively, with $1$s along the diagonal, and $D$ a diagonal matrix. Given tuples of indices $\alpha,\beta\subseteq \{0,1,\dotsc,n{-}1\}$, let $A_{\alpha,\beta} = (a_{ij})_{i\in\alpha,j\in\beta}$. Also, let $[k] := (0,1,\dotsc,k)$ and $[k]+j := (0,\dotsc,k,j)$. Then we have
\[
	L_{ij} = \frac{\det A_{[j-1]+i,[j]}}{\det A_{[j],[j]}},
	\quad
	D_{jj} = \frac{\det A_{[j],[j]}}{\det A_{[j-1],[j-1]}}
	\qand
	U_{jk} = \frac{\det A_{[j],[j-1]+k}}{\det A_{[j],[j]}}. 
\]
That is, the entries of $L, D$ and $U$ are ratios of minors. So if these again correspond to non-permutable directed graphs, Theorem  \ref{thm: path counting} applies.
This is certainly the case for our $\B,\sB,\C,\sC$. In contrast to Paragraph \ref{sec: cramer}, the resulting counting problems are even manageable here. 

\subsubsection{Totally positive matrices}\label{sec: tpm}
(A prelude to Section \ref{sec: concluding}.) 
There is a generalization of Theorem \ref{thm: path counting} to graphs with edges labeled by nonnegative real numbers, where paths are counted with \demph{weights} (the product of the edge labels on the path).\footnote{In fact, this more general result follows from the proof of Theorem \ref{thm: path counting} sketched above.} A fact dating back to \cite{Bre} has come into fashion recently, with the advent of cluster theory \cite{FomZel}. 

\begin{theorem}\label{thm: TPM} 
An $n\times n$ matrix $A=(a_{ij})$ is {totally nonnegative} if and only if there exists a planar, edge-labeled, directed, acyclic graph $G$, with fully non-permutable choice of $n$ origins and destinations, so that $a_{ij}$ equals the weighted sum of paths from $i$ to $j$ in $G$.
\end{theorem}

Recall that a matrix is \demph{totally positive} (\demph{nonnegative}) if all of its minors $\det A_{\alpha,\beta}$ are positve (nonnegative). (Here $\alpha,\beta$ are thought of as subsequences of $(0,1,\dotsc,n{-}1)$.) By \emph{fully non-permutable} we mean that for every subsquence $(i_1,i_2,\dotsc,i_r)$ of $(0,1,\ldots,n-1)$ (with $r\leq n$), the origin and destination choices $\{o_{i_1},\dotsc,o_{i_r}\}$ and $\{d_{i_1},\dotsc,d_{i_r}\}$ are non-permutable.

An immediate corollary is that our matrices $\B,\sB,\C,\sC$ are totally positive, and that $\tilde\M$ is totally nonnegative. (This is not the case for $D_1\M D_2$, however.)

\begin{remark} 
If $A$ is Hankel, then a result of 
\cite[Theorem 4.4]{Pinkus} says that for total positivity (nonnegativity), it is enough to check that $A$ and its lower-left $(n{-}1)\times (n{-}1)$ submatrix are both positive definite (semidefinite) matrices. This can be checked for our binomial and catalan matrices as well. We revisit this approach in Section \ref{sec: alternative boundary conditions}.
\end{remark}

%
%
\section{Boundary Layer Potential Solution}\label{Invertibility of the boundary value operator}
Consider the operator $\V$ as a product
$\V = A_1 \M A_2$, where 
$ \M$ is the matrix of Proposition \ref{scalar_symbol}, while
$A_1$ and $A_2$ are both multiplier operators on $\reals^{d-1}$, possessing symbols
$\left(\sigma(A_1) (\eta) \right)_{j,k} =  (-1)^{(k-m)/2}|\eta|^{k-m}\delta_{j,k}$ and
$\left(\sigma(A_2) (\eta) \right)_{j,k} = (-1)^{(j-m)/2}|\eta|^{1+j-m}\delta_{j,k}$. 
 (Here $\delta_{j,k}$ is the Kronecker delta.) 
 Since $A_1$, $A_2$ have diagonal symbols, they are easily inverted,
and the inverse of each is a ``decoupled'' operator.
To illustrate this, we consider, respectively, the $j$th entries of 
$A_1^{-1}\mathbf{f}$ and $A_2^{-1}\mathbf{f}$, where $\mathbf{f} = (f_0,\dots,f_{m-1})$
is a vector of test functions:
\begin{align*}
 \left(A_1^{-1} \mathbf{f}
 \right)_k &= 
 \begin{cases}
(\Deltad)^{n} f_k, & \text{for }2n={m-k},
\\
-i\sqrt{-\Deltad} \,(\Deltad)^{n} f_k, & \text{for }2n + 1 = {m-k}
 \end{cases}\\
&\text{and}\\
 \left(A_2^{-1} \mathbf{f}
 \right)_j &= 
 \begin{cases}
i(\Deltad)^{n} f_j, & \text{for }2n+1={m-j},
\\
\sqrt{-\Deltad}\, (\Deltad)^{n} f_j, & \text{for }2n = {m-j}.
 \end{cases}
\end{align*}
In the `odd' cases, we have employed the operator $ \sqrt{-\Deltad}=  \mathbf{R} \dotp \nabla$, where
 $\mathbf{R} = \nabla (-\Deltad)^{-1/2} =(\Riesz_\ell)_{\ell=1}^{d-1}$ is  the ensemble of Riesz transforms, $\Riesz_{\ell}$,
 defined via the Fourier symbol as 
  $\widehat{\Riesz_{\ell} f}(\eta) = i \frac{\eta_{\ell}}{|\eta|}\widehat{f}(\eta)$.

It follows that as a (formal) solution of the system of boundary integral
equations, we should take
$\mathbf{g} = A_2^{-1} \M^{-1} A_1^{-1}\mathbf{h}$. 
In the previous section it was observed that $\M^{-1}$ exists and has the same checkerboard
structure as $\M$. It follows that the solution $\bfg$ only involves  pairs of  indices $j,k$
between $0\le j,k\le m-1$ for which both are odd or even. Consequently, the
operator $\sqrt{-\Deltad}$ is a common factor throughout the solution.

In other words, we have auxiliary functions of the form
\begin{equation}\label{Bdry_Function}
g_j := 
\sqrt{-\Deltad} \sum_{\substack{k=0\\ k+j \text{ is even}}}^{m-1} (\M^{-1})_{j,k} \Deltad^{m-1-(j+k)/2} h_k.
\end{equation}

We note that functions $g_j$ are well-defined for boundary data $\mathbf{h}$ 
satisfying $h_k\in C^{2m-1-k+\epsilon}(\reals^{d-1})$. Indeed, each $g_j\in C^{j}(\reals^{d-1})$.
 What is perhaps less obvious is that one can ensure decay of the functions $g_j$ by imposing
 decay and moment conditions on the functions $\bfh$.
By assuming, roughly, that $h_k\in \Space_{k,4}^{2m-1-k+\epsilon}$ and 
$h_k\perp \Pi_{k+1-d}$, we have $g_j \in  \Space_{2m-j-1,2}^{j}$ and $g_j\perp \Pi_{2m -j-d}$.
This brings us to our main theorem.

%

\begin{theorem} \label{main}
Let the functions $\mathbf{h} = (h_k)$ satisfy the following three conditions:
\begin{enumerate}
\item there is $\epsilon>0$ so that each $h_k\in C^{2m-k-1+\epsilon}(\reals^{d-1})$;
\item there  are constants $\delta>4$ and $C>0$ so that for 
each $k =0,\dots, m-1$,   and for all $s\in[0, 2m-k-1+\epsilon]$
and $\rho>0$, 
\[
	|h_k|_{C^{s}(\reals^d\setminus B(0,\rho))} \le C (1+\rho)^{-(k+s)} \bigl(\log (e+ \rho)\bigr)^{-\delta}\text{;}
\]
\item for all $0\leq j<\min(m,2m-d)$,
\begin{equation}\label{weak_conditions}
\left(
	\sum_{k=0}^{m-1} 
		(\M^{-1})_{j,k} (\Deltad)^{m-1-(j+k)/2} h_k
\right)
\perp  
\Pi_{2m -j-d-1}(\reals^{d-1}).
\end{equation}
\end{enumerate}
If $\bfg = (g_j)$ is defined as in \eqref{Bdry_Function}, then the function $u=T\mathbf{g}=\sum_{j=0}^{m-1} V_j g_j$ as defined in \eqref{3:BhRep} 
is the solution of the polyharmonic Dirichlet problem \eqref{Dirichlet}.
\end{theorem}
\begin{proof} 
We proceed by demonstrating that the functions $g_0,\dots,g_{m-1}$ satisfy the three conditions introduced in Section \ref{Boundary Layer Potential Operators}.

\smallskip
\noindent{\em Condition \ref{firstDecay} (Decay):}  
We consider  each function $F_{\ell,j,k} := \frac{\partial}{\partial x_{\ell}} (\Deltad)^{m-1 -(j+k)/2} h_k$, because each $g_j $ is a linear combination of functions $\Riesz_{\ell} F_{\ell,j,k}$. 

Condition (1) ensures that each 
$ F_{\ell,j,k}\in C^{j+\epsilon}(\reals^{d-1})$.  
Condition (2) states that
$h_{k}\in \Space_{k, \delta}^{2m-k-1+\epsilon}$, which 
guarantees that $ F_{\ell,j,k}\in \Space_{2m-j-1, \delta}^{j+\epsilon}$.
Therefore,  for $|\alpha|\le j$, the estimate
$|D^{\alpha}
F_{\ell,j,k}(y)| \le C (1+|y|)^{1+j-2m-|\alpha|}\bigl( \log(e+|y|)\bigr)^{-\delta} $
holds. We note that, because $m>d/2$, each $F_{\ell,j,k}\in W_2^j(\reals^{d-1})$.

Because the Riesz transform from $C^{j+\epsilon}(\reals^{d-1})\cap W_2^j(\reals^{d-1})$ to  $C^{j}(\reals^{d-1})$ is bounded---this is demonstrated in \eqref{riesz_bound} below---we have that 
\[
	g_j=\sum_{\ell=1}^{d-1} \Riesz_{\ell} \sum_{k=0}^{m-1} (\M^{-1})_{j,k}  F_{\ell,j,k}
	\in C^j(\reals^{d-1}).
\]
  
Condition (3)
indicates that for fixed $j$ and $\ell$, the function
$\sum_{k=0}^{m-1} (\M^{-1})_{j,k} F_{\ell,j,k}$ annihilates polynomials of degree ${2m-j-d}$ (with $\Pi_{J}=\{0\}$ if $J<0$).
We are now in a position to apply Lemma \ref{RTE} (a basic result about decay of Riesz transforms---proved in Section \ref{Decay of Riesz transforms}), 
with $J = 2m-j -d$ and $N= d-1$.
We observe that for each $j$  and $|\alpha|\le j$,
\begin{equation}\label{decay_proof}
\left|D^{\alpha} g_j(y)\right|
\le 
\sum_{\ell=1}^{d-1}
\left|
    D^{\alpha}\Riesz_{\ell} 
      \sum_{k=0}^{m-1} 
        (\M^{-1})_{j,k} F_{\ell,j,k}(y)
\right|
\le C(1+|y|)^{j+1-2m-|\alpha|}\bigl(\log(e+|y|)\bigr)^{2-\delta}
\end{equation}
 holds 
and Condition  \ref{firstDecay} is satisfied.

At this point, we note that  Lemma \ref{boundary_values} applies;  $T\bfg$ is $m$-fold polyharmonic in $\reals^d_{+}$, and, moreover, is sufficiently smooth that the boundary values $\lambda_k T\bfg$ are well-defined for $k=0\dots m-1$.

\smallskip
\noindent{\em Condition \ref{fourierSmoothness} (Smoothness):} 
The decay of $g_j$ and its derivatives demonstrated in  \eqref{decay_proof} shows that 
$g_j\in W_2^j(\reals^{d-1})$.
Using Cauchy-Schwarz, we get the inequalities 
$$
\int_{\reals^{d-1}} |\widehat{g}_j(\eta)| (1+|\eta|)^{j-m} \dif \eta 
\le 
C
\left(
  \int_{\reals^{d-1}} 
    \left|
      \widehat{g}_j(\eta) (1+|\eta|)^{j-(m-d/2)} 
    \right|^2
  \dif \eta 
\right)^{1/2} 
\le 
C \|g_j\|_{W_2^{j}(\reals^{d-1})}.
$$
From this, Condition \ref{fourierSmoothness} follows.

\smallskip
\noindent{\em Condition \ref{fourierMoments} (Moments):}  Take $0\le j\le \min(m-1,2m-d)$.
If $J:=2m-d-j\ge0$, then the decay of $g_j$ demonstrated in \eqref{decay_proof} 
ensures that $g_j$ and $|\cdot|^{J} g_j $ are in $L_1(\reals^{d-1})$ 
and that  $\widehat{g_j}\in C^{J}(\reals^{d-1})$. 
The orthogonality condition \eqref{weak_conditions} ensures that $g_j\perp \Pi_{J}$
which implies that $D^{\alpha} \widehat{g_j} (0) = 0$. 
At this point we observe that $|\widehat{g_j}(\xi')| =o( |\xi|^J)$.  

Because $\delta>4$, 
it is not difficult to show that,
for $|\alpha| = J $, 
the functions $D^{\alpha} \widehat{g}_j$ are continuous with $\log$-H{\" o}lder
continuity:
the modulus of continuity of  $D^{\alpha}\widehat{g_j}$ satisfies $\omega(D^{\alpha}\widehat{g_j},h) \le |\log h |^{3-\delta}$.
It follows that
$\widehat{g}_j (\xi')  
\le C|\xi'|^{J}\bigl|\log|\xi'|\bigr|^{-\tau} $ with $\tau = \delta-3$ for $\xi'$ in a neighborhood of $0$.
\end{proof}

\section{Inverting the matrix symbol $\M$}\label{inverting_M}

We have seen that $\M$ is real symmetric positive definite, and hence has a \demph{Cholesky decomposition} $\M = \mathsf{L}\cdot\mathsf{L}^{T}$, where $\mathsf{L}$ is lower triangular. In this section, we extend the work of Radoux and Shapiro to give formulas for the entries of $\M^{-1}$ by computing $\mathsf{L}^{-1}$ explicitly. 
The following observations about matrix inversion will be useful in what follows and are stated without proof.

\begin{lemma}\label{lem: matrix arithmetic}
Let $X$ be the $n\times n$ matrix with $1$s on the first subdiagonal and zeros elsewhere. Let $L = \id + X$ and $K = \id-X$. We have:
\begin{enumerate}
\item\label{itm: L-1} 
$X^k$ has $1$s along the $k$-th subdiagonal and zeros elsewhere. In particular, $X$ is nilpotent of index $n$, i.e., $X^k\neq0$ for $k<n$ and $X^n=0$.

\item\label{itm: L-2} 
The inverse for $L$ is given by the formula $L^{-1} = \id -X+X^2-X^3+\dotsb+X^{n-1}$.
In particular, $L^{-1}$ is a lower-triangular checkerboard matrix of $+1$s and $-1$s, with $+1$s along the diagonal. 

\item\label{itm: L-3} 
The inverse of $K$ is given by the formula $K^{-1} = \id +X+X^2+\dotsb+X^{n-1}$. In particular,  $K^{-1}$ is a lower-triangular matrix of $1$s. 

\end{enumerate}
\end{lemma}

%
\subsection{Inverting Catalan Hankel matrices}\label{sec: C-inv}

In \cite{Rad}, Radoux produces an explicit description of the Cholesky decomposition of $\C(n)$ by introducing what we call here \emph{semi-Catalan} numbers 
\[
    \c_{n,k} := \binom{2n}{n+k} - \binom{2n}{n+k+1} .
\]
The semi-Catalan numbers satisfy a number of algebraic identities. We list a sampling from \cite{Rad} that will be useful in what follows:
\begin{gather}
\label{eq: semicat1}
  \c_{n,0} = \c_n,\quad \c_{n,n} = 1 \qand \c_{n,k} = 0 \hbox{ for }k\not\in[0,n], \\
\label{eq: semicat2}
  \sum_{k\geq0} (-1)^k \c_{n,k} = 0, \\
\label{eq: semicat3}
  \sum_{k\geq0} \c_{n,k} = \binom{2n}{n}.
\end{gather}
We also need a few standard binomial coefficient facts, which we call the \emph{factorial}, \emph{Pascal} and \emph{(Pascal-) hook} identities, respectively:
\begin{equation}\label{eq: binom-identities}
	\binom{p}{q} = \frac{p}{q}\binom{p-1}{q-1}, 
	\quad 
	\binom{p}{q} = \binom{p-1}{q} + \binom{p-1}{q-1}
	\qand
	\binom{p+1}{q} = \sum_{i=0}^q \binom{p-q+i}{i}.
\end{equation}
Let $\CL(n) = (\c_{j,k})_{0\leq j,k< n}$ be the lower-triangular matrix of semi-Catalan numbers. For example,
\[
	\CL(3) = \begin{pmatrix}
	1 \\
	1 & 1 \\
	2 & 3 & 1 \\
	5 & 9 & 5 & 1 
	\end{pmatrix} .
\]
Radoux shows that $\C(n) = \CL(n)\cdot \CL(n)^T$ and finds formulas for the entries of $\CL(n)^{-1}$ and $\C(n)^{-1}$.

\begin{proposition}\label{thm: C-inv}
For $0\leq j,k<n$, we have
\begin{align}
\label{eq: S-inv}
	\left(\CL(n)^{-1}\right)_{j,k} & = (-1)^{j+k}\binom{j+k}{j-k}\,, \\[.5ex]
\label{eq: C-inv}
	\left(\C(n)^{-1}\right)_{j,k} &= (-1)^{j+k} \sum_{i=\mathrm{max}\{j,k\}}^{n-1} \binom{i+j}{i-j}\binom{i+k}{i-k} \,.
\end{align}
\end{proposition}

Similar results for $\sC$, $\B$ and $\sB$ follow from Proposition \ref{thm: C-inv} and Lemma \ref{lem: matrix arithmetic}, as we now see.

%
\subsection{Inverting shifted Catalan matrices}\label{sec: sC-inv}

In \cite{Shap}, Shapiro finds the Cholesky decomposition of $\sC(n)$. He defines what we call here \emph{shifted semi-Catalan} numbers 
\[
    \c'_{n,k} := \frac{k+1}{n+1}\binom{2n+2}{n-k} ,
\]
and proves that the lower triangular matrix $\sCL(n) = (\c'_{j,k})_{0\leq j,k<n}$ satisfies $\sC(n) = \sCL(n)\cdot \sCL(n)^T$.
He does not develop shifted equivalents of \eqref{eq: S-inv} and \eqref{eq: C-inv}, so we do this here. 
We need two lemmas.

\begin{lemma}\label{thm: sS is sum} 
For all $0\leq k\leq n$, the shifted semi-Catalan numbers satisfy 
$
    \c'_{n,k} = \c_{n,k} + \c_{n,k+1}.
$
\end{lemma}

\begin{proof}
We use the identity 
\begin{equation}\label{eq: binom-identity}
    \binom{n}{k} = \frac{n}{n-2k} \left[\binom{n-1}{k} - \binom{n-1}{k-1}\right],
\end{equation}
which may be deduced from the factorial and Pascal identities \eqref{eq: binom-identities}. We have
\begin{align*}
    \c'_{n,k} &= \frac{k+1}{n+1}\binom{2n+2}{n-k} \\
	&= \frac{k+1}{n+1}\frac{2n+2}{2n+2-2(n-k)}\left[\binom{2n+1}{n-k} - \binom{2n+1}{n-k-1}\right] \\[.5ex]
	&= 1\cdot\left[\binom{2n}{n-k} + \binom{2n}{n-k-1} - \binom{2n}{n-k-1} - \binom{2n}{n-k-2}\right] \\[.5ex]
	&= \binom{2n}{n+k} - \binom{2n}{n+k+1} + \binom{2n}{n+k+1} - \binom{2n}{n+k+2} 
	= \c_{n,k} + \c_{n,k+1}\,.\qedhere 
\end{align*}
\end{proof}

Thus we may write $\sCL(n) = \CL(n)\cdot (\id + X)$, where $X$ is the matrix with $1$s along the first subdiagonal and zeros elsewhere. As a consequence, we get the following formulas. 

\begin{corollary}\label{thm: sC-inv}
For $0\leq j,k<n$, we have
\begin{align}
\label{eq: sS-inv}
	\bigl(\sCL(n)^{-1}\bigr)_{j,k} &= (-1)^{j+k}\binom{j+k+1}{j-k} \\[.5ex]
\label{eq: sC-inv}
	(\sC(n)^{-1})_{jk} &= (-1)^{j+k} \sum_{i=\mathrm{max}\{j,k\}}^{n-1} \binom{i+j+1}{i-j}\binom{i+k+1}{i-k} .
\end{align}
\end{corollary}

\begin{proof}
Write $\sCL(n)^{-1} = (\id+X)^{-1}\CL(n)^{-1}$. Using Lemma \ref{lem: matrix arithmetic}\eqref{itm: L-2} and \eqref{eq: S-inv}, we see that the $(j,k)$ entry of $\sCL(n)^{-1}$ takes the form
\[
	\sum_{i=0}^{n-1} (-1)^{j+i}(-1)^{i+k}\binom{i+k}{i-k}
	 = (-1)^{j+k} \sum_{i=k}^j \binom{i+k}{i-k}
	 = (-1)^{j+k} \sum_{i=0}^{j-k} \binom{2k+i}{i}.
\]
Now use the hook identity \eqref{eq: binom-identities} with $(p,q)=(j+k,j-k)$ to get the advertised quantity in \eqref{eq: sS-inv}. 

Now \eqref{eq: sC-inv} follows by expanding $\sC(n)^{-1} = \bigl(\sCL(n)\cdot \sCL(n)^T\bigr)^{-1}$ using \eqref{eq: sS-inv}.
\end{proof}

%
\subsection{Inverting middle binomial Hankel matrices}

In \cite{Rad}, Radoux also finds the Cholesky decomposition of $\B(n)$. More precisely, he builds polynomial generalizations of our $\B(n)$ and decomposes these. Briefly, he defines a family of polynomials 
\[
    \b_{n,k}(x) := \c_{n,k} + \c_{n,k+1}x + \dotsb + \c_{n,n} x^{n-k} \,,
\]
and populates a lower-triangular matrix with them. For our purposes, we may simply set $x=1$ above and speak about the matrix $\BL$ of \emph{semi-binomial} numbers\footnote{Note that $\b_{n,0}(1) = \b_{n}=\binom{2n}{n}$, cf. \eqref{eq: semicat3}, which justifies our choice of name.} $\b_{i,j} := \b_{i,j}(1)$. For example, 
\[
	\BL(4) = 
	\left(\begin{array}{@{\,}llll@{\,}}
	1 \\
	2 & 1 \\
	6 & 4 & 1 \\
	20 & 15 & 6 & 1 
	\end{array}\right) .
\]
These $\BL(n) = \bigl(\b_{j,k}\bigr)_{0\leq j,k<n}$ and the diagonal matrices $D(n) = \diag{1,2,\dotsc,2}$ give the desired Cholesky decomposition, 
$
    \B(n) = \BL(n) \cdot D(n) \cdot \BL(n)^T.
$

Formulas for the inverse of $\BL$, and hence $\B$, may thus be deduced from those for $\C$ and $\CL$, as we now recount. 
We observe a simple identity that will be useful in what follows.
\begin{gather}\label{eq: b recursion}
	\b_{n,k} = \c_{n,k} + \b_{n,k+1}.
\end{gather}

\begin{proposition}\label{thm: B-inv}
For $0\leq j,k<n$, we have 
\begin{align}
\label{eq: R-inv}
	\bigl(\BL(n)^{-1}\bigr)_{j,k} & = (-1)^{j+k}\left[2\binom{j+k}{j-k} - \binom{j+k-1}{j-k} \right]\,, \\[.5ex]
\label{eq: B-inv}
\bigl(\B(n)^{-1}\bigr)_{j,k} &= \frac{\delta_{j0}\delta_{k0}}{2} + \frac{(-1)^{j+k}}{2} \!\! \sum_{i=\mathrm{max}\{j,k\}}^{n-1}  
	\left[2\binom{i+j}{i-j} - \binom{i+j-1}{i-j}\right]
	\left[2\binom{i+k}{i-k} - \binom{i+k-1}{i-k}\right] .
\end{align}
\end{proposition}

\begin{proof}
First, note by \eqref{eq: b recursion} that 
\[
	\BL(n) = \CL(n)\left(\id+X+X^2+\cdots+X^{n-1}\right),
\]
where $X$ is the matrix having $1$s along the first subdiagonal and zeros elsewhere. Thus, we may use Lemma \ref{lem: matrix arithmetic}\eqref{itm: L-3} to write the $(j,k)$ entry of $\BL(n)^{-1}=\BL(n;1)^{-1} = (\id-X)\CL(n)^{-1}$ as
\[
	\bigl(\CL(n)^{-1}\bigr)_{j,k} - \bigl(\CL(n)^{-1}\bigr)_{j-1,k} 
	= (-1)^{j+k}\left[\binom{j+k-1}{j-k-1} + \binom{j+k}{j-k} \right] .
\]
Finally, we use Pascal's identity \eqref{eq: binom-identities} to rewrite this formula as it appears in \eqref{eq: R-inv}.\footnote{Why do we choose to write $2\binom{j+k}{j-k} - \binom{j+k-1}{j-k}$ instead of $\binom{j+k-1}{j-k-1} + \binom{j+k}{j-k}$? Consider the case $j=k=0$. There is some disagreement, e.g., among computer algebra packages, about the value of $\binom{-1}{-1}+\binom{0}{0}$. On the other hand, the value of $2\binom{-1}{0}-\binom{0}{0}$ is considerably less contentious.
}

We next turn to $\B(n)$. The specific form of \eqref{eq: B-inv} combines \eqref{eq: R-inv}, the Cholesky decomposition and the fact that $D(n)^{-1} = \diag{1,\frac12, \frac12,\dotsc,\frac12}$. Simply observe that the general formula 
\[
(\B(n)^{-1})_{j,k} = (-1)^{j+k} \sum_{i=\mathrm{max}\{j,k\}}^{n-1} \frac1{2} 
	\left[2\binom{i+j}{i-j} - \binom{i+j-1}{i-j}\right]
	\left[2\binom{i+k}{i-k} - \binom{i+k-1}{i-k}\right]
\]
holds for $(j,k)\neq(0,0)$. In case $(j,k)=(0,0)$, the factor $\frac1{2}$ for the summand $i=0$ should be replaced by $1$. As the quantity
$\left[2\binom{i+j}{i-j}-\binom{i+j-1}{i-j}\right]
	\left[2\binom{i+k}{i-k}-\binom{i+k-1}{i-k}\right]=\left[2\binom{0}{0}-\binom{-1}{0}\right]
	\left[2\binom{0}{0}+\binom{-1}{0}\right]$ equals 1 in this case, this is accomplished by adding an extra $\frac12$ to the sum. 
\end{proof}

%
\subsection{Inverting shifted middle binomial Hankel matrices}

We now develop the shifted analogs of \eqref{eq: R-inv} and \eqref{eq: B-inv}. 
Note that by construction, $\sB(n)$ is the top-right $n\times n$ submatrix of $\B(n{+}1)$. We may thus begin with the Cholesky decomposition of $\B(n)$ and work towards one for $\sB(n)$. We have
\[
\B(n{+}1) = 
	 \left(\begin{array}{@{}c@{\ }|@{\ }c@{\ }c@{\ }c@{}}
	    \b_0 & \b_1 & \cdots & \b_n \\ 
	    \b_1 & \b_2 & \cdots  & \b_{n+1} \\ 
	    \vdots & \vdots  & \ddots & \vdots \\
	\hline
	    \,\b_n & \b_{n+1} & \dots & \b_{2n-2}	
	 \end{array}\right)
 = 
	 \left(\begin{array}{@{}c@{\ }c@{\ }c@{\ }|@{\ }c@{}}
	    \b_{0,0} & \ & \   \\ 
	    \b_{1,0} & \b_{1,1} & \  \\ 
	    \vdots & \vdots  & \ddots \\
	    \hline
	    \,\b_{n,0} & \b_{n,1} & \cdots & \b_{n,n}
	\end{array}\right)
	 \left(\begin{array}{@{}c@{\ }c@{\ }c@{\ }|@{\ }c@{}}
	    1\  & \ & \   \\ 
	    \ & 2\  & \  \\ 
	    \ & \ & \ddots\, \\
	    \hline
	    &&& \,2\, 
	\end{array}\right)
	 \left(\begin{array}{@{}c@{\ }|@{\ }c@{\ }c@{\ }c@{}}
	    \b_{0,0} & \b_{1,0} & \cdots & \b_{n,0} \\ 
	    \ & \b_{1,1} & \cdots & \b_{n,1} \\ 
	    \ & \ & \ddots & \vdots \\
	    \hline
	    &&&\b_{n,n}
	\end{array}\right) ,
\]
which gives a factorization $(L,D,\widetilde U)$ of $\sB(n)$ that is \emph{nearly} the factorization we seek. E.g., 
\[
\sB(3) =
	 \begin{pmatrix} 
	    \b_1 & \b_2 & \b_3  \\ 
	    \b_2 & \b_3 & \b_4  \\ 
	    \b_3 & \b_4 & \b_5 
	\end{pmatrix}
 = 
	 \begin{pmatrix} 
	    \b_{0,0} & \ & \    \\ 
	    \b_{1,0} & \b_{1,1} & \    \\ 
	    \b_{2,0} & \b_{2,1} & \b_{2,2}   
	\end{pmatrix}
\!
	 \begin{pmatrix} 
	    1 & \ & \   \\ 
	    \ & 2 & \   \\ 
	    \ & \ & 2   
	\end{pmatrix}
\!
	 \begin{pmatrix} 
	    \b_{1,0} & \b_{2,0}  & \b_{3,0}  \\ 
	    \b_{1,1} & \b_{2,1}  & \b_{3,1}  \\ 
	    \ & \b_{2,2} & \b_{3,2}   
	\end{pmatrix} \!.
\]

We aim for a standard Cholesky decomposition $(\sBL,D',\sBL^T)$, with $\sBL$ a lower uni-triangular matrix. As a preview of the final outcome, the relationship between $\sBL$ and $\BL$ will resemble the relationship between $\sCL$ and $\CL$ announced in Lemma \ref{thm: sS is sum}, namely $\b'_{n,k} = \b_{n,k} + \b_{n,k+1}$ for most $n$ and $k$. We arrive at the Cholesky decomposition in four steps: pulling the lower-triangular bits out of $\widetilde U$ (Steps 1--3) and passing them to $L$ (Step 4). In Steps 1--3, we index the rows and columns of $\widetilde U$ as they are found in $\B(n{+}1)$, so, e.g., the top-left position of $\widetilde U$ has index $(0,1)$. 

\medskip\noindent\emph{Step 1: Rescaling the first row.}
Notice that the top-left entry of $\widetilde U$ is $\b_{1,0} = 2$. Let us rescale this zeroth row (dividing by two).

\medskip\noindent\emph{Step 2: Initial reduction.} 
We perform our first row reduction, pivoting on the top-left position $(0,1)$ of $\widetilde U$. Since $\b_{1,1}=1$, we simply subtract the zeroth row of $\widetilde U$ from the first row.

\medskip\noindent\emph{Step 3: Subsequent reductions.} 
As it happens, the $(k{-}1,k)$ position of $\widetilde U$ will be the pivot position for the $k$-th step in the row-reduction process ($k>0$). Moreover, just as in Step 2, the reduction amounts to simply subtracting the \emph{new} $(k{-}1)$-st row from the \emph{current} $k$-th row. Keeping track of all reductions that have come before, our assertion may be summarized as follows.

\begin{lemma}\label{thm: tilde U}
For all $k>0$, the semi-binomial numbers $\b_{k,j}$ satisfy
\[
	\frac{1}{2}\b_{k,0} + \sum_{j=1}^{k} (-1)^{j} \b_{k,j} = 0.
\]
\end{lemma}

\begin{proof}
Multiply the identity by two, then apply \eqref{eq: b recursion} and \eqref{eq: semicat2}, in that order. 
\end{proof}
For example,
\begin{align*}
	\b_{3,0}  -2\b_{3,1} + 2\b_{3,2} - 2\b_{3,3} &= 	(\b_{3,0}-\b_{3,1}) - (\b_{3,1}-\b_{3,2}) + (\b_{3,2}-\b_{3,3}) - (\b_{3,3}-\b_{3,4}) \\
	&= \c_{3,0} - \c_{3,1} + \c_{3,2} - \c_{3,3} \\
	&=0.
\end{align*}

At the end of the reduction process, $\widetilde U$ is upper uni-triangular. Let us now reindex and relabel its entries, since they are the shifted semi-binomial numbers we seek, $\widetilde U = \bigl(\b'_{j,i}\bigr)_{0\leq i,j<n}$. 

\medskip\noindent\emph{Step 4: Combining the lower triangular factors.} 
What remains is to slip the (inverse of the) matrix that records our performed row operations past $D$, and combine it with the lower-triangular matrix $L$. We have 
\begin{align*}
\sB(n) &= 
	 \begin{pmatrix} 
	    \b_{0,0} & \ & \   \\ 
	    \b_{1,0} & \b_{1,1} & \    \\ 
	    \b_{2,0} & \b_{2,1} & \b_{2,2}   \\
	    \vdots && & \ddots
	\end{pmatrix}
\!
	 \begin{pmatrix} 
	    1 & \ & \   \\ 
	    \ & 2 & \   \\ 
	    \ & \ & 2   \\
	    &&& \ddots
	\end{pmatrix}
\!
	 \begin{pmatrix} 
	    2 & \ & \   \\ 
	    1/2\ & 1 & \   \\ 
	    \ & 1 & 1   \\
	    && \ddots & \ddots
	\end{pmatrix}
\!
	 \begin{pmatrix} 
	    1 & \b'_{1,0}  & \b'_{2,0}  & \cdots \\ 
	    \ & 1 & \b'_{2,1}  \\ 
	    \ & \ & 1 \\
	    &&& \ddots
	\end{pmatrix} \\
&=
	 \begin{pmatrix} 
	    \b_{0,0} & \ & \    \\ 
	    \b_{1,0} & \b_{1,1} & \    \\ 
	    \b_{2,0} & \b_{2,1} & \b_{2,2}   \\
	    \vdots &&& \ddots
	\end{pmatrix}
\!
	 \begin{pmatrix} 
	    1 & \ & \   \\ 
	    1/2\ & 1 & \   \\ 
	    \ & 1 & 1   \\
	    &&\ddots & \ddots 
	\end{pmatrix}
\!
	 \begin{pmatrix} 
	    2 & \ & \   \\ 
	    \ & 2 & \   \\ 
	    \ & \ & 2   \\
	    &&& \ddots
	\end{pmatrix}
\!
	 \begin{pmatrix} 
	    1 & \b'_{1,0}  & \b'_{2,0}  & \cdots \\ 
	    \ & 1 & \b'_{2,1}  \\ 
	    \ & \ & 1 \\
	    &&& \ddots 
	\end{pmatrix} 
\\
&=
	 \begin{pmatrix} 
	    \b''_{0,0} & \ & \    \\ 
	    \b''_{1,0} & \b''_{1,1} & \    \\ 
	    \b''_{2,0} & \b''_{2,1} & \b''_{2,2}   \\
	    \vdots &&& \ddots
	\end{pmatrix}
\!
	 \begin{pmatrix} 
	    2 & \ & \   \\ 
	    \ & 2 & \   \\ 
	    \ & \ & 2   \\
	    &&& \ddots
	\end{pmatrix}
\!
	 \begin{pmatrix} 
	    \b'_{0,0} & \b'_{1,0}  & \b'_{2,0}  & \cdots \\ 
	    \ & \b'_{1,1} & \b'_{2,1}  \\ 
	    \ & \ & \b'_{2,2} \\
	    &&& \ddots 
	\end{pmatrix} .
\end{align*}
Note that the entries $\b''_{j,k}$ and $\b'_{j,k}$ appearing above are equal. Indeed, both describe uni-triangular arrays of numbers, and any such decomposition $(L,D,U)$ of a symmetric matrix must have $L=U^T$. Performing the column operations indicated in the displayed equation above, we deduce a simple formula for the shifted semi-binomial numbers $\b'_{j,k}$:
\[
	\b'_{j,0} = \b_{j,0} + \frac{1}{2}\b_{j,1} 
	\qand
	\b'_{j,k} = \b_{j,k} + \b_{j,k+1} \ \ (\hbox{for }k>0).
\]

With $\sBL$ in hand, we are ready to compute the inverse of $\sB$.

\begin{corollary} 
For $0\leq j,k<n$, we have 
\begin{align}
 \label{eq: sR-inv}
 	\bigl(\sBL(n)^{-1}\bigr)_{j,k} &= (-1)^{j+k}\left[2\binom{j+k+1}{j-k} - \binom{j+k}{j-k}\right],  \\[.5ex]
\label{eq: sB-inv}
	\bigl(\sB(n)^{-1}\bigr)_{j,k} &= (-1)^{j+k} \sum_{i=\mathrm{max}\{j,k\}}^{n-1} \frac12 
	\left[2\binom{i+j+1}{i-j}-\binom{i+j}{i-j}\right]
	\left[2\binom{i+k+1}{i-k}-\binom{i+k}{i-k}\right] .
\end{align}
\end{corollary}

\begin{proof}
We write $\sBL(n)$ in $(1,n{-}1)$-block-triangular form, and compute its inverse in blocks. Most of the steps proceed by analogy with Corollary \ref{thm: sC-inv}, using Lemma \ref{lem: matrix arithmetic}\eqref{itm: L-2} and \eqref{eq: R-inv}. The computations for the first column are a bit more intricate, but ultimately simplify to a uniform expression:
\begin{equation*}
	\bigl(\sBL(n)^{-1}\bigr)_{j,k} = \sum_{s=0}^j (-1)^s(-1)^{j-s+k}\left[2\binom{j-s+k}{j-s-k} - \binom{j-s+k-1}{j-s-k}\right]. 
\end{equation*}
The hook identity \eqref{eq: binom-identities} yields \eqref{eq: sR-inv}, and \eqref{eq: sB-inv} then follows from the Cholesky decomposition $\sB = \sBL\cdot\diag{2,2,\dotsc,2}\cdot\sBL^{T}$.
\end{proof}

%
%
\section{Decay of Riesz transforms}
\label{Decay of Riesz transforms}
In this section we address the decay of Riesz transforms, 
which is key to the proof of Theorem \ref{main}. 
There is a vast literature concerning decay properties of Riesz transforms of functions, often by way of 
 determining weights $w:\reals^d\to [0,\infty)$ for which the Riesz transforms 
are bounded on $L_p(w)$, $1<p<\infty$ 
(absent extra restrictions, like moment conditions, these are the Muckenhoupt weights). 
The article \cite{AdamsE} identifies
weights for which the $\Riesz_j$, when restricted to functions with vanishing moment conditions, 
are bounded on $L_p(w)$. A recent article that investigates 
decay of Riesz transforms of wavelet systems is \cite{WardUnser}.

 We are concerned with establishing pointwise decay of Riesz transforms. To this end, we work with spaces 
 $\Space_{J,\delta}$, which are more restrictive than those used by Adams in \cite{AdamsE}. 
 The main goal of this section is the basic decay result Lemma \ref{RTE}.
It can be developed from scratch with little effort, however we could not find a suitable result in the literature.
For the sake of completeness, we include it here. 

\subsection{Riesz transforms} The $j$th Riesz transform 
can be expressed on $C^1(\reals^N) \cap L_1(\reals^N)$ as
$
\Riesz_j: 
f
\mapsto 
c_N \lim_{\epsilon\to 0^+}\int_{\reals^N\setminus B(0,\epsilon)} \frac{y_jf(\cdot-y)}{|y|^{N+1}} \dif y
$ 
for some positive constant $c_N$.
Thus it is a convolution operator  
$\Riesz_j f = f*\riesz_j$
using the tempered distribution 
$\riesz_j(y) := c_{N} \mathrm{p.v.} ({y_{j}}/{|y|^{N+1}})$
(cf. \cite[Section 4.1.4]{Grafakos} for a discussion).
Because
$\riesz_j$ is smooth on $\reals^N\setminus\{0\}$,
this implies that  for compactly supported $h$, each $\Riesz_j h$ is $C^{\infty}$ outside of $ \mathrm{supp}\,h$. 

If 
$f$ is in 
$C^{\epsilon}(\reals^N)\cap L_p(\reals^N)$, with $p<\infty$, 
then $|\Riesz_jf(x)|$ is controlled by $\|f\|_p + [f]_{\epsilon,x}$, since 
%
\begin{equation*}
|\Riesz_jf(x)|
 \le
 \left|\int_{\reals^N\setminus B(0,1)}\frac{y_{j}f(x-y)}{|y|^{N+1}}\dif y\right|
+
\left|\lim_{\epsilon\to 0} \int_{B(0,1)\setminus B(0,\epsilon)}
\frac{y_{j}\bigl(f(x-y) - f(x)\bigr)}{|y|^{N+1}}\dif y\right|.
\end{equation*}
%
Here, we have used the fact that $y\mapsto f(x)\frac{y_{j}}{|y|^{N+1}}$ is odd
and that $y\mapsto \frac{y_{j}}{|y|^{N+1}}$ lies in $L_{p'}\bigl(\reals^N\setminus B(0,1)\bigr)$ for
$p'>1$.
Indeed, for $k\ge 0$ and $f\in C^{k+\epsilon}(\reals^N)\cap W_p^k(\reals^N)$, if $|\alpha|\le k$ then
%
\begin{equation}
|D^{\alpha} \Riesz_j f(x)|
\le 
C 
\left( 
     \int_{\reals^N\setminus B(0,1)}\frac{|D^{\alpha}f(x-y)|}{|y|^{N}}\dif y   + | f|_{C^{|\alpha|+\epsilon}\bigl(B(x,1)\bigr)}
\right)
\label{riesz_bound}
\end{equation} 
for some constant $C$ depending on $\epsilon$ and $N$.
%
\subsection{Riesz transform of compactly supported functions} 
For $f\in L_1(\reals^N)$ with compact support, the estimate 
$|D^{\alpha} \Riesz_j f (x) | \le c_{N} \|f\|_1 \bigl(\dist(x,\supp{f})\bigr)^{-(N+|\alpha|)}$ holds for $x$ sufficiently far from $\supp{f}$.

When $f$ is compactly supported and has many vanishing moments, $\Riesz_j f$ 
exhibits a faster rate of decay. Specifically, for $f\perp \Pi_L$ and $x$ far from $\supp{f}$, 
$|D^{\alpha} \Riesz_j f (x) | = \mathcal{O}(|x|^{-(L+N+1+|\alpha|)})$ at $\infty$.

%
%
\begin{lemma}\label{riesz_vanishing_moment}
Assume $f$ is bounded, has compact support and satisfies  $f\perp\Pi_L(\reals^N)$.
For a multi-index $\alpha$ there is a constant $C$ depending on $L$, $|\alpha|$ and $N$ so that
\begin{equation*}
\left| D^{\alpha}\Riesz_jf(x) \right| 
\le C\,
\|f\|_{\infty}
\bigl(\diam\bigl(\mathrm{supp}(f)\bigr)\bigr)^{L+N+1}\,
\dist\bigl(x,\mathrm{supp} (f)\bigr)^{-(L+N+1+|\alpha|)}.
\end{equation*}
\end{lemma}
\begin{proof}
Fix $x\notin \supp{f}$.
Let 
$P$
be the Taylor polynomial of degree $L$ 
to the function 
$
y\mapsto 
\left(D^{\alpha} \riesz_j\right)(x-y)
$, centered at some $\overline{x}$ (assume it is near to the support of $f$).
Thus, we have
$
\left(D^{\alpha} \riesz_j\right)(x-y)
= P(y) +\Err_{L} (y)
$ 
with $\Err_{L}$ the remainder
from Taylor's theorem.
One can estimate this remainder by taking partial derivatives of order $L+1$ of  $D^{\alpha}\riesz_j$
noting that, by homogeneity, $|D^{\beta} \riesz_j(x)|\le C |x|^{-(N+|\beta|)}$ for $x\ne0$ and $C$ depending on $N$ and $\beta$. 
Consequently,
%
\begin{equation*}
| \Err_{L} (y)| 
\le
C \int_0^1 \frac{|\overline{x}-y|^{J+1}}{|x-\overline{x} +t(\overline{x}-y)|^{L+N+|\alpha|+1}} (1-t)^{J}\dif t .
\end{equation*}
The lemma follows from the fact that
$
D^{\alpha}\Riesz_j f(x)
 = 
c_{N}   
\left\langle 
f, 
\Err_{L} 
\right \rangle.
$ 
\end{proof}
%

\subsection{Decomposition}
Recall the space $\Space_{J,\delta}^s(\reals^N)$ introduced in \eqref{appendix_decay}.
In this subsection, we develop a simple moment preserving decomposition of a function into  pieces 
supported on dyadic balls.
Given $\delta\ge 0$ and $J = L+N$, with $L\in \nats$, we show that a function $f\in \Space_{J,\delta}^{s}(\reals^N)$ having $L$ vanishing moments
(i.e., $f\perp \Pi_L$)
can be decomposed into a number of compactly supported terms and a (globally supported) tail:
$f = 
f_0+f_1+\dots+f_{k-1}
+ \widetilde{f}_k$. 
For this decomposition, each term $f_{\ell}$ has support in $B(0,2^{\ell})$ and satisfies $f_{\ell} \perp \Pi_{L}$.
Finally, each term obeys the estimates given below in Lemma \ref{trunc_bounds}.

It is likely that the kind of decomposition we are after can be accomplished by using an atomic decomposition or  a suitable expansion in wavelets (see, e.g., \cite{WardUnser}). We avoid this machinery and develop a moment preserving decomposition from scratch.

Assume $f\in \Space_{J,\delta}^s(\reals^N)$.
For $R>1$ we split the function  $f$ into two parts $f = f_a +f_b$. This is done simply by truncating with a smooth
cut-off:
the compactly supported part $f_a(x) := f(x) \psi(x/R)$ has $\supp {f_a}\subset B(0,R)$ and satsifies 
\begin{equation}\label{compact_part}
|f\psi(\cdot/R)|_{C^{\sigma}(\reals^N)}\le 
C  \sup_{\gamma\le \sigma} R^{\gamma - \sigma} |f|_{C^{\gamma}(\reals^N)}.
\end{equation}
The tail $f_b(x) := \bigl(1-\psi(x/R)\bigr) f(x)$ satisfies,   for $\sigma\le s$ and for $\rho\ge0$,
%
\begin{eqnarray}\label{fb_decay}
|f_b|_{C^{\sigma}(\reals^N\setminus B(0,\rho))} 
&\le& 
C \|f\|_{ \Space_{J,\delta}^s(\reals^N)}
\left(R+\rho\right)^{-(J+\sigma)}  \bigl( \log \bigl(e+ \max(R, \rho)\bigr)\bigr)^{-\delta} .\label{tail}
\end{eqnarray}
%
In other words, $\|f_b\|_{\Space_{J,\delta}^s(\reals^N)}\le C\|f\|_{\Space_{J,\delta}^s(\reals^N)}$.

We consider  a family of smooth functions 
$(\Phi_{\beta})_{|\beta|\le L}$, supported in $B(0,1/2)$ 
and dual to the monomials: $\int_{B(0,1/2)}  x^{\gamma} \Phi_{\beta}(x)\dif x = \delta_{\beta,\gamma}$.
Let 
$c_{\beta,R}(f) 
:= 
\int_{\reals^N} x^{\beta} 
f_b(Rx)\,
\dif x $
 and note that
this equals 
$R^{-(N+|\beta|)} 
\int_{\reals^N} x^{\beta} 
f_b(x)
\dif x
$. 
Applying the radial decay of $f_b$ from \eqref{fb_decay}, we obtain 
%
\begin{equation} \label{moment}
|c_{\beta,{R} }(f)|
\le
C\|f\|_{ \Space_{J,\delta}(\reals^N)} 
{R} ^{-J} \bigl(\log(e+ {R} )\bigr)^{1-\delta} .
\end{equation} 
%
Let 
$
P_{R}  f(x) 
:= 
\sum_{|\beta\le L}    c_{\beta,{R} }(f)  \Phi_{\beta}\left({x/{R} }\right).
$
It follows that $\int p(x) P_{{R} } f(x) \dif x = \int p(x)f_{b}(x) \dif x$ for all $p\in \Pi_L$. 
It is not hard to see that for any $\sigma \ge 0$,
%
%
  \begin{equation}\label{PR_est}
| P_{R}f|_{C^{\sigma}(\reals^N)}
\le
   C
   \|f\|_{ \Space_{J,\delta}(\reals^N)} R^{-(J+\sigma)}
   \bigl(\log(e+ R)\bigr)^{1-\delta},
\end{equation}
with the constant $C$ depending on  the family 
 $(\Phi_{\beta})_{|\beta|\le L}$ and $\sigma$.
%

 The scaled truncation operator is 
%
\begin{equation*}
 \T_R f
 := 
\psi\left(\cdot/R\right) f+ P_R f .
\end{equation*}
 This is simply the truncation by the cut-off function $\psi\left(\cdot/R\right)$, plus a 
 projection which preserves the original moment conditions.
%

%
%
\begin{lemma}\label{truncation}
For $R\ge 1$ and $f\in \Space_{J,p}(\reals^N)$,  
$f -\T_{R}f$ annihilates polynomials of 
degree $L$. 
Moreover, $\supp {\T_{R} f}\subset B(0,{R})$ and 
%
\begin{equation*}
| f - \T_{R} f|_{C^{\sigma}(\reals^N)}\le 
    C\|f\|_{\Space_{J,\delta}^s(\reals^N)} 
    {R}^{-(J+\sigma)} \bigl(\log(e+ {R})\bigr)^{1-\delta}.
\end{equation*}
%
\end{lemma}
%
\begin{proof}
We note that $\supp {\T_{R} f} \subset \supp{f_a}\cup \supp{P_{R} f} \subset B(0,{R})$.
The bounds follow from \eqref{tail} and \eqref{PR_est}. 
\end{proof}


We can iterate applications of  the  
operator $\T_{R}$, obtaining initially
${f}_0 := \T_1 f$ and a complementary part
$\widetilde{f}_1  := f-f_0.$
For an integer $\ell\ge 1$,
%
\begin{equation*}
f_\ell 
:= 
\T_{2^\ell}(\widetilde{f}_{\ell})
\qand  
\widetilde{f}_{\ell+1}  
:= 
\widetilde{f}_{\ell}-f_\ell,
\end{equation*}
%
so that for any $k\ge 1$, we can write $f  = f_0+f_1+\dots+f_{k-1}+\widetilde{f}_k$, 
with a compactly supported component 
$ 
\sum_{\ell=0}^{k-1} f_{\ell}$.
Indeed, each $f_{\ell}$ is supported in $B(0,2^{\ell})$,
so the sum of the first $k$ terms  is supported in $ B(0,2^{k-1})$.

The tail can be written as 
$
  \widetilde{f}_{k} 
= 
  \widetilde{f}_{k-1} - \T_{2^{k-1}} \widetilde{f}_{k-1} 
=
  \bigl(1 - \psi(\cdot/2^{k-1})\bigr)     \widetilde{f}_{k-1}
  - 
  P_{2^{k-1}} \widetilde{f}_{k-1}
$.
The first term has support in $\reals^{N}\setminus B(0,2^{k-2})$ while the second is supported in $B(0,2^{k-2})$.
Because $\widetilde{f}_{k-1} = \widetilde{f}_{k-2} = \dots = f$ outside of  $ B(0,2^{k-1})$,
this leads to the expression
%
\begin{equation*}
\widetilde{f}_{k}  
= 
\bigl(1 - \psi\left({\cdot}/{2^{k-1}}\right)\bigr)f
  - 
P_{2^{k-1}} \widetilde{f}_{k-1}.
\end{equation*}
%
Because the difference
$
 f-\widetilde{f}_{k-1}
 $ 
has support in $B(0,2^{k -2})$, it is in the kernel of $P_{2^{k-1}}$.
Consequently,
$
P_{2^{k-1}}
\widetilde{f}_{k-1}
= P_{2^{k-1}}f ,
$
and the tail can be computed directly from $f$ in one step:
\begin{equation}\label{tail_formula}
\widetilde{f}_{k}  = \bigl(1 - \psi\left(\cdot/{2^{k-1}}\right)\bigr) f- P_{2^{k-1}} f
=
 f - \T_{2^{k-1}}f.
\end{equation}
%
It is worth noting that 
on the punctured ball $|x|\ge 2^{k-1}$ the tail is simply $\widetilde{f}_{k}(x) = f(x)$, while
on the annulus 
$2^{k-2}\le |x|\le2^{k-1}$ the tail is $\widetilde{f}_{k}(x) = \bigl(1 - \psi\left({x}/{2^{k-1}}\right)\bigr) f(x)$.
%
%
\begin{lemma}\label{trunc_bounds}
Suppose $\delta>1$, $J=L+N$ and $f\in\Space_{J,\delta}^s(\reals^N)$  satisfies the  
moment condition $f\perp \Pi_L$.
For $|x|\le 2^{k}$ and $f_\ell,\widetilde{f}_{\ell}$ as above,
\begin{eqnarray}\label{A_D_decay}
\bigl|{f}_{\ell}\bigr|_{C^{\sigma}(\reals^N)}
&\le&
C 
\|f\|_{\Space_{J,p}^s(\reals^N)}\,
k^{1-\delta}\,
2^{-{k}(J+\sigma)}.
\end{eqnarray}
Likewise, for $|x|\le 2^{k-1}$,
\begin{eqnarray}\label{A_D_tail}
\bigl|\widetilde{f}_{k}\bigr|_{C^{\sigma}(\reals^N)}
&\le& 
C \|f\|_{\Space_{J,p}^s(\reals^N)}\,
k^{1-\delta}\,2^{-{k}(J+\sigma)},
\end{eqnarray}
%
while for $|x|>2^{k-1}$, $ \widetilde{f}_{k}(x) =f(x).$

%
\end{lemma}
\begin{proof}
The second inequality follows from  \eqref{tail_formula} and Lemma \ref{truncation}.
In fact, in the punctured space of radius  $2^{k-2}$, we have that 
$$
|\widetilde{f}_{k}|_{C^{\sigma}\bigl(\reals^N\setminus B(0,2^{k-2})\bigr)}
\le
\bigl|
  \bigl(1 - \psi({\cdot}/2^{k-1})\bigr) f 
\bigr|_{C^{\sigma}\bigl(\reals^N\setminus B(0,2^{k-2})\bigr)}
\le 
  C 
  \|f\|_{\Space_{J,p}^s(\reals^N)}\,k^{-\delta}\,2^{-{k}(J+\sigma)} .
$$

The inequality \eqref{A_D_decay} follows from the fact that
$
f_{k} = \T_{2^k}(\widetilde{f}_{k}) 
=
\psi(\cdot/2^{k}) \widetilde{f}_{k} + P_{2^{k}}  \widetilde{f}_{k}.
$ 
Because the kernel of $P_{2^{k}}$ includes the range of $P_{2^{k-1}},$ we have
$
f_{k} 
=
  \psi(\cdot/2^{k}) 
  \widetilde{f}_{k} 
+
  P_{2^{k}}
  \bigl[
    \bigl(1 - \psi({\cdot}/{2^{k-1}})
   \bigr) 
    f
  \bigr]
$.
Both terms can be estimated by the estimate on $\widetilde{f}_{k}$. The first by using \eqref{compact_part} and 
the second term can be estimated with the help of \eqref{PR_est} followed by \eqref{tail}. 
\end{proof}

%
%
\subsection{Riesz transform estimates for globally supported functions}
Using the moment preserving decompostion in conjunction with Lemma \ref{riesz_vanishing_moment} and \eqref{riesz_bound}
we can estimate the size and decay of a smooth function satisfying decay and moment
conditions.
%
%
\begin{lemma}\label{RTE}
Let $L$ be a non-negative integer, $\delta>2$ and assume 
$f\in \Space_{J,\delta}^{s}(\reals^N)$ 
with $J=L+N$ satisfies the moment condition $f\perp \Pi_L$. 
Then for $|\alpha|\le s$ we have
%
\begin{equation*}
|D^{\alpha}\Riesz_j f(x)|\le C \|f\|_{\Space_{J,\delta}^s(\reals^N)} (1+|x|)^{-(J+|\alpha|)} \left(\log(e+|x|)\right)^{2-\delta},
\end{equation*}
with the constant $C$ depending only on $L$, $\delta$, $s$ and $N$.
%
\end{lemma}
\begin{proof}
Fix $k>1$ and let $2^{k}\le |x| < 2^{k+1}$.

For  ${\ell}< k$ we apply Lemma \ref{riesz_vanishing_moment} to $f_{\ell}$, 
noting that $\supp{f_\ell}\subset B(0,2^{\ell})$ implies
$\mathrm{diam}(\supp{f_{\ell}})\le 2^{\ell+1}$ and
 $\dist(x,\supp{ {f}_{\ell}}) \ge 2^{k-1}$, 
 while \eqref{A_D_decay} gives 
 $\|f_{\ell}\|_{\infty} \le  C \|f\|_{\Space_{J,\delta}^s}  \ell^{1-\delta} 2^{-{\ell}J}$. Thus,
%
\begin{eqnarray*}
|D^{\alpha} \Riesz_j  f_{\ell} (x) | 
&\le& 
 C  
 \|f\|_{\Space_{J,\delta}^s} 
 \ell^{1-\delta} 2^{-{\ell}J} 
 \times 
 2^{(\ell+1) (J+1)} 
 \times 
 2^{(1-k)(J+|\alpha|+1)} \\
&\le& 
 C  
 \|f\|_{\Space_{J,\delta}^s}\, 
 \ell^{1-\delta} 
 |x|^{-(J+|\alpha|)}. 
\end{eqnarray*}
%
Summing this gives 
$
|D^{\alpha} \Riesz_j \sum_{\ell =0}^{k-1} f_{\ell}|
< 
C \|f\|_{\Space_{J,\delta}^s} |x|^{-(J+|\alpha|)}(\log|x|)^{2-\delta}.
$

Applying \eqref{riesz_bound} to the tail $\widetilde{f}_{k}$ gives 
%
\begin{equation*}
|D^{\alpha}\Riesz_j \widetilde{f}_{k}(x)| 
\le 
C 
\left( 
   \int_{\reals^N\setminus B(0,1)}\frac{|D^{\alpha}\widetilde{f}_{k}(x-y)|}{|y|^{N}}\dif y
  + 
   |  \widetilde{f}_{k}|_{C^{|\alpha|+\epsilon}\bigl(B(x,1)\bigr)}
\right).
\end{equation*}
%
We can estimate the H{\"o}lder seminorm by
$
C\|f\|_{\Space_{J,\delta}^s(\reals^N)}  k^{1-\delta}2^{-k(J+|\alpha|+\epsilon)}
$
 directly from 
  \eqref{A_D_tail}.
The integral is split in two parts. For $|y|>2^{k+2}$, we have that 
$|y|/2 \le |y-x| $,  
and we obtain 
%
\begin{eqnarray*}
\int_{\reals^N\setminus B(0,2^{k+2})}\frac{|D^{\alpha}\widetilde{f}_{k}(x-y)|}{|y|^{N}}\dif y
&\le&
C
\int_{\reals^N\setminus B(0,2^{k+2})}
 {\|f\|_{\Space_{J,\delta}^s} |y|^{-(J+|\alpha|)} (\log |y|)^{-\delta}}{|y|^{-N}}
\dif y\\
&\le& 
C \|f\|_{\Space_{J,\delta}^s}k^{1-\delta} 2^{-k(J+|\alpha|)} .
\end{eqnarray*}
%
The final inequality follows from \eqref{A_D_tail} and by bounding the interior integral, over $y\in B(0,2^{k+2})$, by
%
\begin{eqnarray*}
\int_{B(0,2^{k+2})\setminus B(0,1)}\frac{|D^{\alpha}\widetilde{f}_{k}(x-y)|}{|y|^{N}}\dif y
&\le&
C\|f\|_{\Space_{J,\delta}^s}\int_1^{2^{k+2}}
{k^{1-\delta}2^{-k(J+|\alpha|)}}r^{-1}\dif r\\
&\le& 
C\|f\|_{\Space_{J,\delta}^s} k^{2-\delta}2^{-k(J+|\alpha|)},
\end{eqnarray*}
%
where we have employed the uniform estimate 
$\|D^{\alpha}\widetilde{f}_{k}\|_{\infty} \le C|f|_{\Space_{J,\delta}^s}  k^{1-\delta}2^{-k(J+|\alpha|)}$, 
which follows from 
\eqref{A_D_tail}.
\end{proof}
%
%
%
\section{Concluding Remarks}\label{sec: concluding}
In this section we make some generalizations of the approach to solving Dirichlet problems presented here.
We begin by considering a second natural family of boundary operators. After another detour through path counting, we finish by considering a generalization of the underlying elliptic differential operator. 

This is meant to demonstrate two things. First, that the previous method can be extended
to other settings  where determining auxiliary functions $\bfg$ from Dirichlet data $\bfh$ can be tied to solving a path counting problem.
Second, that the operators selected in \eqref{DirichletOperators} are exceptional in the sense that the resulting path counting problem yields a  nice solution.

\subsection{Alternative boundary conditions}\label{sec: alternative boundary conditions}
We now consider a polyharmonic Dirichlet problem using the (perhaps more natural) family of boundary differential operators ${\lambda}_j$ where
\begin{equation}\label{standard_D}
{\lambda}_ju(x_1\dots x_{d-1})  =(-1)^j \frac{\partial^j}{\partial x_d^j}u(x_1,\dots,x_{d-1},0).
\end{equation} 
In this case, we again seek a solution of the form 
$\sum_{j=0}^{m-1} \int_{\reals^{d-1}} g_j(\alpha) \lambda_{j,\alpha} \phi(x-\alpha) \dif \alpha$ 
(using the new boundary operators). To do so we consider the resulting system of integral equations
$\bfh = \mathcal{V} \bfg$. To solve this system, we consider the symbol $\sigma(\mathcal{V})$ 
of the operator $\mathcal{V}$, 
repeating the calculations of Sections \ref{Boundary Layer Potential Operators} and \ref{Boundary values of the boundary layer potentials}.
Each entry $v_{k,j}$ has symbol (when restricted to $g$ satisfying suitable decay, moment and smoothness conditions - the details of this are left to the reader)
$$\sigma(v_{k,j})(\xi') =
\frac{(-1)^{(j+k)/2}}{2\pi}\left( \int_\reals \frac{\zeta^{j+k}}{(1+\zeta^2)^{m}}\dif \zeta\right)
|\xi'|^{j+k+1-2m} .$$

The solution of the problem follows along the same lines as before.
The symbol of the operator $\mathcal{V}$ is
 $\sigma(\mathcal{V} ) = \bigl(\sigma(v_{k,j})\bigr)_{k,j} = A_1(\xi')\M A_2(\xi')$ with $A_1$ and $A_2$ exactly as before. 
 The matrix $\M$ has a checkerboard pattern, since
 $\M_{k,j}= \int_\reals {\zeta^{j+k}}{(1+\zeta^2)^{-m}}\dif \zeta=0$ when $j+k$ is odd. 
 The matrix $\M$ is invertible and the entries of its inverse provide the formulas for the
 auxiliary functions $g_j = \sqrt{- \Deltad}\sum_{k=0}^{m-1} (\M^{-1})_{j,k} \Deltad^{m-1-(j+k)/2}  h_k$.
 
The matrix $\M$ can be restructured into two blocks as in 
Section \ref{sec: path counting and invertibility}.
 These two blocks are essentially the same---they come from a single family of symmetric Hankel matrices. In short, we have
  $$\tilde\M = \begin{pmatrix} H_1 & 0 \\ 0 & H_2\end{pmatrix}, 
  \quad
  \text{with}
  \quad
  H_1 = (a_{j+k})_{0\le j,k\le \lfloor \frac{m-1}{2}\rfloor}
  \qand
  H_2 = 
 (a_{j+k+1})_{0\le j,k\le \lfloor\frac{m-2}{2}\rfloor},
 $$
 where 
 $$a_n =\int_\reals \frac{\zeta^{2n}}{(1+\zeta^2)^{m}}\dif \zeta\,.
 $$
 
The matrices $H_1$ and $H_2$ are totally positive. 
This follows easily from from \cite[Theorem 4.4]{Pinkus}, which states that a Hankel matrix is totally positive 
if and only if both it and its lower, left-hand submatrix (removing the top row and rightmost column) are positive 
definite.

\subsection{More on path counting}\label{sec: more path counting}
(We continue the discussion begun in Section \ref{sec: tpm}, as promised.) In the preceding examples---and the one that follows---we know by Theorem \ref{thm: TPM} that there exists a directed acyclic graph $G$ so that minors of $\tilde\M$ may be computed by path counting. In fact, there is a constructive way to arrive at such a graph, which we now illustrate.

Suppose $A = (a_{i,j})_{0\leq i,j<n}$ is a totally positive matrix. One constructs a ``planar network'' of vertices and labeled (\emph{weighted}) directed edges, illustrated for $n=3$ in Figure \ref{fig: planar network} and below, with $n$ origin and destination vertices chosen. Only the diagonal edges and central horizontal edges carry non-unity weight.  All edges are oriented eastward. (The reader should observe that the choices made for the $o_i$ and $d_j$ are non-permutable, and indeed remain so for any subsequences $(o_{i_1},\ldots,o_{i_r})$ and $(d_{j_1},\ldots, d_{j_r})$ of the chosen vertices.) 
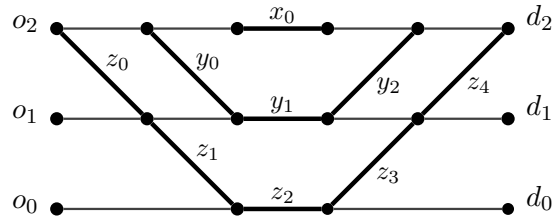
\begin{figure}[!hbt]
\psset{unit=1.2, linewidth=.025}
\begin{pspicture}(5.4,2.4)(-.2,-.2)
\cnode*(0,0){2.5pt}{p00}	\cnode*(0,1){2.5pt}{p01}	\cnode*(0,2){2.5pt}{p02} 
					\cnode*(1,1){2.5pt}{p11}	\cnode*(1,2){2.5pt}{p12} 
\cnode*(2,0){2.5pt}{p20}	\cnode*(2,1){2.5pt}{p21}	\cnode*(2,2){2.5pt}{p22} 
\cnode*(3,0){2.25pt}{p30}	\cnode*(3,1){2.5pt}{p31}	\cnode*(3,2){2.5pt}{p32} 
					\cnode*(4,1){2.5pt}{p41}	\cnode*(4,2){2.5pt}{p42} 
\cnode*(5,0){2.25pt}{p50}	\cnode*(5,1){2.5pt}{p51}	\cnode*(5,2){2.5pt}{p52} 

\ncline[linecolor=darkgray]{p00}{p20}
\ncline[linewidth=.05]{p20}{p30}\Aput[.5pt]{\small$z_2$}
\ncline[linecolor=darkgray]{p30}{p50}
\ncline[linecolor=darkgray]{p01}{p11}\ncline[linecolor=darkgray]{p11}{p21}
\ncline[linewidth=.05]{p21}{p31}\Aput[.5pt]{\small$y_1$}
\ncline[linecolor=darkgray]{p31}{p41}\ncline[linecolor=darkgray]{p41}{p51} 
\ncline[linecolor=darkgray]{p02}{p12}\ncline[linecolor=darkgray]{p12}{p22}
\ncline[linewidth=.05]{p22}{p32}\Aput[.5pt]{\small$x_0$}
\ncline[linecolor=darkgray]{p32}{p42}\ncline[linecolor=darkgray]{p42}{p52} 

\ncline[linewidth=.05]{p02}{p11}\Aput[.3pt]{\small$z_0$}
\ncline[linewidth=.05]{p11}{p20}\Aput[.3pt]{\small$z_1$}
\ncline[linewidth=.05]{p41}{p30}\Aput[.3pt]{\small$z_3$}
\ncline[linewidth=.05]{p52}{p41}\Aput[.3pt]{\small$z_4$}
\ncline[linewidth=.05]{p12}{p21}\Aput[.3pt]{\small$y_0$}
\ncline[linewidth=.05]{p42}{p31}\Aput[.3pt]{\small$y_2$}

\uput{0.20}[20](5,2){$d_2$}
\uput{0.20}[170](0,2){$o_2$}
\uput{0.20}[20](5,1){$d_1$}
\uput{0.20}[170](0,1){$o_1$}
\uput{0.20}[20](5,0){$d_0$}
\uput{0.20}[170](0,0){$o_0$}

\end{pspicture}
\caption{A labeled planar network with three non-permutable origins and destinations.}
\label{fig: planar network}
\end{figure}

To properly assign values to the weights $x_i,y_j,z_k$, we evidently must solve the following system of equations:
\[
\left\{\begin{array}{@{\,}l@{\ \quad\ }l@{\ \quad\ }l}
a_{0,0} = z_2, & a_{0,1} = z_2z_3, & a_{0,2} = z_2z_3z_4, \\[.5ex]
a_{1,0} = z_1z_2, & a_{1,1} = z_1z_2z_3 + y_1,  
		&  a_{1,2} = z_1z_2z_3z_4+y_1y_2 + y_1z_4, \\[.5ex]
a_{2,0} = z_0z_1z_2, & a_{2,1} = z_0(\ a_{1,1} \ ) +  y_0y_1,    
		& a_{2,2} = z_0(\ a_{1,2} \ ) + y_0\bigl(y_1y_2 + y_1z_4\bigr)+x_0. 
\end{array}\right.
\]
Now, the $z_i$ are clearly uniquely determined (and nonzero), since $A$ is totally positive. Thus one can solve for, say, $y_1$ as well (using the equation for $a_{1,1}$). If $y_1<0$, we are sunk (as the definition we have taken in Section \ref{sec: tpm} requires that all weights are nonnegative). However, we are lucky: $y_1$ is the ratio of minors $(a_{0,0}a_{1,1}-a_{1,0}a_{0,1})/a_{0,0}$, which is nonnegative by the total positivity of $A$. Similarly, $y_0$ will be a ratio of (a product of) minors of $A$. We leave further analysis of this example to the reader. The obvious extension to larger $n$ amounts to a constructive proof of Theorem \ref{thm: TPM} for totally positive matrices. (See \cite[Ch. 2]{FJ} for a discussion of Whitney's bidiagonal factorization and the extension to totally nonnegative matrices.)

Unfortunately, the planar networks produced from the above construction are not necessarily as ``elegant'' as the directed graphs presented in Section \ref{sec: path counting and invertibility}, even when the entries of $A$ are integers. We illustrate with the $H_1$ from  Section \ref{sec: alternative boundary conditions}, taking $m=5$ and first scaling by $2^{m-1}/\pi$. 
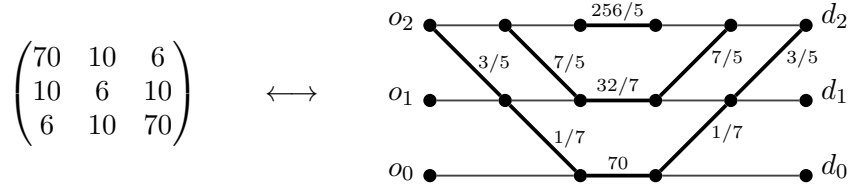
\begin{center}
$\displaystyle
\begin{pmatrix}
70 & 10 & 6 \\ 
10 & 6 & 10 \\ 
6 & 10 & 70 \end{pmatrix}$
\quad \quad $\longleftrightarrow$ \quad\quad\quad
\raisebox{-.5\height}{\psset{unit=1, linewidth=.025}
\begin{pspicture}(5.4,2.5)(-.2,-.5)
\cnode*(0,0){2.5pt}{p00}	\cnode*(0,1){2.5pt}{p01}	\cnode*(0,2){2.5pt}{p02} 
					\cnode*(1,1){2.5pt}{p11}	\cnode*(1,2){2.5pt}{p12} 
\cnode*(2,0){2.5pt}{p20}	\cnode*(2,1){2.5pt}{p21}	\cnode*(2,2){2.5pt}{p22} 
\cnode*(3,0){2.5pt}{p30}	\cnode*(3,1){2.5pt}{p31}	\cnode*(3,2){2.5pt}{p32} 
					\cnode*(4,1){2.5pt}{p41}	\cnode*(4,2){2.5pt}{p42} 
\cnode*(5,0){2.5pt}{p50}	\cnode*(5,1){2.5pt}{p51}	\cnode*(5,2){2.5pt}{p52} 

\ncline[linecolor=darkgray]{p00}{p20}
\ncline[linewidth=.05]{p20}{p30}\Aput[1pt]{\tiny$70$}
\ncline[linecolor=darkgray]{p30}{p50}
\ncline[linecolor=darkgray]{p01}{p11}\ncline[linecolor=darkgray]{p11}{p21}
\ncline[linewidth=.05]{p21}{p31}\Aput[.5pt]{\tiny$32/7$}
\ncline[linecolor=darkgray]{p31}{p41}\ncline[linecolor=darkgray]{p41}{p51} 
\ncline[linecolor=darkgray]{p02}{p12}\ncline[linecolor=darkgray]{p12}{p22}
\ncline[linewidth=.05]{p22}{p32}\Aput[.5pt]{\tiny$256/5$}
\ncline[linecolor=darkgray]{p32}{p42}\ncline[linecolor=darkgray]{p42}{p52} 

\ncline[linewidth=.05]{p02}{p11}\Aput[-1pt,npos=.65]{\tiny$3/5$}
\ncline[linewidth=.05]{p11}{p20}\Aput[-1pt,npos=.65]{\tiny$1/7$}
\ncline[linewidth=.05]{p41}{p30}\Aput[-.3pt,npos=.26]{\tiny$1/7$}
\ncline[linewidth=.05]{p52}{p41}\Aput[-.3pt,npos=.26]{\tiny$3/5$}
\ncline[linewidth=.05]{p12}{p21}\Aput[-1pt,npos=.65]{\tiny$7/5$}
\ncline[linewidth=.05]{p42}{p31}\Aput[-.3pt,npos=.26]{\tiny$7/5$}

\uput{0.20}[20](5,2){$d_2$}
\uput{0.20}[170](0,2){$o_2$}
\uput{0.20}[20](5,1){$d_1$}
\uput{0.20}[170](0,1){$o_1$}
\uput{0.20}[20](5,0){$d_0$}
\uput{0.20}[170](0,0){$o_0$}
\end{pspicture}
}
\end{center}
The appearance of binomial coefficients in the modified $H_1$ above suggests an alternative approach along the lines of the grid graphs in Section \ref{sec: path counting and invertibility}. 
We leave this search to the reader. 

\subsection{Alternative elliptic differential operator}\label{sec: alternative operator}
We now consider a boundary value problem of the form
\begin{equation}\label{second_PDE}
\begin{cases} 
  \L_m u(x) = 0,& \text{for $x\in \reals^{d-1}\times\reals_+$};\\
  \lambda_k u = h_k,& \text{for }k=0,\dots, m-1,\\
\end{cases}
\end{equation}
where $\L_m$ is a constant coefficient elliptic operator of the form $\L_m= \prod_{j=0}^{m-1} (\Delta - r_j^2)$, and the numbers $0<r_0^2< r_1^2< \dots<r_{m-1}^2$ are positive, distinct real numbers. 
In this case, we utilize the Dirichlet operators \eqref{standard_D} of the previous example 
and again consider a solution of the form 
$\sum_{j=0}^{m-1} \int_{\reals^{d-1}} g_j(\alpha) \lambda_{j,\alpha} \phi(x-\alpha) \dif \alpha$,
this time using a fundamental solution $\phi$ for $\L_m$. 

In this case, $(-1)^m\widehat{\phi}(\xi)>0$, 
and  the operators  $v_{k,j}g
=
\lambda_k\int_{\reals^{d-1}} g(\alpha) \lambda_{j,\alpha} \phi(\cdot-\alpha) \dif \alpha
$ 
can be expressed as
$$v_{k,j} g(x') = \frac{i^{j+k}(-1)^m}{(2\pi)^d}\int_{\reals^d}\frac{\xi_d^{j+k}\widehat{g}(\xi')e^{i\langle x',\xi'\rangle}}
{(|\xi|^2+r_0^2)\dots(|\xi|^2+r_{m-1}^2)} \dif \xi.$$
For odd $j+k$, this yields $v_{k,j}=0$. Otherwise, we have the symbol
\footnote{This may come as a surprise to some readers. We note that the symbol can be expressed as a divided difference:
$\sigma(v_{k,j})(\xi')) = [r_0^2,r_1^2,\dots,r_{m-1}^2](|\xi'|^2 +\cdot)^{\frac{j+k-1}{2}}$ (see \cite[Chapter 4 (7.7)]{DevLor}).
Because this functional annihilates polynomials of degree $m-1$, the behavior of $\sigma(v_{k,j})(\xi')$ as $|\xi'|\to \infty$ is $\mathcal{O}(|\xi'|^{j+k+1-2m})$, the same rate of decay as observed in the previous examples.}
$$
\sigma(v_{k,j})(\xi')
=
\frac12
\sum_{\ell =0}^{m-1} \frac{(|\xi'|^2 +r_{\ell}^2)^{\frac{j+k-1}{2}}}{\prod_{\nu \ne \ell} (r_{\nu}^2 - r_{\ell}^2)}.
$$

The symbol of $\mathcal{V}$ is checkerboard and Hankel. 
In this case, we make no attempt to factor the  symbol, although we permute the rows and columns to make use of the checkerboard Hankel structure:
$$
\widetilde {\sigma(\mathcal{V})(\xi') }= \begin{pmatrix} H_1(|\xi'|) & 0 \\ 0 & H_2(|\xi'|) \end{pmatrix}, 
\quad\text{with}\quad
(H_1(t))_{\nu,\mu} = a_{\nu+\mu}(t)
\qand
(H_2(t))_{\nu,\mu} = a_{\nu+\mu+1}(t),
$$
where
$$
a_n(t) = 
\frac12
\sum_{\ell =0}^{m-1} \frac{(t^2 +r_{\ell}^2)^{\frac{j+k-1}{2}}}{\prod_{\nu \ne \ell} (r_{\nu}^2 - r_{\ell}^2)}\,.
$$
For each $ t\ge0$, $H_1(t)$ and $H_2(t)$ are totally positive (as easily determined using \cite[Theorem 4.4]{Pinkus}), and so there is a planar network, with weights depending on $t$, governing the combinatorics of this situation as well.


\bibliographystyle{abbrv}
\bibliography{bibl}
\end{document}